%% file: main_involutions.tex
\documentclass[11pt]{amsart}

\input{macro.tex}
\title{Generating the plane Cremona groups by involutions}
\author{Stéphane Lamy}
\author{Julia Schneider}
\address{Institut de Mathématiques de Toulouse UMR 5219, Université de Toulouse, UPS F-31062 Toulouse Cedex 9, France}
\email{slamy@math.univ-toulouse.fr}
\email{julia.schneider@math.univ-toulouse.fr}
\date{\today}

\thanks{
J.S. is supported by the Swiss National Science Foundation project P2BSP2\_200209 and hosted by the Institut de Math\'ematiques de Toulouse.}

\begin{document}

\begin{abstract}
 We prove that over any perfect field the plane Cremona group is generated by involutions.
\end{abstract}

\maketitle

\section{Introduction}

The plane Cremona group over a field $\k$ is the group $\Bir_\k(\P^2)$ of birational transformations of the projective plane.
In concrete terms a map $g \in \Bir_\k(\P^2)$ can be written in homogeneous coordinates as
\[
g \colon [x:y:z] \rat [P_0:P_1:P_2]
\]
where the $P_i \in \k[x,y,z]$ are homogeneous polynomials of the same degree and without non-constant common factor, and such that $g$ admits an inverse of the same form.
In a more geometrical way, over an algebraically closed field any birational map between surfaces can be understood as a sequence of blow-ups and inverses of blow-ups.
Similarly, in the context of surfaces defined over a perfect field, the elementary operations to factorize birational maps are blow-ups of Galois orbits.
It is remarkable that a single class of elementary transformations allows to reconstruct any birational  map.
However a blow-up is a transformation between two non-isomorphic surfaces, so it leaves open the question of finding a natural generating set for the group  $\Bir_\k(\P^2)$.
Over an algebraically closed field, Noether's Theorem gives a neat answer: the Cremona group is generated by $\Aut_\k(\P^2) = \PGL_3(\k)$ and by a single extra generator, the standard quadratic involution $\sigma\colon [x:y:z] \rat [yz:xz:xy]$.
Since in this case $\PGL_3(\k)$ is generated by involutions, we obtain in particular that the Cremona group is generated by involutions.

The main result of the present paper is a generalization of this last statement to the case of an arbitrary perfect field $\k$:

\begin{theorem}
\label{t:main}
Let $\k$ be a perfect field.
The Cremona group $\Bir_\k(\P^2)$ is generated by involutions.
\end{theorem}

One motivation for this result was to understand the abelianization of the Cremona group, or in other words the possible surjective homomorphisms from the Cremona group to an abelian group.
Over many perfect fields including all number fields, finite fields, or the real numbers, we know \cite{Zimmermann, LamyZimmermann, SchneiderRelations} that the abelianization of the Cremona group contains an infinite direct sums of groups of order 2.
An immediate corollary of Theorem \ref{t:main} is that there is no surjective homomorphism from $\Bir_\k(\P^2)$ to $\Z$, or to $\Z/n\Z$ for $n \ge 3$. We can rephrase this remark as follows:

\begin{corollary}
\label{cor:main}
Let $\k$ be a perfect field.
The abelianization of the Cremona group $\Bir_\k(\P^2)$ is a group of exponent $2$ \parent{any non-trivial element has order $2$}.
\end{corollary}

In contrast, in higher dimension it was recently proved by E. Shinder and H.-Y. Lin \cite{LinShinder} that $\Bir_\Q(\P^3)$ and $\Bir_\C(\P^4)$ are \emph{not} generated by involutions, as a consequence of the fact that these groups admit homomorphisms to $\Z$.
In another direction, C. Shramov has proved that some Severi--Brauer surfaces (forms of $\P^2$ without any rational point) admit an infinite group of birational maps without any involution \cite[Theorem 1.2]{Shramov}.
So it was not clear a priori whether the groups $\Bir_\k(\P^2)$ should always be generated by involutions.
In retrospect the result seems to be an accident of low dimension, and this might justify that the proof cannot be completely conceptual: even if we use the Sarkisov program as a general framework, at some points some combinatorial miracles have to occur in order for the result to hold true.

We mention that over the field of real numbers, the generation of $\Bir_\R(\P^2)$ by involutions was proved by S. Zimmermann in her PhD thesis \cite[Corollary II.4.12]{ZimmermannPhD}, together with a complete description of the abelianization of $\Bir_\R(\P^2)$ \cite{Zimmermann}.
Over the field with 2 elements, the generation of $\Bir_{\FF_2}(\P^2)$ by involutions was established by the second author \cite{SchneiderF2}, with a similar strategy as in the present paper, but with some cases relying on an exhaustive search assisted by computer.

We now explain our strategy of proof for Theorem \ref{t:main}.
To obtain a set of generators for $\Bir_\k(\P^2)$, our first step is to apply the Sarkisov program.
As a byproduct of the Sarkisov factorization, we get a natural invariant associated to each element $f \in \Bir_\k(\P^2)$, namely the minimal number $\sl(f) \ge 0$ of Sarkisov links necessary to factorize $f$.
Then we can define a notion of irreducible element with respect to this Sarkisov length, and we obtain in Proposition \ref{p:generated_by_irreducible} an abstract factorization result in terms of irreducible elements.

Then we look more closely at the geometry of the involved links.
There is a known list of possible Sarkisov links between rational surfaces, due to Iskovskikh \cite{Iskovskikh93}, see also \cite{Corti} and \cite{Iskovskikh96}.
Here we face the same problem as with blow-ups: a Sarkisov link is (in general) not a birational map between two isomorphic surfaces, so to deduce a set of generators for the Cremona group we have to explain how to concatenate Sarkisov links.
This was also done by Iskovskikh, but resulted in very long lists not easy to use (see for instance \cite[Theorem 2.6]{Iskovskikh96}, whose statement runs over 8 pages).
In Theorem \ref{t:russian} and Proposition \ref{p:irreducible} we propose a compact way to express these results, and also provide in Appendix~\ref{app:russian_bis} a mostly self-contained proof of the classification of Sarkisov links between rational surfaces.

We can distinguish two kinds of irreducible elements in $\Bir_\k(\P^2)$: those involving links between conic bundles, and so preserving a rational fibration, and other sporadic cases with no such fibration.
In Section \ref{sec:3fibrations} we describe the possible fibrations, which turn out to be one of the following types: pencil of lines through a rational point, or pencil of conics through a Galois orbit of size 4, or two Galois orbits of size 2.
Then in Proposition \ref{prop:CremonaGeneratedByDelPezzoAndFiberTypes} we get a more geometric set of generators.
We can see this as an analogue of another classical result, Castelnuovo's Theorem, which asserts that over an algebraically closed field the Cremona group is generated by linear automorphisms and Jonquières maps.

The last step towards the proof of Theorem \ref{t:main} is to factorize each of these generators into a product of involutions.
For some generators we identify some matrix groups to which they belong, such as a projective linear group or an orthogonal group.
Then we use classical results about the generation of these groups by involutions, such as the Cartan-Dieudonné's Theorem.
For the remaining sporadic generators, in Section~\ref{sec:delPezzo} we manage to write them as products of quadratic, Geiser or Bertini involutions. To find these involutions we use elementary relations between Sarkisov links, that we can visualize as polygonal pieces encoded by Del Pezzo surfaces of Picard rank 3.
In Appendix \ref{app:relations} we give an exhaustive lists of all such pieces, even if in the present paper we only use a few of them.

\section{Sarkisov links}
\label{sec:link}

In this section we recall the notion of Sarkisov link, and state the classification of Sarkisov links for rational surfaces over a perfect field.
We also introduce the notion of irreducible element in the Cremona group, and in Proposition \ref{p:irreducible} we get a first description of such irreducible elements.

\subsection{Generation by irreducible elements}

Let $\k$ be a perfect field, and $\alg \k$ an algebraic closure of $\k$.
Let $X$ be an algebraic variety defined over $\k$, and denote by $X(\k)$ the set of $\k$-rational points on $X$.
The absolute Galois group $\Gal(\alg \k / \k)$ acts on $X \times_{\Spec \k} \Spec \alg \k$ through the second factor, and in particular it acts on $X(\alg \k)$.
We call \emph{$d$-point} on $X$ an orbit of size $d$ in $X(\alg \k)$ under the action of $\Gal(\alg \k / \k)$.
When $d = 1$ we keep the terminology \emph{rational point} instead of $1$-point.
If $p = \{p_1, \dots, p_d \}$ is a $d$-point on a surface $X$, we say that each $p_i \in X(\alg \k)$ is a \emph{geometric component} of $p$ (or simply a component).
Given $g \in \Gal(\alg \k/\k)$, we denote $p_i^g$ the image of the component $p_i$ under the action of $g$.

Following \cite{LamyZimmermann}, we now recall the definition of rank $r$ fibration, which unifies the concepts of Del Pezzo surfaces and conic bundles, of Sarkisov links between such surfaces, and of elementary relations between such links.

By a surface we always mean a smooth projective surface defined over a perfect field $\k$.
We say that a surface $X$ is rational if it is birational to $\P^2$ over $\k$.
Let $X$ be a rational surface, and $r \ge 1$ an integer.
We say that $X$ is a \emph{rank~$r$ fibration} if there exists a surjective morphism $X \to B$ with connected fibers, anticanonical divisor relatively ample and relative Picard rank equal to $r$, where $B$ is a point or a smooth curve.
Since we assume $X$ rational, there are only two possibilities for $B$: either $B = \pt$ is a point and $X$ is a Del Pezzo surface of Picard rank $r$ over $\k$; or $B = \P^1$ and $X \to \P^1$ is a conic bundle with $r-1$ orbits of singular fibers.
In particular, a rank $1$ fibration is the same as a rational surface with a structure of Mori fiber space.
Since in this paper we are only interested in rational surfaces, we put the rationality condition in the definition in order to avoid repeating `rational' rank $r$ fibration everywhere.

A \emph{marked} rank $r$ fibration is a rank $r$ fibration $X/B$ together with a birational map $\phi\colon X \rat \P^2$.
Let $(X/B, \phi)$ and $(X'/B', \phi')$ be two marked fibrations, of respective rank $r$ and $r'$, and consider the birational map $\phi'^{-1} \circ \phi\colon X \rat X'$ induced by the markings.
We say that $X/B$ \emph{factorizes through} $X'/B'$, or that $X'/B'$ is \emph{dominated} by $X/B$, if this induced map is a morphism, and if there exists a morphism $B' \to B$ such that the following diagram commutes:
\[
\begin{tikzcd}[row sep=small]
X \ar[rrr] \ar[dr,"\phi'^{-1} \circ \phi",swap] &&& B \\
& X' \ar[r] & B' \ar[ur]
\end{tikzcd}
\]
This implies $r \ge r'$.
If $B' \to B$ and $\phi'^{-1} \circ \phi\colon X \to X'$ are both isomorphism, we say that the two fibrations are \emph{equivalent}.
The Cremona group acts on equivalence classes of marked fibration, via post-composition:
\[
f \cdot  (X/B, \phi) = (X/B, f \circ \phi).
\]
In the sequel all rank $r$ fibrations are supposed to be marked, but we usually keep the marking implicit.

We say that a $d$-point on a Del Pezzo surface (resp. on a conic bundle) is \emph{general} if the blow-up of the orbit is still a del Pezzo surface (resp. a conic bundle over the same base curve).

\begin{lemma}
\label{l:2_domination_type}
Assume that $X/B$ is a rank $r+1$ fibration that factorizes through a rank~$r$ fibration $X'/B'$.
Then one of the following holds:
\begin{enumerate}
\item \label{domination:blowup} Either $B \simeq B'$, and there exists a general $d$-point on $X'$ such that $X \to X'$ is the blow-up of $p$;
\item \label{domination:base} Or $B = \P^1$, $B' = \pt$ and $X \to X'$ is an isomorphism.
\end{enumerate}
\end{lemma}

\begin{proof}
By additivity of the relative Picard rank, one of the morphisms $X \to X'$ of $B' \to B$ is an isomorphism, and the other one has relative Picard rank 1.
This gives the two cases of the statement.
\end{proof}

The \emph{piece} \label{def:piece} of a rank $r$ fibration $X/B$ is the $(r-1)$-dimensional combinatorial polytope constructed as follows:
 Each rank $d$ fibration dominated by $X/B$ is a $(d-1)$-dimensional face, and for each pair of faces $X_i/B_i$, $i=1,2$, $X_2/B_2$ lies in $X_1/B_1$ if and only if $X_1/B_1$ dominates $X_2/B_2$.
We write $(r-1)$-piece when we want to emphasize the dimension of the piece (associated to a rank $r$ fibration).
We now consider more closely the case of $1$-pieces, which turn out to encode Sarkisov links.

Let $Y/B_Y$ be a rank 2 fibration.
As a consequence of the two-rays game, there are exactly two rank~$1$ fibrations $X/B$ and $X'/B'$ dominated by $Y/B_Y$.
We say that the induced birational map $X \rat X'$ is a \emph{Sarkisov link}.
Using Lemma \ref{l:2_domination_type} we can distinguish between four types of Sarkisov links, depending if the domination of $X/ B$ (resp. $X'/B'$) by $Y/B_Y$ is a blow-up or a change of base.
In Table \ref{tab:links} we describe these four types.
The first column is the usual numbering in the literature, where for type~II we moreover distinguish between the case where the common base of the fibrations is $\P^1$ or a point.
Observe here that by definition a link of type II over $\P^1$ sends a general fiber of $X/\P^1$ to a fiber of $X'/\P^1$.
The second column shows the classical diagram, which is often taken as a definition for the Sarkisov links. Here an arrow $\stackrel{_d}{\to}$ refers to the blow-up of a general $d$-point.
The third column shows the $1$-piece with the rank~$2$ fibration in the center of the edge, dominating the two rank 1 fibrations $X/ B$ and $X'/B'$ on the left and right vertices.
Finally the last column shows the shorthand that we shall use in the text.

\begin{table}[t]
\begin{tabular}{cCCC}
\toprule
Type & \text{Diagram} & \text{Piece} & \text{Short notation} \\
\midrule
I &
\begin{tikzcd}[row sep=small]
& X'\ar[dl,"d",swap]\ar[d] \\ X\ar[d] & \P^1\ar[dl] \\ \pt
\end{tikzcd} &
\begin{tikzcd}
X/\pt & X'/\pt\ar[l,-,ultra thick,"d",swap]\ar[r,-,ultra thick] & X'/\P^1
\end{tikzcd} &
X \linkI{d} X'
\\ \\
II/$\pt$ &
\begin{tikzcd}[row sep=small]
& Y\ar[dl,"d",swap]\ar[dr,"d'"] \\ X\ar[dr] && X' \ar[dl] \\ & \pt
\end{tikzcd} &
\begin{tikzcd}
X/\pt & Y/\pt\ar[l,-,ultra thick,"d",swap]\ar[r,-,ultra thick,"d'"] & X'/\pt
\end{tikzcd} &
X \link{d}{d'} X'
\\ \\
II/$\P^1$ &
\begin{tikzcd}[row sep=small]
& Y\ar[dl,"d",swap]\ar[dr,"d"] \\ X\ar[dr] && X' \ar[dl] \\ & \P^1
\end{tikzcd} &
\begin{tikzcd}
X/\P^1 & Y/\P^1\ar[l,-,ultra thick,"d",swap]\ar[r,-,ultra thick,"d"] & X'/\P^1
\end{tikzcd} &
X \link{d}{d} X'
\\ \\
III &
\begin{tikzcd}[row sep=small]
X\ar[dr,"d"]\ar[d] \\  \P^1\ar[dr] & X'\ar[d] \\ & \pt
\end{tikzcd} &
\begin{tikzcd}
X/\P^1 & X/\pt\ar[l,-,ultra thick]\ar[r,-,ultra thick,"d"] & X'/\pt
\end{tikzcd} &
X \linkIII{d} X'
\\ \\
IV &
\begin{tikzcd}[row sep=small]
& X\ar[dl]\ar[dr] \\ \P^1\ar[dr] && \P^1 \ar[dl] \\ & \pt
\end{tikzcd} &
\begin{tikzcd}
X/\P^1 & X/\pt\ar[l,-,ultra thick]\ar[r,-,ultra thick] & X/\P^1
\end{tikzcd} &
X \linkIV X\\
\end{tabular}
\caption{Sarkisov links}
\label{tab:links}
\end{table}

In \cite[Proposition 2.6]{LamyZimmermann} it was shown that a $2$-piece (or rather its geometric realization) is homeomorphic to a disk, and we will draw it as a regular polygon.
The boundary of the polygon corresponds to a sequences of Sarkisov links whose product is an automorphism: we say that the piece encodes an \emph{elementary relation} between Sarkisov links.
 See Appendix~\ref{app:relations} for a list of all $2$-dimensional pieces given by a rank $3$ fibration $X/\pt$ where $X$ is a Del Pezzo surface.
It is in fact true that for any $d$, a $d$-piece is homeomorphic to a convex polytope of dimension $d$, with vertices correspond to Mori fiber spaces, edges to Sarkisov links and $2$-dimensional faces to elementary relations.
We do not give details since in this paper we shall only use $d$-pieces with $d = 1$ or $2$.

We define $\BirMori_\k(\P^2)$ as the groupoid of birational maps between rank 1 fibrations (or equivalently, between rational Mori fiber spaces, hence the name).
The Sarkisov program can be phrased as follows (see \cite[Proposition 3.14]{LamyZimmermann}):

\begin{proposition}
\label{p:Sarkisov}
The groupoid $\BirMori_\k(\P^2)$ is generated by Sarkisov links and automorphisms.
\end{proposition}

Given $g \in \BirMori_\k(\P^2)$, we call \emph{Sarkisov length} of $g$, denoted $\sl(g)$, the minimal number of Sarkisov links necessary to factorize $g$.
Now assume $f \in \Bir_\k(\P^2)$.
We say that $f$ is \emph{reducible} if we can write $f = f_2 \circ f_1$ with $f_i \in \Bir_\k(\P^2)$, $\sl(f_i) < \sl(f)$ for $i=1,2$.
Otherwise we say that $f$ is \emph{irreducible}.
In particular $\sl(f) = 0$ if and only if $f \in \Aut_\k(\P^2)$, and such elements are trivially irreducible.
As an immediate consequence of Proposition \ref{p:Sarkisov} and of the definition of irreducible elements we have:

\begin{proposition}
\label{p:generated_by_irreducible}
The Cremona group $\Bir_\k(\P^2)$ is generated by its irreducible elements.
\end{proposition}

\subsection{The graph of Sarkisov links}

Let $\Dl$ be the set of isomorphy classes of $\k$-rational Del Pezzo surfaces of Picard rank~1, and $\Cl$ the set of isomorphy classes of $\k$-rational conic bundles of relative Picard rank~1.
We now define several subsets $\Dl_i \subset \Dl$ and $\Cl_j \subset \Cl$.
The index always refers to the degree of the corresponding surfaces, by which we mean the self-intersection of the canonical divisor.
Starting from $\P^2 \in \Dl$, each definition is in terms of explicit Sarkisov links, from an already defined class of surfaces.

\begin{itemize}
\item
Let $\Dl_8 \subset \Dl$ be the set of surfaces obtained by blowing-up a $2$-point on $\P^2$, and then blowing-down the transform of the line through this point.
Observe that any $X \in \Dl_8$ is isomorphic to $\P^1 \times \P^1$ over $\alg \k$, and so we can speak of the bidegree of a divisor on~$X$.
We call \emph{diagonal} any curve on $X$ of bidegree $(1,1)$.
Similarly, we call vertical ruling (resp. horizontal ruling ) any curve of bidegree $(1,0)$ (resp. $(0,1)$), necessarily defined over $\alg \k$ but not over $\k$.
\item
Let $\Dl_5  \subset \Dl$ be the set of surfaces obtained by blowing-up a general $5$-point on $\P^2$, and then blowing-down the transform of the smooth conic through this point.
\item
Let $\Dl_6  \subset \Dl$ be the set of surfaces obtained by blowing-up a general $3$-point on a surface $X \in \Dl_8$, and then blowing-down the transform of the smooth diagonal through this point.
\item
Let $\Cl_5  \subset \Cl$ be the set of conic bundles obtained by blowing-up a general $4$-point on $\P^2$, and taking the transform of conics through this point.
\item
Let $\Cl_6  \subset \Cl$ be the set of conic bundles obtained by blowing-up a general $2$-point on $X \in \Dl_8$, and taking the transform of diagonals through this point.
Observe that thinking of $X$ as coming from the blow-up of a $2$-point on $\P^2$ followed by the contraction of a line, the conic bundle on a surface in the set $\Cl_6$ corresponds to the transform of conics in $\P^2$ passing through two $2$-points.
\item
Let $\Cl_8 = \{\F_n \mid n \ge 0\} \subset \Cl$ be the set of Hirzebruch surfaces, with their structure of conic bundle (precisely, of $\P^1$-bundle over $\P^1$, with a section of self-intersection $-n$).
\end{itemize}

We say that a surface $X \in \Dl_8$ is of \emph{type} $\Dl_8$, and similarly with the other subsets.
In Appendix~\ref{app:russian_bis} we give a proof of the following theorem, which can also be extracted from \cite{Iskovskikh96}.
The content of the theorem is summed-up in Figure \ref{diag:links}, adapted from \cite{SchneiderRelations}.
The label on each edge indicates the type of Sarkisov link (according to the numbering in Table \ref{tab:links}), and the size of the blown-up Galois orbits.

\begin{figure}[t]
\[
\includegraphics[scale=1]{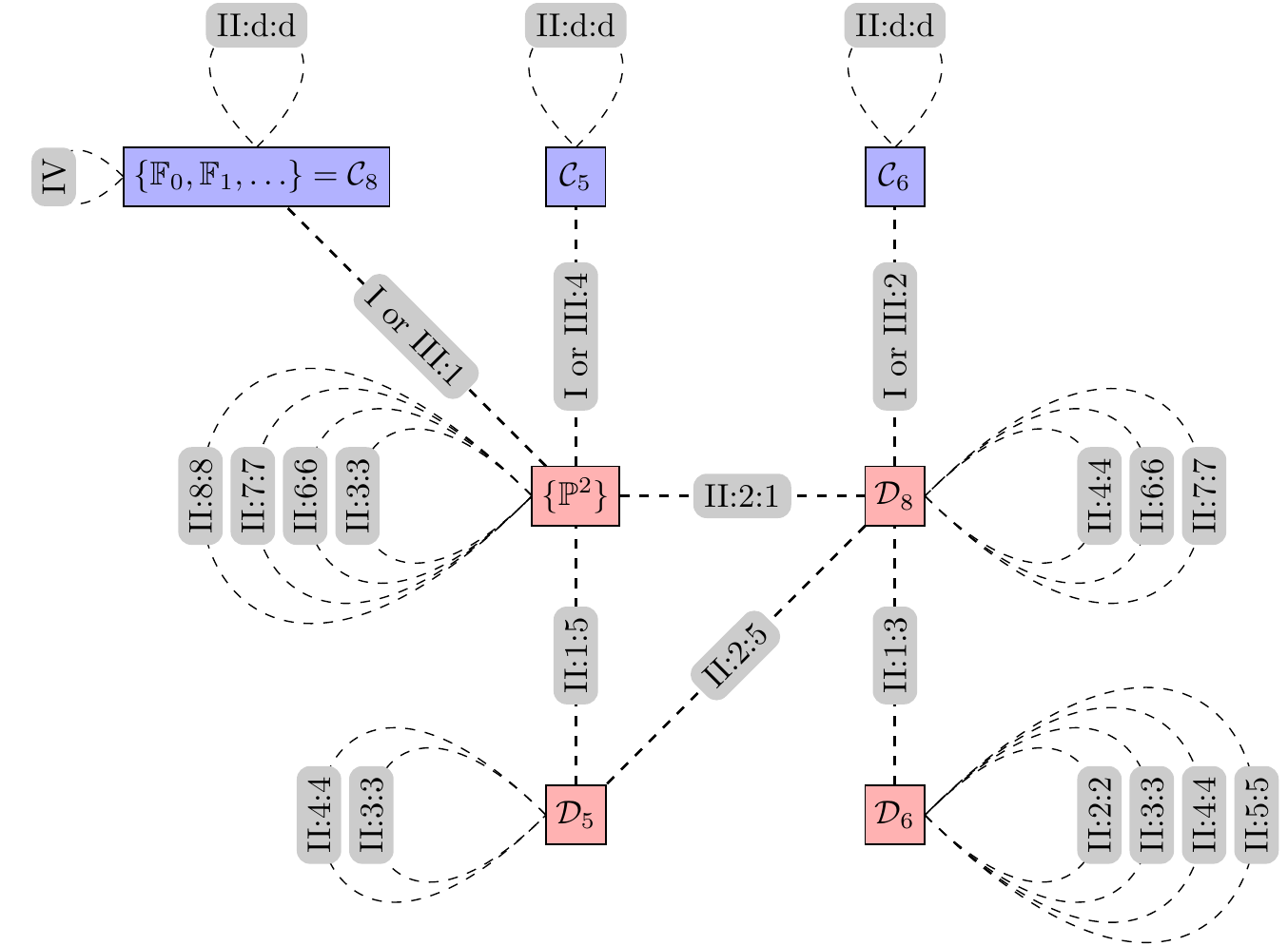}
\]
\caption{Sarkisov links between rational surfaces over a perfect field}
\label{diag:links}
\end{figure}

\begin{theorem}
\label{t:russian}
Let $\k$ be a perfect field.
\begin{enumerate}
\item \label{russian:1}
Given any \parent{marked} rank $1$ fibration $X/B$, the surface $X$ lies in one of the following seven pairwise disjoint sets, which form the vertices of the graph in Figure \textup{\ref{diag:links}}:
\[
\{\P^2\}, \Dl_5, \Dl_6, \Dl_8, \Cl_5, \Cl_6, \Cl_8.
\]
In particular $\Dl = \Dl_5 \cup \Dl_6 \cup \Dl_8 \cup \{\P^2\}$ and $\Cl = \Cl_5 \cup \Cl_6 \cup \Cl_8$.
\item \label{russian:2}
Let $X \rat X'$ be a Sarkisov link of type I or II between rank $1$ fibrations.
Then the type of $X'$ is uniquely defined by the type of $X$ and the size of the blown-up Galois orbit, as indicated by the edges on the graph in Figure \textup{\ref{diag:links}}.
\item \label{russian:3}
Let $X \rat X'$ be a Sarkisov link of type IV between rank $1$ fibrations.
Then $X = Y = \F_0 = \P^1 \times \P^1$ and the link is the change of ruling.
\end{enumerate}
\end{theorem}

From Theorem \ref{t:russian}, any factorization into Sarkisov links of a birational map between rank $1$ fibrations corresponds to a path in the graph of Figure \textup{\ref{diag:links}}.
In particular, any factorization of an element of $\Bir_\k(\P^2)$ corresponds to a closed path based at the vertex $\{\P^2\}$.
A natural question is to ask about the converse: is any path in the graph realized by at least one composition of Sarkisov links?
We shall use the following partial positive answer to this question, which follows from the definition of the sets $\Dl_i$ and $\Cl_i$:
\begin{itemize}
\item
Given $X \in \Dl_8 \cup \Dl_5 \cup \Cl_5$, there exists a Sarkisov link from $X$ to $\P^2$;
\item Given $X \in \Dl_6 \cup \Cl_6$, there exists a composition of two Sarkisov links from $X$ to $\P^2$, via an intermediate surface $X' \in \Dl_8$;
\item Given $X \in \Cl_8$, $X$ admits a link to $\P^2$ if and only if $X = \F_1$.
\end{itemize}

Let $f \in \Bir_\k(\P^2)$ be an irreducible element.
We say that $f$ is of \emph{Del Pezzo type} if it admits a minimal factorization using only links of type II over $\pt$, and $f$ is of \emph{fibering type} if it admits a minimal factorization containing a link of type I (and so also a link of type III).
In other words, in terms of the closed path associated with a minimal factorization: $f$ is of Del Pezzo type if the path visits at most the vertices $\{\P^2\}, \Dl_8, \Dl_6, \Dl_5$, and is of fibering type if the path visits at least one of the vertices $\Cl_8, \Cl_6$ or $\Cl_5$.
Observe that these two classes of irreducible elements have no reason to be disjoint, since a minimal factorization into Sarkisov links is not unique.
(From Figure \ref{P2_24} the interested reader can cook-up an example of a map $f \in \Bir_\k(\P^2)$ with $\sl(f) = 3$ and admitting two minimal factorizations of distinct types.)

\begin{table}[t]
\begin{center}
\begin{tabular}{CCl}
\toprule
\sl(f) & \text{Factorization type} & Note \\
\midrule
0 & - & Automorphism \\
\midrule
\multirow{4}{*}{1}
& \P^2 \link{3}{3} \P^2 & Quadratic simplification \\
& \P^2 \link{6}{6} \P^2 & See Lemma \ref{lem:linkP2-66} \\
& \P^2 \link{7}{7} \P^2 & Geiser simplification \\
& \P^2 \link{8}{8} \P^2 & Bertini simplification \\
\midrule
\multirow{2}{*}{2}
& \P^2 \link{2}{1} \Dl_8 \link{1}{2} \P^2 & Quadratic simplification \\
& \P^2 \link{5}{1} \Dl_5 \link{1}{5} \P^2 & See Lemma \ref{lem:linkD5} \\
\midrule
\multirow{7}{*}{3}
& \P^2 \link{2}{1} \Dl_8 \link{4}{4} \Dl_8 \link{1}{2} \P^2 &  See Lemma \ref{lem:linkD8-44} \\
& \P^2 \link{2}{1} \Dl_8 \link{6}{6} \Dl_8 \link{1}{2} \P^2 & Geiser simplification \\
& \P^2 \link{2}{1} \Dl_8 \link{7}{7} \Dl_8 \link{1}{2} \P^2 & Bertini simplification \\
& \P^2 \link{5}{1} \Dl_5 \link{3}{3} \Dl_5 \link{1}{5} \P^2 & Geiser simplification \\
& \P^2 \link{5}{1} \Dl_5 \link{4}{4} \Dl_5 \link{1}{5} \P^2 & Bertini simplification \\
& \P^2 \link{5}{1} \Dl_5 \link{2}{5} \Dl_8 \link{1}{2} \P^2 & See Lemma \ref{lem:loop}\\
& \P^2 \link{2}{1} \Dl_8 \link{5}{2} \Dl_5 \link{1}{5} \P^2 & inverse of the previous one \\
\midrule
4
& \P^2 \link{2}{1} \Dl_8 \link{3}{1} \Dl_6 \link{1}{3} \Dl_8 \link{1}{2} \P^2
& See Lemma \ref{lem:linkD6} \\
\midrule
\multirow{4}{*}{5}
& \P^2 \link{2}{1} \Dl_8 \link{3}{1} \Dl_6 \link{2}{2} \Dl_6 \link{1}{3} \Dl_8 \link{1}{2} \P^2
& See Lemma \ref{lem:linkD6-22} \\
& \P^2 \link{2}{1} \Dl_8 \link{3}{1} \Dl_6 \link{3}{3} \Dl_6 \link{1}{3} \Dl_8 \link{1}{2} \P^2
& See Lemma \ref{lem:linkD6-33} \\
& \P^2 \link{2}{1} \Dl_8 \link{3}{1} \Dl_6 \link{4}{4} \Dl_6 \link{1}{3} \Dl_8 \link{1}{2} \P^2
& Geiser simplification \\
& \P^2 \link{2}{1} \Dl_8 \link{3}{1} \Dl_6 \link{5}{5} \Dl_6 \link{1}{3} \Dl_8 \link{1}{2} \P^2
& Bertini simplification \\
\bottomrule\\
\end{tabular}
\end{center}
\caption{Irreducible generators of Del Pezzo type in Proposition~\ref{p:irreducible}\ref{irreducible_delpezzo} (the last column corresponds to the study in \S \ref{sec:delPezzo})}
\label{tab:del Pezzo type}
\end{table}

\begin{proposition}
\label{p:irreducible}
Let $f \in \Bir_\k(\P^2)$ be an irreducible element.
\begin{enumerate}
\item \label{no_type_IV}
A minimal factorization of $f$ never contains a link of type IV.
\item \label{no_type_I_III}
A minimal factorization of $f$ never contains a link of type III immediately following a link of type I.
\item \label{irreducible_fibering}
If $f$ is of fibering type, then $f$ admits a minimal factorization of one of the following forms:
\begin{enumerate}
\item\label{C8} $\P^2 \linkI{1} \F_1 \in \Cl_8$, followed by a certain number $r \ge 1$ of links of type II between Hirzebruch surfaces, and a final link $\F_1 \linkIII{1} \P^2$, so in particular $\sl(f) = 2 + r \ge 3$;
\item $\P^2 \linkI{4} \Cl_5 \link{d}{d} \Cl_5 \linkIII{4} \P^2$, so $\sl(f) = 3$;
\item $\P^2 \link21 \Dl_8 \linkI{2} \Cl_6 \link{d}{d} \Cl_6 \linkIII{2} \Dl_8 \link12 \P^2$, so $\sl(f) = 5$.
\end{enumerate}
\item \label{irreducible_delpezzo}
If $f$ is of Del Pezzo type, then $\sl(f) \le 5$, and $f$ admits a minimal factorization of one the forms listed in Table \textup{\ref{tab:del Pezzo type}}.
\end{enumerate}
\end{proposition}

\begin{proof}
\ref{no_type_IV}.
By Theorem \ref{t:russian}\ref{russian:3}, we know that a link of type IV is the change of ruling on $\F_0$.
Assume that $f$ admits a minimal factorization $f = \alpha^{-1} \circ \chi \circ \beta$, where $\chi\colon \F_0 \to \F_0$ is a link of type IV, and $\alpha, \beta\colon \P^2 \to \F_0$ are compositions of Sarkisov links.
Observe that $\P^2$ and $\F_0$ cannot be connected by a single link, so $\sl(\alpha) \ge 2$, $\sl(\beta) \ge 2$.
Let $\gamma^{-1}\colon \F_0 \rat \P^2$ be the composition of two Sarkisov links, corresponding to the blow-up of a rational point $x \in \F_0$ followed by the contraction of the two rules through $x$.
Then we can write
\[
f = (\alpha^{-1} \circ \chi \circ \gamma) \circ (\gamma^{-1} \circ \beta).
\]
We have $\sl(\alpha^{-1} \circ \chi \circ \gamma) \le \sl(f)$ and $\sl(\gamma^{-1} \circ \beta) < \sl(f)$, so by irreducibility of $f$ the first inequality must be an equality and this gives $\sl(\beta) = \sl(\gamma) = 2$. Similarly $\sl(\alpha) = 2$, and both $\alpha$ and $\beta$ are composition of two links $\P^2 \link{}{} \F_1 \link{}{} \F_0$.
But then $\alpha^{-1} \chi \alpha \in \Aut_\k(\P^2)$ and writing
\[
f = (\alpha^{-1} \circ \chi \circ \alpha) \circ (\alpha^{-1} \circ \beta)
\]
we contradict our assumption that $f$ is irreducible.

\ref{no_type_I_III}.
If $X \in \Cl_i$ with $i \in \{5,6,8\}$, the Picard rank of $X$ is 2, and the two extremal rays correspond respectively to the fibration to $\P^1$ and to a link of type III.
In particular if $X_1 \linkI{d} X \linkIII{d} X_2$ is a composition of a link of type I followed by a link of type III, then $X_1 \to X_2$ is an isomorphism and this contradicts that such a composition of two such links can be part of a minimal factorization.

\ref{irreducible_fibering} and  \ref{irreducible_delpezzo}.
Let $m = \sl(f)$, and consider the successive vertices
\[
v_0 = \{\P^2\}, v_1, \dots, v_{m-1}, v_m = \{\P^2\}
\]
visited by the closed path associated with a minimal factorization of $f$.
So each pair of vertices $v_i, v_{i+1}$ are joined by an edge (which can be a loop) in the graph of Figure~\ref{diag:links}.
Since $f$ is irreducible, any intermediate vertex $v_i$, $1 \le i \le m-1$, is distinct from  $\{\P^2\}$.
If one of the vertices is in $\Cl_8$ then since we know by \ref{no_type_IV} that we cannot use any link of type IV, by inspection of the graph the factorization is as described in~\ref{C8}.
Since the factorization of $f$ is minimal, if the path does not visit the vertex $\Cl_8$ then by the remark after Theorem \ref{t:russian} the associated path is unimodal in the following sense:
\begin{itemize}
\item For each index $i$ such that $1 \le i \le \frac{m}{2}$, $d(v_i, v_0) = d(v_{i-1}, v_0) + 1$, where $d(\cdot, \cdot)$ is the distance in the graph;
\item For each index $i$ such that $\frac{m}{2} + 1 \le i \le m$, $d(v_i, v_0) = d(v_{i-1}, v_0) - 1$;
\item If $m = 2n+1$ is odd then $d(v_{n}, v_0) = d(v_{n+1},v_0)$.
\end{itemize}

Then the possible factorization types in each case (fibering or Del Pezzo type) are obtained by inspection of Figure \ref{diag:links}, together with assertion  \ref{no_type_I_III}.
\end{proof}

\begin{remark}\label{rem:MinimalFactorizations}
We do not claim that any factorization of the forms given in Table \ref{tab:del Pezzo type} is automatically minimal, or even that there exists an irreducible element corresponding to a given line of the table (in particular, the existence can depend on the field $\k$ we are working with).
See Figure \ref{P2_23} for an example of a map with a factorization of type $\P^2 \link{2}{1} \Dl_8 \link{3}{1} \Dl_6 \link{1}{3} \Dl_8 \link{1}{2} \P^2$ which is not minimal because it also admits a factorization of type $\P^2 \link33 \P^2$.
\end{remark}

\input{main_fibering_type}

\input{main_problematic_cases}

\appendix

\input{appendix_russianbis}

\input{appendix_relations}

\bibliographystyle{myalpha}
\bibliography{biblio}

\end{document}

%% file: macro.tex
\usepackage[a4paper,text={14.5cm,23cm},centering]{geometry}
\usepackage[utf8]{inputenc}
\usepackage[T1]{fontenc}
\usepackage[svgnames,hyperref]{xcolor}
\usepackage{lmodern}
\usepackage{upgreek} 
\usepackage{amssymb} 
\usepackage{placeins} 

\usepackage{amsmath}

\usepackage{caption} 
\usepackage[shortlabels]{enumitem}
\setlist[1]{wide}
\setlist[2]{leftmargin=15mm}
\setlist[enumerate]{label=\rm{(\arabic*)}}
\setlist[enumerate,2]{label=\rm({\roman*}), }
\setlist[itemize]{label=\raisebox{0.25ex}{\tiny$\bullet$}}

\usepackage{booktabs, tabularx, multirow}
\newcolumntype{L}{>{$}l<{$}}
\newcolumntype{C}{>{$}c<{$}}

\usepackage{tikz}
\usetikzlibrary{cd,arrows,positioning,calc}
\usetikzlibrary{shapes.geometric}
\tikzset{map/.style={row sep=0em, column sep=0em}}
\usetikzlibrary{intersections, through}

\newcommand{\nicecolor}{Navy}
\usepackage[backref=page, pdfauthor={S. Lamy, J. Schneider}, colorlinks, citecolor=\nicecolor, linkcolor=\nicecolor, urlcolor=\nicecolor]{hyperref}
\usepackage[all]{hypcap} 

\theoremstyle{plain}
\newtheorem{theorem}{Theorem}[section]
\newtheorem{corollary}[theorem]{Corollary}
\newtheorem{proposition}[theorem]{Proposition}
\newtheorem{lemma}[theorem]{Lemma}

\theoremstyle{definition}

\newtheorem{example}[theorem]{Example}
\newtheorem{remark}[theorem]{Remark}

\newtheorem{setup}[theorem]{Set-Up}

\numberwithin{equation}{theorem}  

\newcommand{\p}{\mathbb{P}}
\let \P \relax
\newcommand{\P}{\mathbb{P}}
\newcommand{\A}{\mathbb{A}}
\newcommand{\F}{\mathbb{F}}
\newcommand{\Z}{\mathbf Z}
\newcommand{\R}{\mathbf R}
\newcommand{\Q}{\mathbf Q}
\newcommand{\C}{\mathbf C}
\newcommand{\FF}{\mathbf F}
\renewcommand{\k}{\mathbf{k}}
\newcommand{\alg}[1]{#1^a}
\newcommand{\bk}{{\alg \k}}
\newcommand{\id}{\textup{id}}
\newcommand{\pt}{\textup{pt}}

\newcommand{\Cl}{\mathcal C}
\newcommand{\Dl}{\mathcal D}

\newcommand{\Pl}{\mathcal P}
\newcommand{\Ql}{\mathcal Q}

\newcommand{\rat}{\dashrightarrow}
\newcommand{\parent}[1]{\textup{(}#1\textup{)}}
\newcommand{\smallmat}[1]{\left(\begin{smallmatrix}#1\end{smallmatrix}\right)}
\newcommand{\mat}[1]{\begin{pmatrix}#1\end{pmatrix}}

\newcommand{\link}[2]{\stackrel{_{#1}\;_{#2}}{\dashrightarrow}}
\newcommand{\linkI}[1]{\stackrel{_{#1}}{\dashrightarrow}}
\newcommand{\linkIII}[1]{\stackrel{_{#1}}{\to}}
\newcommand{\linkIV}{\stackrel{_\textup{\tiny IV}}{\to}}

\newcommand{\piece}[3]{\Pl(#1;#2,#3)}
\newcommand{\pieceF}[1]{\Pl(\F_0;#1)}

\newcommand{\two}{\uptheta} 

\renewcommand{\phi}{\varphi}
\newcommand{\eps}{\varepsilon}
\renewcommand{\le}{\leqslant}
\renewcommand{\leq}{\leqslant}
\renewcommand{\ge}{\geqslant}
\renewcommand{\geq}{\geqslant}

\DeclareMathOperator{\Bir}{Bir}
\DeclareMathOperator{\BirMori}{BirMori}
\DeclareMathOperator{\Aut}{Aut}

\DeclareMathOperator{\Spec}{Spec}
\DeclareMathOperator{\Pic}{Pic}
\let \O \relax
\DeclareMathOperator{\O}{O}
\DeclareMathOperator{\PGO}{PGO}
\DeclareMathOperator{\GO}{GO}
\DeclareMathOperator{\SO}{SO}
\DeclareMathOperator{\PGL}{PGL}
\DeclareMathOperator{\PSL}{PSL}
\DeclareMathOperator{\GL}{GL}
\DeclareMathOperator{\Gal}{Gal}
\let \sl \relax
\DeclareMathOperator{\sl}{sl}
\DeclareMathOperator{\Char}{char}

\renewcommand{\to}{ \, \tikz[baseline=-.6ex] \draw[->,line width=.5] (0,0) -- +(.5,0); \, }
\renewcommand{\mapsto}{ \, \tikz[baseline=-.6ex] \draw[|->,line width=.5] (0,0) -- +(.5,0); \, }

%% file: main_fibering_type.tex
\section{Subgroups of matrix type}
\label{sec:matrix}

Over an algebraically closed field it is a classical result that any element in $\Bir(\P^2)$ preserving a pencil of rational curves is conjugate to a Jonquières map, that is, a map preserving a pencil of lines.
Over a perfect field we have to introduce two more types of normal forms, that we call Jonquières type 2+2 and Jonquières type 4, and we call Jonquières maps of type 1 the classical ones.
In Proposition \ref{prop:CremonaGeneratedByDelPezzoAndFiberTypes} we  obtain a set of generators: first automorphisms and Jonquières maps of type 1, that are easily seen to be compositions of involutions; second Jonquières maps of type 2+2 or type 4, that are studied in the rest of this section; and finally irreducible elements of Del Pezzo type, which will be studied in Section \ref{sec:delPezzo}.
The proof that type 4 and type 2+2 Jonquières maps are compositions of involutions is the most technical part of this paper.
The argument naturally splits into two parts: Jonquières maps that preserve the fibration fiberwise correspond to an orthogonal group, and by explicit computations we obtain Jonquières involutions whose compositions can realize all possible permutations between fibers.

\subsection{Three types of fibrations}
\label{sec:3fibrations}

We call rational fibration on $\P^2$ any dominant rational map $\pi\colon\P^2\rat \P^1$.
We say that $f \in \Bir_\k(\P^2)$ \emph{preserves the fibration} if there exists $\alpha\in\Aut_\k(\P^1)$ such that $\pi\circ f=\alpha\circ\pi$, and we say it \emph{fixes the fibration} if $\alpha=\id_{\P^1}$:

\[
\begin{tikzcd}
\P^2 \ar[d,"\pi",dashed,swap] \ar[r,dashed,"f"]& \P^2 \ar[d,"\pi",dashed]\\
\P^1 \ar[r,"\alpha"]& \P^1
\end{tikzcd}
\]

We write $\Bir_\k(\P^2,\pi)$ for the group that preserves the fibration, $\Bir_\k(\P^2/\pi)$ for the one that fixes the fibration, and $\Aut_\k^{\pi}(\P^1)$ for the image in $\Aut_\k(\P^1)$ of the application $f\mapsto\alpha$.
So by definition we have an exact sequence
\begin{align*}
1\to\Bir_\k(\P^2/\pi)\to\Bir_\k(\P^2,\pi)\to \Aut_\k^{\pi}(\P^1)\to 1 \tag{Seq}\label{exact_sequence}.
\end{align*}

  We say that $f\in\Bir_\k(\P^2)$ is of
  \begin{itemize}
    \item \emph{Jonquières type $1$} if $f$ preserves the fibration given by the pencil of lines through a rational point;
    \item \emph{Jonquières type $2+2$}  with residue fields $L,L'$ if $f$ preserves the fibration given by the pencil of conics through two general $2$-points with respective residue fields $L$ and  $L'$;
    \item \emph{Jonquières type $4$} with residue field $F$ if $f$ preserves the fibration given by the pencil of conics through a general $4$-point with residue field $F$.
  \end{itemize}

Note that when we speak about ``the'' residue field we always mean up to $\k$-isomor\-phism.
For Jonquières type $1$ the sequence \eqref{exact_sequence} splits and gives a description as a semidirect product $\PGL_2(\k)\ltimes \PGL_2(\k(x))$.
For Jonquières type 2+2 and 4 it is not clear whether the sequence always splits.

\begin{remark}\label{rem:residuefielddeg4}
A field extension $F/\k$ is the residue field of a general $4$-point if and only if it is the residue field of an irreducible polynomial of degree $4$:

Choosing a conic $C$ defined over $\k$ going through the general $4$-point $p$, and applying a $\k$-coordinate change, we can assume that $C$ is given by $y^2-xz=0$. Hence, the components of $p$ are of the form $[a_i^2:a_i:1]\in\p^2(\bk)$ and its residue field is $\k(a_1)\simeq\k(a_2)\simeq\k(a_3)\simeq\k(a_4)$, which equals the residue field of the irreducible polynomial $f=(x-a_1)\cdots(x-a_4)\in\k$ of degree $4$.

For the converse direction, see Example~\ref{ex:SpecificFibrations} to see how one can construct a general $4$-point given an irreducible polynomial of degree $4$.
\end{remark}

\begin{lemma}\label{lem:FiberTypeAlways124}
Let $f\in\Bir_\k(\P^2)$ be an irreducible element of fibering type.
Then there exists $\alpha\in\PGL_3(\k)$ such that $\alpha\circ f$ is of Jonquières type $1$, $2+2$, or $4$.
\end{lemma}

\begin{proof}
We consider each of the three cases given by Proposition \ref{p:irreducible}\ref{irreducible_fibering}.

An irreducible element going through $\Cl_8$ is a map that sends the pencil of lines through one point onto the pencil of lines through possibly some other point.
Hence, there exists $\alpha\in\PGL_3(\k)$ such that $\alpha\circ f$ preserves the pencil of lines through a point and so $f$ is of Jonquières type 1.

If $f$ has factorization type $\P^2 \linkI{4} X \link{d}{d} X' \linkIII{4} \P^2$ with $X,X'\in \Cl_5$, then $f$ maps the pencil of conics through the $4$-point associated to $X$ onto the one through the $4$-point associated to $X'$.
By Lemma \ref{l:Julia_4pts} the surfaces $X$ and $X'$ are isomorphic and the associated $4$-points are equivalent under the action of $\PGL_3(\k)$. This gives the existence of $\alpha$.

Finally if $f$ has factorization type
$\P^2 \link21 Y \linkI{2} X \link{d}{d} X' \linkIII{2} Y' \link12 \P^2$ with $X,X'\in \Cl_6$ and $Y,Y'\in\Dl_8$ then $f$ send the pencil of conics through the two $2$-points associated to $X$ and $Y$ onto the one associated to $X'$ and $Y'$.
By Lemma \ref{l:XisoX'} we know that the two surfaces $X, X' \in \Cl_6$ of the middle link are dominated by isomorphic surfaces $Z, Z'$, each of which is the blow-up of $\P^2$ at two $2$-points, and the two sets of four geometric points are equivalent under the action of $\PGL_3(\k)$.
\end{proof}

Let $f,g$ be two elements of Jonquières type $1$, $2+2$, or $4$, preserving fibrations $\pi_f,\pi_g$.
We say that $f$ and $g$ are \emph{equivalent} if there exists $\alpha\in\PGL_3(\k)$ such that $\alpha^{-1}\circ g\circ \alpha$ preserves $\pi_f$.

In order to classify Jonquières maps up to equivalence, we first need a  reinforcement of \cite[Lemma 6.11]{SchneiderRelations},
which states that any Galois equivariant bijection between two Galois invariant sets of 4 points can be realized by an element of $\Aut_\k(\P^2)$:

\begin{lemma}
\label{l:residue}
Let $\k$ be a perfect field, and let $p,q \in \P^2$ be two general $4$-points.
Then the following are equivalent:
\begin{enumerate}
    \item\label{it:aut} There exists $\alpha\in\PGL_3(\k)$ such that $\alpha(p_i)=q_i$ for $i = 1, \dots, 4$.
    \item\label{it:property} For every $g\in\Gal(\bk/\k)$ there exists $\sigma\in S_4$ such that $g(p_i)=p_{\sigma(i)}$ and $g(q_i)=q_{\sigma(i)}$ for $i=1,\ldots,4$.
    \item\label{it:residue} The residue fields of $p$ and $q$ are $\k$-isomorphic.
  \end{enumerate}
\end{lemma}

\begin{proof}
The fact that \ref{it:property} implies \ref{it:aut} is \cite[Lemma 6.11]{SchneiderRelations}.
  As elements of $\PGL_3(\k)$ do not change the residue field, \ref{it:aut} implies \ref{it:residue}.
  It remains to see that \ref{it:residue} implies \ref{it:property}.
As in Remark~\ref{rem:residuefielddeg4} there exist $\alpha,\beta\in\PGL_3(k)$ such that $\alpha(p_i)=[a_i:a_i^2:1]$ and $\beta(q_i)=[b_i:b_i^2:1]$, where $a_1,\ldots,a_4$ and $b_1,\ldots,b_4$ are the roots of two irreducible polynomials $f,f'\in\k[x]$ of degree $4$.
In particular, the residue field of $p_i$ is $F_i=\k(a_i)$, and the one of $q_i$ is $F_i'=\k(b_i)$.
The residue field $F$ of $f$ is $\k$-isomorphic to $F_i$, and the residue field $F'$ is $\k$-isomorphic to $F_i'$.

  Let $L/\k$ be a finite Galois extension containing $a_1,\ldots,a_4,b_1,\ldots,b_4$.
  Having that $F_i$ and $F_i'$ are $\k$-isomorphic means that there exists $h\in\Gal(L/\k)$ (with corresponding $\tau\in S_4$ such that $h(b_i)=b_{\tau(i)}$) such that $F_{\tau(i)}'=\k(b_{\tau(i)})=\k(a_i)=F_i$. Therefore, after exchanging $b_i$ with $b_{\tau(i)}$, we find $F_i=F_i'$.
  By the fundamental theorem of Galois theory, this is equivalent to $\Gal(L/F_i)=\Gal(L/F_i')$.
  In other words: For every $g\in\Gal(L/\k)$ there exists $\sigma\in S_4$ such that $g(a_i)=a_{\sigma(i)}$ and $g(b_i)=b_{\sigma(i)}$ for $i=1,\ldots,4$.

  Therefore, also $\alpha(g(p_i)) = g(\alpha(p_i)) = g([a_i:a_i^2:1]) = [a_{\sigma(i)}:a_{\sigma(i)}^2:1] = \alpha(p_{\sigma(i)})$, and so by applying $\alpha^{-1}$ we find $g(p_i)=p_{\sigma(i)}$. Similarly, we find $\beta(g(q_i)) = \beta(q_{\sigma(i)})$ and so $g(q_i)=q_{\sigma(i)}$.
\end{proof}

\begin{lemma}\label{lem:SameFiberTypeAreEquivalent}
The classification up to equivalence of elements of Jonquières type is as follows:
\begin{enumerate}
\item\label{item:SameFiberTypeAreEquivalent--1} All elements of Jonquières type $1$ are equivalent.
\item\label{item:SameFiberTypeAreEquivalent--2} Two elements of Jonquières type $2+2$ are equivalent if and only if the pairs of associated residue fields are $\k$-isomorphic \parent{and in fact equal}.
\item\label{item:SameFiberTypeAreEquivalent--4} Two elements of Jonquières type $4$ are equivalent if and only if the associated residue fields are $\k$-isomorphic.
\end{enumerate}
\end{lemma}

\begin{proof}
\ref{item:SameFiberTypeAreEquivalent--1} is clear,
\ref{item:SameFiberTypeAreEquivalent--2} follows directly from \cite[Lemma 6.11]{SchneiderRelations},
and \ref{item:SameFiberTypeAreEquivalent--4} follows from Lemma \ref{l:residue}
\end{proof}

With the above lemma, for each fibering type and each (pair of) residue field(s) we choose one fibration $\pi\colon\P^2\rat\P^1$. We will denote it by $\pi_\times$ (Jonquières type 1), $\pi_{L,L'}$ (Jonquières type 2+2 with residue fields $L,L'$), $\pi_F$ (Jonquières type 4 with residue field $F$).
Such a choice can be made explicit:

\begin{example}
\label{ex:SpecificFibrations}
\text{ }
  \begin{description}
    \item[Jonquières type 1]
    Set
    \[
    \begin{tikzcd}[map]
    \pi_\times\colon & \P^2 & \rat & \P^1 \\
                  &{[x:y:z]}&\mapsto & {[x:z]},
    \end{tikzcd}
    \]
    which is the fibration corresponding to the pencil of lines through $[0:1:0]$.
    \item[Jonquières type 2+2]
    Let $L=\k(a_1)$ and $L'=\k(a_1')$ be two quadratic extensions over $\k$ with respective minimal polynomials $(t-a_1)(t-a_2)$, $(t-a_1')(t-a_2') \in\k[t]$.
    Let $p=\{[a_i:1:0]\}_{i=1,2}$ and $p'=\{[a_i':0:1]\}_{i=1,2}$ and set
    \[
    \begin{tikzcd}[map]
    \pi_{L,L'}\colon & \P^2 & \rat & \P^1\\
    & {[x:y:z]}&\mapsto& {[(x-a_1y)(x-a_2y) + (x-a_1'z)(x-a_2'z) - x^2:yz]},
    \end{tikzcd}
    \]
    which is a fibration corresponding to the pencil of conics through $p,p'$.
    \item[Jonquières type 4] Let $L=\k(a_1,a_2,a_3,a_4)$ be a splitting field of an irreducible polynomial of degree $4$ over $\k$ with minimal polynomial $(t-a_1)(t-a_2)(t-a_3)(t-a_4)=t^4+at^3+bt^2+ct+d\in\k[t]$.
    Then $p=\{[a_i^2:a_i:1]\}_{i=1,\ldots,4}$ is a general $4$-point, and we set
    \[
    \begin{tikzcd}[map]
    \pi_F\colon & \P^2 & \rat & \P^1\\
      &{[x:y:z]}&\mapsto &{[x^2+axy+by^2+cyz+dz^2:y^2-xz]},
    \end{tikzcd}
    \]
    which is a fibration corresponding to the pencil of conics through $p$.
    The associated residue field is $F=\k(a_1)\simeq\k(a_2)\simeq\k(a_3)\simeq\k(a_4)$, all of which are $\k$-isomorphic.
  \end{description}
\end{example}

\begin{remark}
  Note that two $d$-points with $d\leq 3$ have $\k$-isomorphic residue fields if and only if they have isomorphic splitting fields.
  For $d=4$ this is not true:
  Let $\k=\Q$ and consider the $4$-points given by $p_1=[\sqrt[4]{2}:\sqrt{4}^2:1]\in\Q[\sqrt[4]{2}]$ respectively $q_1=[\sqrt[4]{2}(1-i):\left(\sqrt[4]{2}(1-i)\right)^2:1]\in\Q[\sqrt[4]{2}(1-i)]$.
  They both have the splitting field $\Q[\sqrt[4]{2},i]$ but their residue fields are not $\k$-isomorphic.
  However, one can show that this can happen only if the Galois group over the splitting field is isomorphic to the dihedral group $D_8$.
\end{remark}

\subsection{A geometric generating set}

We can now describe a generating set for the Cremona group, valid over any perfect field:

\begin{proposition}
\label{prop:CremonaGeneratedByDelPezzoAndFiberTypes}
Let $\k$ be a perfect field.
The Cremona group $\Bir_\k(\P^2)$ is generated by the following elements:
\begin{itemize}
\item $\Aut_\k(\P^2) \simeq \PGL_3(\k)$;
\item Irreducible elements $f$ of Del Pezzo type with $1 \le \sl(f) \le 5$;
\item The Jonquières group $\Bir_\k(\P^2,\pi_\times)$ of type $1$;
\item The Jonquières groups $\Bir_\k(\P^2,\pi_{L,L'})$ of type $2+2$, for each pair of quadratic extensions $L/\k$, $L'/\k$;
\item The Jonquières groups $\Bir_\k(\P^2,\pi_F)$ of type $4$, for each residue field $F/\k$ of an irreducible polynomial of degree~$4$.
\end{itemize}
\end{proposition}

\begin{proof}
By Proposition~\ref{p:generated_by_irreducible}, $\Bir_\k(\P^2)$ is generated by the irreducible elements, which are of Del Pezzo type and/or of fibering type.
An irreducible element of Del Pezzo type has Sarkisov length at most $5$ by Proposition \ref{p:irreducible}\ref{irreducible_delpezzo}, and the case $\sl(f) = 0$ corresponds to the case $f \in \Aut_\k(\P^2)$.
If an irreducible element is of fibering type, then by Lemma~\ref{lem:FiberTypeAlways124} up to an automorphism it is equal to an element of Jonquières type 1, 2+2 or 4.
Then Lemma~\ref{lem:SameFiberTypeAreEquivalent} implies the statement, where we use Remark~\ref{rem:residuefielddeg4}.
\end{proof}

Our aim in the rest of the paper is to study each type of generator in Proposition~\ref{prop:CremonaGeneratedByDelPezzoAndFiberTypes} and show that it can be written as a product of involutions.

We shall use the basic remark that in any group, the subgroup generated by involutions is a normal subgroup.
This implies that if $G$ is a simple group containing an involution, then $G$ is generated by involutions.
For instance, for any field $\k$ and any $n \ge 2$, (except for $n=2$ and $|\k| = 2$ or $3$), the group $\PSL_n(\k)$ is simple and contains the involution
$(x_1, x_2, \dots, x_n) \mapsto (-x_1, -x_2, x_3, \dots, x_n)$ (in characteristic $\neq2$) or $(x_1, x_2, \dots, x_n) \mapsto (x_2, x_1, x_3, \dots, x_n)$ (in characteristic $2$),
so it is generated by involutions.
Observe that $\PSL_2(\FF_2) \simeq S_3$ also is generated by involutions even if it is not simple, but $\PSL_2(\FF_3) \simeq A_4$ is not generated by involutions.
On the other hand, it is known that $\PGL_n(\k)$ is not always generated by involutions: for instance a necessary and sufficient condition for $\PGL_3(\k)$ to be generated by involutions is that all elements in $\k$ are cubes.
However using the ambient birational group we have:

\begin{lemma}
\label{l:PGL3_and_jonq}
Let $\k$ be any field.
The group $\Aut_\k(\P^2) = \PGL_3(\k)$ and the Jonquières group $\Bir_\k(\P^2, \pi_\times) = \PGL_2(\k)\ltimes \PGL_2(\k(x))$ are contained in the subgroup of $\Bir_\k(\P^2)$ generated by involutions.
\end{lemma}

\begin{proof}
The group $\PGL_3(\k)$ is generated by the involutions in $\PSL_3(\k)$ and by the subgroup of dilatations of the form $[x:y:z] \mapsto [ax:y:z]$, with $a \in \k^*$.
But in affine coordinates we have
\[
(ax,y) = \Bigl (\frac{1}{x}, y \Bigr) \circ \Bigl(\frac{1}{ax}, y\Bigr),
\]
so we see that any dilatation is a composition of two birational involutions.

Similarly, over any field $\k$ with $|\k| \neq 3$, the Jonquières group is generated by the involutions in $\PSL_2(\k) \ltimes \PSL_2(\k(x))$ and by dilatations of the form $(x,y) \mapsto (ax,y)$ and $(x,y) \rat (x,a(x)y)$, which we can write as a product of two birational involutions as above.

Finally we can handle the case of $\k = \FF_3$ by noticing that $\PSL_2(\FF_3) \subset \PSL_2(\FF_3(y)) \subset \Bir_{\FF_3}(\P^2)$, and using that $\PSL_2(\FF_3(y))$ is generated by involutions.
\end{proof}

\subsection{Study of \texorpdfstring{$\Bir_\k(\P^2/\pi)$}{Bir(P2/pi)} via quadratic forms}

In this section we show that in the context of Jonquières groups of type 4 or 2+2, $\Bir_\k(\P^2/\pi)$ is isomorphic to a special orthogonal group $\SO(\k(t)^3,q)$, where $q$ is a quadratic form on $\k(t)^3$ corresponding to the conic fibration.

\subsubsection{Quadratic forms and Cartan--Dieudonn\'e}


Here, after recalling basic facts about quadratic forms in arbitrary characteristic, we give a short proof of Cartan--Dieudonné's Theorem in the special case of anisotropic quadratic spaces.

Let $K$ be an arbitrary field, and consider the vector space $E=K^n$ for $n\geq 1$.
Given a symmetric (or skew-symmetric) bilinear form $b\colon E\times E \to K$ we set $E^\perp=\{x\in E\mid b(x,y)=0~\forall y\in E\}$.
We say that $b$ is \emph{non-degenerate} if $E^\perp=\{0\}$.

A quadratic form on $E$ is a homogeneous polynomial $q\in K[x_1,\ldots,x_n]_2$ of degree $2$.
To treat the cases of $\Char K\neq 2$ and $\Char K= 2$ simultaneously while keeping the classical notations in each case, we set \[\delta= \begin{cases}
  2 & \text{if $\Char K\neq 2$}\\
  1 & \text{if $\Char K = 2$.}
\end{cases}\]
We say that the symmetric bilinear form $b$ given by
\begin{align*}
  b(x,y) = \frac{1}{\delta}(q(x+y)-q(x)-q(y))
\end{align*}
is the \textit{polar form} of the quadratic form $q$.
Note that
\[b(x,x)=\frac{1}{\delta}(q(2x)-2q(x))=\frac{1}{\delta}(2q(x))=\begin{cases}
  q(x) & \text{if $\Char K\neq 2$}\\
  0 & \text{if $\Char K= 2$}.
\end{cases}\]
We call $(E,q)$ a \textit{quadratic space}, and say it is \emph{non-degenerate} if $b$ is.
A non-zero vector $x\in E$ is \emph{isotropic} if $q(x)= 0$ and \emph{anisotropic} otherwise.
Similarly a quadratic space $(E,q)$ is \emph{isotropic} if $E$ contains an isotropic vector, and \emph{anisotropic} otherwise.
We say that $q$ has \emph{defect $d\geq1$} if $d=\dim E^\perp$ and $E^\perp$ is anisotropic.
Defect $\geq 1$ occurs only in characteristic $2$ and can be thought of as a weaker form of non-degeneracy.

The quadratic spaces we are interested in will turn out to be non-degenerate or with defect $1$, and anisotropic (see Lemma~\ref{lem:RationalPointMeansIsotropic}).

\begin{example}\label{ex:quadraticform}
	Consider $q(x)=x_1^2+x_0x_2$ as a quadratic form on $E=K^3$, where $K$ is an arbitrary field.
	The polar form of $q$ is given by $b(x,y)=2 x_1y_1-\frac{1}{\delta}(x_0y_2+x_2y_0)$.
  If $\Char K\neq 2$,  $q$ is non-degenerate.
  If $\Char K=2$ we have $E^\perp = \{(0,a,0)\mid a\in K\}$. Since $q((0,a,0))\neq 0$ for $a\in K^*$, the quadratic space $(E,q)$ has defect $1$.
\end{example}

The \emph{orthogonal group} $\O(E,q)$ of the quadratic space $(E,q)$ is the group consisting of maps $\phi\in\GL(E)$ such that $q(\phi(x))=q(x)$ for all $x\in E$.
For an anisotropic vector $a\in E$ we define the map $\tau_a\colon E\to E$ by
  \[
    \tau_a(x) = x-\delta\frac{b(x,a)}{q(a)}a.
  \]
  Note that $\tau_a(a)=-a$, and $\tau_a(x)=x$ if and only if $x\in a^\perp$.
  When $(E,q)$ is \emph{not totally degenerate}, that is $b\neq0$, the kernel $a^\perp$ of the linear map $b(a,\cdot)$ has codimension 1 for $a\notin E^\perp$. In this case, we use the following terminology for $\tau_a$:
  \begin{itemize}
    \item If $\Char K\neq 2$, the linear map $\tau_a$ is the \emph{orthogonal reflection} along $a$, in particular $\tau_a\in\O(E,q)$ is an involution with determinant $-1$ \cite[Chapter 5]{grove}.
    \item If $\Char K=2$, assume that $a\notin E^\perp$. Then $\tau_a$ is the \emph{transvection} with fixed hyperplane $a^\perp$. In particular, $\tau_a\in\GL(E)$ is an involution with determinant $1$ \cite[Chapter 1]{grove}. Observe that $\tau_a\in\O(E,q)$.
    Indeed setting $\lambda=\frac{b(x,a)}{q(a)}$, we have $q(\tau_a(x))=q(x+\lambda a)=b(x,\lambda a)+\lambda^2q(a) + q(x)= \lambda (b(x,a)+\lambda q(a))+q(x)=q(x)$, because $\lambda q(a)=b(x,a)$.
    We will call $\tau_a$ an \emph{orthogonal transvection}.
  \end{itemize}

  Also note that a quadratic space $(E,q)$ of dimension $n$ is not totally degenerate as soon as $q$ is non-degenerate or with defect $<n$.

\begin{lemma}\label{lem:ReflectionExchangesPoints}
  Let $x,y$ be two anisotropic elements in the quadratic space $(E,q)$ such that $q(x)=q(y)$ and $x-y$ is anisotropic. Then $\tau_{x-y}(x)=y$.
\end{lemma}
\begin{proof}
Using $q(x)=q(y) = q(-y)$ and $\delta b(x,x)=2q(x)$, in any characteristic we have
\begin{align*}
q(x-y) & = q(x)+q(-y) + \delta b(x,-y)\\
& = \delta b(x,x) -\delta b(x,y) \\
& = \delta b(x,x-y).
\end{align*}
Hence,
\[
\tau_{x-y}(x) = x - \frac{\delta b(x,x-y)}{q(x-y)}(x-y) = x - (x-y) = y.
\qedhere
\]
\end{proof}

The Cartan--Dieudonné's Theorem states that $\O(E,q)$ is generated by reflections, except when the underlying field has two elements, the dimension of $E$ is $4$ and $q$ is hyperbolic (see \cite[Theorem I.5.1]{Chevalley}).
When $(E,q)$ is anisotropic, there is an elementary proof:

\begin{lemma}[Anisotropic Cartan--Dieudonné]\label{lem:CartanDieudonne}
Let $(E,q)$ be an anisotropic quadratic space that is not totally degenerate.
Then every element $\phi$ of $\O(E,q)$ can be expressed as a product of $\mathrm{codim}(F)$ orthogonal reflections \parent{or transvections}, where $F=F_\phi\subset E$ is the space of fixed points of $\phi$.
\end{lemma}

\begin{proof}
First note that any product $\phi$ of $k$ orthogonal reflections \parent{or transvections} fixes a space of dimension at least $n-k$, namely the intersection of all the fixed hyperplanes from the involution. Hence, $\mathrm{codim}(F)\leq k$.

Let $\phi\in \O(E,q)$.
We proceed by induction on $\mathrm{codim}(F)$.
If $\mathrm{codim}(F)=0$, then $\phi$ is the identity.
Assume that $F$ has codimension $k\geq1$ in $E$ and let $v\in E\setminus F$,
so that $v-\phi(v)\neq 0$.
Since by assumption $E$ is anisotropic, applying Lemma~\ref{lem:ReflectionExchangesPoints} we find that $\tau_{v-\phi(v)}(v)=\phi(v)$.
In particular, $\tau_{v-\phi(v)}\neq\id$, so $v-\phi(v)\notin E^\perp$ and so the map is indeed an orthogonal reflection \parent{or transvection} since we assume $q$ to be not totally degenerate.

Let $x\in F$, that is, $\phi(x)=x$.
We compute
\[
b(x,v-\phi(v)) = b(x,v)-b(x,\phi(v))
= b(x,v)-b(\phi(x),\phi(v))= 0.
\]
Therefore, $\tau_{v-\phi(v)}(x)=x$, and so $\tau_{v-\phi(v)}$ fixes $F$.

This implies that $\phi'=\phi^{-1}\circ \tau_{v-\phi(v)}$ fixes both $F$ and $v$, hence it fixes a subspace containing $F$ and $v$, which has codimension at most $k-1$.
By the remark at the beginning of the proof, the fixed space of $\phi'$ has codimension exactly $k-1$, and hence by induction $\phi'$ is the composition of $k-1$ orthogonal reflections \parent{or transvections}, and so $\phi$ is the composition of $k=\mathrm{codim}(F)$ such involutions.
\end{proof}

\subsubsection{Similitudes and the special orthogonal group}


Let $(E,q)$ be a quadratic space over a field $K$.
A \emph{similitude} of $(E,q)$ is a map $f \in \GL(E)$ such that there exists $\lambda\in K^*$ with $q(f(x))=\lambda q(x)$ for all $x\in E$. The constant $\lambda$ is called the \emph{multiplier} of $f$.
We denote the group of similitudes by $\GO(E,q)\subset \GL(E)$.
The map $\GO(E,q)\to K^*$ given by $f\mapsto \lambda$ is a group homomorphism with kernel $\O(E,q)$.
We define $\PGO(E,q)=\GO(E,q)/K^*$ to be the image of $\GL(E)\to\PGL(E)$ restricted to $\GO(E,q)$, so we can write
\begin{align*}
  \PGO(E,q) &= \{[A]\in\PGL_n(K)\mid A\in \GO(E,q)\}\\
  &= \{[A]\in\PGL_n(K)\mid \exists \lambda\in K^*:q\circ A =\lambda q\}
\end{align*}
where $n$ is the dimension of $E$.

\begin{lemma}
	The following hold:
	\begin{enumerate}
		\item\label{det-b} Let $b$ be a \parent{skew-}symmetric non-degenerate bilinear form on $E$, and let $M\in\GL(E)$ such that
		$b(Mx,My)=b(x,y)$ for all $x,y\in E$. Then $\det M=\pm1$.
		\item\label{det-q} Let $(E,q)$ be a quadratic space that is non-degenerate or with defect $1$. Then $\det(M)=\pm1$ for all $M\in \O(E,q)$.
	\end{enumerate}
\end{lemma}

\begin{proof}
	In \ref{det-b}, since $b$ is a non-degenerate bilinear form there exists $A\in\GL(E)$ such that $b(x,y)=x^tAy$ for all $x,y\in E$.
	The assumption on $M$ implies that $x^tAy=x^tM^tAMy$ for all $x,y\in E$. Therefore $A=M^tAM$. As $A$ is invertible this implies that $1=\det(M)^2$.

	For \ref{det-q} observe that if $M\in \O(E,q)$ then $b(Mx,My)=b(x,y)$ for all $x,y\in E$, where $b$ is the polar form of $q$.
	So if $q$ is non-degenerate, part \ref{det-b} implies the statement.
  Assume now that $q$ is with defect $1$, that is $E^\perp$ is $1$-dimensional and anisotropic. In particular, we are in characteristic $2$, and so we will show that $\det(M)=1$.

	First of all, if $M$ fixes a $1$-dimensional subspace $L$ on which $q$ does not vanish outside the origin, then it fixes $L$ pointwise:
	Let $x\in L$, $x\neq 0$ and let $\lambda\in K$ such that $M(x)=\lambda x$.
	Since $M\in \O(E,q)$, $q(x)=q(M(x))=q(\lambda x)=\lambda^2q(x)$, and as $q(x)\neq 0$ this implies that $\lambda=\pm1$. Being in characteristic $2$ implies that $M$ fixes the point.

	By extension of the basis we can write $E = E^\perp\oplus W$ for some $n-1$-dimensional subspace $W$.
	Writing $M$ in this basis we get a matrix $\left(\begin{smallmatrix} a & B\\ 0 & M'\end{smallmatrix}\right)$	with $a\in K$, $M'\in\GL_{n-1}(K)$.
	Note that $M$ fixes $E^\perp$.
	By the remark above this implies that $M$ fixes $E^\perp$ pointwise, therefore $a=1$.
	Next, we show that $\det(M')=1$. For this we consider $M'$ as an endomorphism on the quotient $\bar E = E/E^\perp$. Note that the polar bilinear form $b$ on $E$ induces a bilinear form $\bar b$ on $\bar E$ by $\bar b(\bar x,\bar y)=b(x,y)$ for $\bar x,\bar y\in\bar E$.
	Now, $\bar b$ is symmetric and non-degenerate, and for all $\bar x, \bar y$ we have $\bar b(M'\bar x,M'\bar y)=\bar b(\bar x,\bar y)$.
	By~\ref{det-b}, the determinant of $M'$ is $1$.
	Therefore, $\det(M)=a\det(M')=1$.
\end{proof}

We set \begin{align*}
  \SO(E,q) = \{f\in \O(E,q)\mid \det f=1\}.
\end{align*}

In odd dimensions there is the following relationship between similitudes and the orthogonal group:

\begin{proposition}[{\cite[Propositions 12.4 and 12.6]{BookOfInvolutions}}]\label{prop:PGOvsSO}
  Let $(E,q)$ be a quadratic space in odd dimension over an arbitrary field $K$. Assume that $q$ is non-degenerate or with defect $1$.
  Then
  \begin{align*}
    \GO(E,q) = \SO(E,q)\cdot K^* &\simeq \SO(E,q)\times K^*\\
    \PGO(E,q) &\simeq \SO(E,q).
  \end{align*}
\end{proposition}

\begin{lemma}\label{lem:SOGeneratedByInvolutions}
Let $(E,q)$ be an anisotropic quadratic space of odd dimension $n$ over an arbitrary field~$K$. Assume that $q$ is not totally degenerate.
Then each element of $\SO(E,q)$ can be written as the composition of at most $n$ involutions \parent{and in fact at most $n-1$ when $\Char K \neq 2$}.
\end{lemma}

\begin{proof}
If $\Char K = 2$ then $\SO(E,q) = \O(E,q)$ and the result is Lemma \ref{lem:CartanDieudonne}.

So assume now that $\Char K \neq 2$ and let $A\in \SO(E,q)\subset \O(E,q)$.
Lemma~\ref{lem:CartanDieudonne} implies that $A$ is a composition of at most $n$ orthogonal reflections in $\O(E,q)$,
all of which have determinant $-1$.
  Hence there exist $B_1,\ldots, B_k\in \O(E,q)$ with determinant $-1$ and $k\leq n$ such that $A=B_1\cdots B_k$. Taking the determinant on both sides gives that $k$ has to be even, so $k\leq n-1$.
Since the dimension is odd, $\det(-A)=-\det(A)$ for $A\in\GL_n(K)$.
Use the surjective group homomorphism $\O(E,q)\to \SO(E,q)$ given by
\begin{align*}
\eta\colon A &\mapsto\begin{cases}
-A & \text{if } \det(A)=-1\\
A & \text{if } \det(A)=1,
\end{cases}
\end{align*}
  which sends $B_i$ onto involutions $\eta(B_i)\in \SO(E,q)$ whose composition is $\eta(A)=A$.
\end{proof}

\subsubsection{Conic fibrations and quadratic forms}

The following lemma shows in particular that defect $1$ is a useful notion in characteristic $2$, tightly related to the notion of ``strange curves'' \cite[Chapter IV, Section 3]{Hartshorne}:

\begin{lemma}\label{lem:qNonSing}
  Let $q\in K[x,y,z]$ be a homogeneous polynomial of degree $2$ over an arbitrary field $K$ \parent{not necessarily perfect}.
  Assume that $q$ is irreducible over $\alg K$ and consider the conic $C=V_{\P^2_K}(q)$.
  Then the quadratic space $(K^3,q)$ is non-degenerate if $\Char \k\neq 2$, and has defect $1$ if $\Char K = 2$.

  Moreover, if $\Char K =2$, the $1$-dimensional kernel $E^\perp$ corresponds to the point $p\in\P^2(K)$ where all tangent lines to $C$ meet.
\end{lemma}

\begin{proof}
Write $E=K^3$.
Note that $C$ contains a $K$-point if and only if there exists a non-zero element $x\in K^3$ such that $q(x)=0$, that is, if $(E,q)$ is isotropic.

If $(E,q)$ is isotropic then $C$ contains a rational point and hence there is a linear $K$-automorphism of $E$ sending $C$ onto the conic $V(x_1^2-x_0x_2)$ from Example~\ref{ex:quadraticform}. Hence, $q$ is non-degenerate if $\Char K\neq 2$ and with defect $1$ if $\Char K=2$.

Consider now the anisotropic case.
If $\Char K\neq 2$, having $q(x)=b(x,x)\neq0$ for all non-zero $x\in K^3$ directly implies that $q$ is non-degenerate.
Assume now $\Char K=2$.
Let $B = (b_{ij})_{1 \le i,j \le 3}$ be the matrix with $b_{ij} = b(e_i,e_j)$, which has the form $B = \left(\begin{smallmatrix}
    0 & a & b\\
    a & 0 & c\\
    b & c & 0
  \end{smallmatrix}\right)\in\GL_3(K)$.
Since we are in the anisotropic case it is enough to show that $\dim E^\perp=1$.
If $B$ is the zero matrix then $q$ is a sum of square monomials, hence $q$ is reducible over $\alg K$, a contradiction.
So the kernel $E^\perp=\{(cx,bx,ax)\in E\mid x\in K\}$ is $1$-dimensional.

Finally, note that in characteristic $2$ all $V(q)$ with the same polar form $b$ have the same tangent lines.
With coordinates as above, we can assume up to a permutation that $a\neq 0$ and so the tangent lines are given by $\mu x+ \lambda y+ \frac{\mu c+\lambda b}{a}z=0$ for $[\lambda:\mu]\in \P^1(\alg K)$. All of them go through the point $[c:b:a]$.
\end{proof}

The following lemma shows that any quadratic space $(\k(t)^3,q)$ corresponding to a Jonquières group $\Bir_\k(\P^2/\pi)$ of type 2+2 or 4 is anisotropic, and either is non-degenerate or has defect $1$.

\begin{lemma}\label{lem:RationalPointMeansIsotropic}
  Let $\{p_1,\ldots,p_4\}\subset\P^2(\bk)$ be a Galois-invariant set of four points such that no three are collinear and let $q_1,q_2\in\k[x,y,z]$ be two homogeneous polynomials of degree $2$ such that $C_1=V(q_1)$, $C_2=V(q_2)$ are two distinct conics intersecting at the four points.
 Let $q=q_1+tq_2\in\k[t,x,y,z]\subset\k(t)[x,y,z]$ and $C=V_{\P^2_{\k(t)}}(q)$.
  Then the quadratic space $(\k(t)^3,q)$ is non-degenerate \parent{$\Char \k\neq2$} or has defect $1$ \parent{$\Char \k =2$}, and the following are equivalent:
  \begin{enumerate}
    \item\label{item:RationalPointMeansIsotropic--4pts} $C_1\cap C_2$ contains a $\k$-point;
    \item\label{item:RationalPointMeansIsotropic--GeneralFiber} $C\subset\P^2_{\k(t)}$ contains a $\k(t)$-point;
    \item \label{item:RationalPointMeansIsotropic--QuadraticSpace} The quadratic space $(\k(t)^3,q)$ is isotropic.
  \end{enumerate}
\end{lemma}

\begin{proof}
  First, note that since no three of the four points are collinear, $q$ is irreducible over $\alg{\k(t)}$.
  By Lemma~\ref{lem:qNonSing}, $q$ is non-degenerate if $\Char K\neq 2$, and has defect $1$ if $\Char K=2$.

  The equivalence \ref{item:RationalPointMeansIsotropic--GeneralFiber} $\iff$ \ref{item:RationalPointMeansIsotropic--QuadraticSpace} is direct.
  Since $\k\subset\k(t)$, \ref{item:RationalPointMeansIsotropic--4pts} implies \ref{item:RationalPointMeansIsotropic--GeneralFiber}.
  We now show the converse direction.
  Assume that $C$ contains a $\k(t)$-point.
  Hence, $C$ is isomorphic to $\P^1_{\k(t)}$ over $\k(t)$.
  In particular,
  \[X=\{([x:y:z],t)\mid q(t,x,y,z)=0\}\subset\P^2\times\mathbb{A}^1\] contains sections that are defined over $\k$.
  Note that the projection $\rho\colon X\to\P^2$ is the blow-up at the four points $p_1,\ldots,p_4$.
  We assume now that $C_1\cap C_2$ contains no $\k$-point and we will show that $X$ contains no sections defined over $\k$, giving a contradiction.
  As none of the four points $p_1,\ldots,p_4\in\P^2(\bk)$ is a $\k$-point, they form either one $4$-point or two $2$-points (say $\{p_1,p_2\}$ and $\{p_3,p_4\}$).
  Let $D\subset X$ be a curve, defined over $\k$, and let $D'=\rho(D)\subset\P^2$.
  The essential remark is that all geometric components of a point have the same multiplicity on $D'$.
  Write $d$ for the degree of $D'$, and $m$ for the multiplicity of $D'$ at $p_1,p_2$, and $m'$ for the one at $p_3, p_4$. After resolving the singularities $x_1,\ldots,x_r$ of $D$ and writing $L$ for the strict transform of a general line in $\P^2$, and $E_i$ for the exceptional divisor of $p_i$ and $E_{x_i}$ for the one of $x_i$,
  we can write the intersection of $D$ with a general fiber $f$ of $X$ as
  \begin{align*}
    D\cdot f &= (dL-m(E_1+E_2)-m'(E_3+E_4)-\sum m_{x_i}(D)E_{x_i})(2L-(E_1+\ldots+E_4))\\
    & = 2d-2m-2m'.
  \end{align*}
  As $D\cdot f$ is even, $D$ is not a section.
\end{proof}

\begin{remark}\label{rem:qNonDeg}
	For $\Char \k=2$, we have seen in Lemma~\ref{lem:qNonSing} that the defect corresponds to the point where all the tangent lines to $C\subset \P^2_{\k(t)}$ meet. This in turn corresponds to the line in $\P^2_\k$ that is tangent to all conics that go through $p_1,\ldots,p_4$.
	This line played a crucial role in \cite{SchneiderF2}, where involutive generators of the groups $\Bir_{\F_2}(\P^2/\pi_{\F_4,\F_4})$ and $\Bir_{\F_2}(\P^2/\pi_{\F_{16}})$ were described explicitely.
\end{remark}

\begin{lemma}\label{lem:ProjectiveOrthogonal}
  Let $\{p_1,\ldots,p_4\}\subset\P^2(\bk)$ be a Galois-invariant set of four points such that no three are collinear and let $q_1,q_2\in\k[x,y,z]$ be two homogeneous polynomials of degree $2$ such that the pencil of conics through $p_1,\ldots,p_4$ is given by $\pi\colon\P^2\rat\P^1$, $\pi([x:y:z])=[q_1:q_2]$.
  Let $q=q_1+tq_2\in\k[t][x,y,z]\subset\k(t)[x,y,z]$, which we see as a quadratic form of the $3$-dimensional vector space $E = \k(t)^3$ over $\k(t)$.
  Then
  \[\Bir_\k(\P^2/\pi) \simeq \PGO(\k(t)^3,q).\]
\end{lemma}

\begin{proof}
  In Lemma~\ref{lem:RationalPointMeansIsotropic} we have observed that the quadratic space $(\k(t)^3,q)$ is non-degenerate or with defect $1$.
By construction the generic fiber $C$ of $\pi$ is given by the zero set of $q$ in $\P^2_{\k(t)}$.
Note that the splitting field $L$ of the four points $p_1,\ldots,p_4$ is a Galois extension of $\k$ over which the generic fiber $C$ is nonempty.
So \cite[Lemma 3.29]{SchneiderF2} yields
\[
    \Bir_\k(\P^2/\pi)  \simeq \Bir_{\k(\P^1)}(C) =\{f\in\Aut_{\k(\P^1)}(\P^2)\mid f(C)=C\}.
  \]
We have $f(C)= C$ if and only if $V(q\circ f)=f^{-1}(C)=C=V(q)$, so the above group is equal to
  \begin{align*}
  \{f\in\Aut_{\k(\P^1)}(\P^2)\mid \exists\lambda\in \k(\P^1)^*
  :q\circ f=\lambda q\}&= \PGO(\k(\P^1)^3,q).
  \end{align*}
  Since $\k(\P^1)\simeq\k(t)$, the statement follows.
\end{proof}

\begin{proposition}\label{prop:KernelGeneratedByInvo}
  The following hold:
  \begin{enumerate}
    \item Let $F/\k$ be a residue field of an irreducible polynomial of degree $4$. Then $\Bir_\k(\P^2/\pi_F)\simeq \SO(\k(t)^3,q_F)$.
    \item Let $L,L'$ be two field extensions of degree $2$ over $\k$. Then $\Bir_\k(\P^2/\pi_{L,L'})\simeq \SO(\k(t)^3,q_{L,L'})$.
  \end{enumerate}
  In particular, each element of these groups is a composition of at most three \parent{two if $\Char \k \neq 2$} involutions.
\end{proposition}

\begin{proof}
  None of the four geometric points where $\pi_F$ respectively $\pi_{L,L'}$ are not defined is $\k$-rational. Therefore, Lemma~\ref{lem:RationalPointMeansIsotropic} implies that in both cases the associated quadratic space $(\k(t)^3,q)$ is anisotropic,
  and so by Lemma~\ref{lem:SOGeneratedByInvolutions} each element in the group $\SO(\k(t)^3,q)$ is the product of at most three (two if $\Char\k\neq 2$) involutions.
  With Lemma~\ref{lem:ProjectiveOrthogonal} and Proposition~\ref{prop:PGOvsSO} we get
  \begin{equation*}
    \Bir_\k(\P^2/\pi)  \simeq \PGO(\k(t)^3,q) \simeq \SO(\k(t)^3,q).\qedhere
  \end{equation*}
\end{proof}

\subsection{The image for Jonquières type 4}

Here we study the image $\Aut_\k^{\pi_F}(\P^1)$ from the exact sequence \eqref{exact_sequence} from page \pageref{exact_sequence}, for fibering type 4.
Denote by $\{p_1,\ldots,p_4\}\in\p^2(\bk)$ the base points of $\pi_F$, and by $L_{ij}$ the line through $p_i$ and $p_j$.
In the following lemma we show that the permutation of the three singular fibers by a map in $\Bir_\k(\p^2,\pi_F)$ is always achieved by an automorphism.

\begin{lemma}\label{lem:image_Jonq_4}
Let $F/\k$ be a residue field of an irreducible polynomial of degree $4$.
	Let $\varphi\in\Bir_\k(\p^2,\pi_F)$. Then there exists $\alpha\in\Aut_\k(\p^2)\cap\Bir_\k(\p^2,\pi_F)$ such that $\pi_F\circ\alpha =\pi_F\circ\varphi$.
\end{lemma}

\begin{proof}
	The blow-up of the $4$-point gives a link $\P^2\linkI{4} X_5$ with $X_5\in\Cl_5$.
  Hence, $\phi\in\Bir_\k(\P^2,\pi_F)$ factors through a birational map $X_5\rat X_5$ that preserves the fibration $X_5\to\P^1$.
  By \cite[Corollary 3.2]{SchneiderRelations}, $\phi$ sends the set of three singular fibers onto itself.
  As any element of $\Aut_\k(\P^1)$ is uniquely determined by its value at three points, the image of $\Bir_\k(\P^2,\pi)$ in $\Aut_\k(\P^1)$ is isomorphic to a subgroup of the symmetric group $S_3$ determined by the action on the projection of the three singular fibres $L_{ij}\cup L_{kl}$ for $\{i,j,k,l\}=\{1,2,3,4\}$.
  Hence, it is enough to find $\alpha\in\Aut_\k(\p^2)\cap \Bir_\k(\p^2,\pi_F)$ such that $\alpha(L_{ij}\cup L_{kl})=\varphi(L_{ij}\cup L_{kl})$ for $\{i,j,k,l\}=\{1,2,3,4\}$. We show that there is an $\alpha$ that satisfies $\alpha(L_{ij})=\varphi(L_{ij})$.

  We follow the proof of \cite[Proposition~6.12]{SchneiderRelations}.
  Write $F_{ij}=\varphi(L_{ij})$, and note that there are two configurations possible (see \cite[Lemma 6.5 (4)]{SchneiderRelations}):
  \begin{enumerate}
    \item\label{item:3linesintersect-a} For each $i$, the three lines $F_{ij}$ for $j\neq i$ intersect in one point, or
    \item\label{item:3linesintersect-b} for each $i$, the three lines $F_{jk}$ for $i\notin \{j,k\}$ intersect in one point.
  \end{enumerate}
  In both cases, write $q_i$ for the intersection point of the three lines, and note that $q_i\in\{p_1,\ldots,p_4\}$ since $\varphi(L_{ij})=L_{kl}$ for some $k,l$.
  Let $L$ be the Galois closure of $F/\k$.
  For $g\in\Gal(L/\k)$ let $\tau\in S_4$ be such that $p_i^g=p_{\tau(i)}$.
  In case \ref{item:3linesintersect-a}, $q_i^g$ is the intersection of the lines $\varphi(L_{ij})^g=\varphi(L_{\tau(i),\tau(j)})$ for $j\neq i$, which is $q_{\tau(i)}$.
  In case \ref{item:3linesintersect-b}, $q_i^g$ is the intersection of the three lines $\varphi(L_{jk})^g=\varphi(L_{\tau(j),\tau(k)})$ for $i \notin\{j,k\}$, which is again $q_{\tau(i)}$.
  Hence, in both cases there exists $\alpha\in\Aut_\k(\p^2)$ such that $\alpha(p_i)=q_i$ \cite[Lemma 6.11]{SchneiderRelations}.
  In particular, $\alpha$ preserves the pencil of conics through the $p_i$, which concludes the proof.
\end{proof}

\begin{corollary}
\label{c:Jonq_type_4}
Let $\pi_F\colon \P^2 \rat \P^1$ be a rational fibration of type $4$, associated with a residue field $F/\k$ of an irreducible polynomial of degree $4$ \parent{see Example \textup{\ref{ex:SpecificFibrations}}}.
Then $\Bir(\P^2, \pi_F)$ is contained in the subgroup of $\Bir_\k(\P^2)$ generated by involutions.
\end{corollary}

\begin{proof}
By the exact sequence \eqref{exact_sequence} and Lemma \ref{lem:image_Jonq_4}, the group $\Bir(\P^2, \pi_F)$ is generated by $\Bir(\P^2/\pi_F)$ and by $\Aut_\k(\p^2)\cap \Bir_\k(\p^2,\pi_F)$.
By Proposition~\ref{prop:KernelGeneratedByInvo}, any element in $\Bir(\P^2/ \pi_F) \simeq \SO_3(\k(t)^3,q_F)$ is a product of at most three (two if $\Char \k\neq 2$) involutions.
On the other hand, by Lemma \ref{l:PGL3_and_jonq} the full automorphism group $\Aut_\k(\P^2) \simeq \PGL_3(\k)$ is contained in the subgroup of $\Bir_\k(\P^2)$ generated by involutions, which gives the result.
\end{proof}

\input{fibering_22}

%% file: fibering_22.tex
\subsection{The image for Jonquières type 2+2}
\label{sec:2+2}

Let $L/\k$ and $L'/\k$ be two quadratic extensions and write $K=LL'$ for the composite field.
The main result of this section is Proposition \ref{prop:DescriptionOfJonq22}, which gives a family of involutions in $\Bir_\k(\p^2,\pi_{L,L'})$ that surject onto the image $\Aut_\k^{\pi_{L,L'}}(\p^1)$.

\subsubsection{The statement}

Denote by $\pi_1,\pi_2\colon\p^1\times\p^1 \to \P^1$ the projection onto the first respectively second $\p^1$. We call the fibers of $\pi_1$ \emph{vertical} and the fibers of $\pi_2$ \emph{horizontal} curves.
First, we will describe geometrically a birational map $\eps\colon\p^2\rat\p^1\times\p^1$ that is defined over $K=LL'$.
We will see that the fibration $\pi_{L,L'}$ corresponding to the pencil of conics through the two points of degree $2$ is sent onto the fibration that is given by $\pi_1\colon \p^1\times\p^1\to\p^1$ (Lemma~\ref{lem:GeometricDescriptionOfExorcist}).
Next, we will describe the map $\eps$ explicitly on affine charts and keep track of the induced Galois action $\eps\Gal(K/\k)\eps^{-1}$.

\begin{lemma}\label{lem:GeometricDescriptionOfExorcist}
  Let $\{p_1,\ldots,p_4\}\subset\p^2(\bk)$ be four points in general position and $\pi\colon\p^2\rat\p^1$ be the fibration corresponding to the pencil of conics through the four points.
  Denote by $K$ the splitting field of $p_1,\ldots,p_4$.
  Let $\eps\colon\p^2\rat \p^1\times\p^1$ be the composition $\gamma\circ\beta\circ\alpha$ of the following maps, all defined over $K$:
  \begin{itemize}
    \item $\alpha\colon\p^2\rat \p^1\times\p^1$ is the blow-up of $p_1,p_2$ followed by the contraction of the strict transform of the line through $p_1, p_2$,
    \item $\beta\colon\p^1\times\p^1\stackrel{_\sim}\to\p^1\times\p^1$ is an automorphism preserving the two rulings,
    \item $\gamma\colon\p^1\times\p^1\rat\p^1\times \p^1$ is the blow-up of $\beta \circ \alpha(p_3), \beta \circ \alpha(p_4)$, followed by the contraction of the strict transforms of the horizontal curves through $\beta \circ \alpha(p_3)$, respectively $\beta \circ \alpha(p_4)$.
  \end{itemize}
  Then, $\eps$ preserves the fibrations $\p^2/\pi$ and $\p^1\times\p^1/\pi_1$, that is, there exists $\varphi\in\Aut_\bk(\p^1)$ such that $\pi\circ\eps^{-1}=\varphi\circ\pi_1$.
\end{lemma}

\begin{figure}[ht]
    \begin{center}
    \begin{tikzpicture}
      \coordinate (A) at (0,0);
      \draw[name path=conic1] (A) ellipse (2cm and 1cm);
      \draw[name path=conic2] (A) ellipse (1cm and 2cm);
      \path[name intersections={of=conic1 and conic2,by ={p1,p2,p3,p4}}];

      \node [draw] at (A) [circle through={(p1)}] {};
      \draw[blue,ultra thick,shorten <=-0.5cm, shorten >=-0.5cm] (p1) -- (p4);
      \draw[gray, shorten <=-0.5cm, shorten >=-0.5cm] (p2) -- (p1);
      \draw[gray, shorten <=-0.5cm, shorten >=-0.5cm] (p3) -- (p1);

      \fill[green] (p1) circle (3pt);
      \fill[green] (p4) circle (3pt);
      \fill[FireBrick] (p2) circle (3pt);
      \fill[FireBrick] (p3) circle (3pt);

		\node[] at (3,0) {$\stackrel{\beta\circ\alpha}{\rat}$};

      \coordinate (B1) at (4,1);
      \coordinate (B2) at (6,-1);
      \coordinate (B3) at (6.5,1.3);

      \draw[green,ultra thick,shorten <=-0.5cm,shorten >=-0.2cm] (B3) to (4,1.3);
      \draw[green,ultra thick,shorten <=-0.5cm,shorten >=-0.2cm] (B3) to (6.5,-1);
      \draw[gray,shorten <=-0.5cm,shorten >=-0.2cm] (B1) to (6.5,1);
      \draw[gray,shorten <=-0.5cm,shorten >=-0.5cm] (B2) to (4,-1);
      \draw[shorten <=-0.7cm,shorten >=-0.7cm] (B1) to (B2);
      \draw[shorten <=-0.7cm,shorten >=-0.7cm] (B1) to [bend left=30] (B2);
      \draw[shorten <=-0.7cm,shorten >=-0.7cm] (B1) to [bend right=30] (B2);
      \fill[FireBrick] (B1) circle (3pt);
      \fill[FireBrick] (B2) circle (3pt);
      \fill[blue] (B3) circle (3pt);

		\node[] at (8,0) {$\stackrel{\gamma}{\rat}$};

      \coordinate (C1) at (8.7,0.5);
      \coordinate (C2) at (10.2,-1);
      \coordinate (C3) at (10.5,1.3);
      \draw (9,-1.5) -- (9,1.5);
      \draw (9.5,-1.5) -- (9.5,1.5);
      \draw (10,-1.5) -- (10,1.5);
      \draw[green,ultra thick,shorten <=-0.2cm,shorten >=-0.5cm] (C3) to (9,1.3);
      \draw[FireBrick,ultra thick,shorten <=-0.2cm,shorten >=-0.5cm] (C1) -- (10.5,0.5);
      \draw[FireBrick,ultra thick,shorten <=-0.5cm,shorten >=-0.2cm] (C2) -- (8.5,-1);
      \draw[green,ultra thick] (C1) to[out=90,in=-160] (C3);
      \draw[green,ultra thick] (C1) to[out=-90,in=140] (C2);
      \fill[blue] (C3) circle (3pt);
      \fill[gray] (C1) circle (3pt);
      \fill[gray] (C2) circle (3pt);

    \end{tikzpicture}
    \caption{The birational map $\eps\colon\p^2\rat\p^1\times\p^1$ from Lemma~\ref{lem:GeometricDescriptionOfExorcist}.}
    \label{fig:alpha_gama}
  \end{center}
\end{figure}

\begin{proof}
  The birational map $\beta\circ\alpha$ sends the pencil of conics through $p_1,\ldots,p_4$ onto the pencil of diagonal curves going through $\beta\circ\alpha(p_3),\beta\circ\alpha(p_4)$, which are sent by $\gamma$ onto the pencil of vertical curves (see Figure \ref{fig:alpha_gama}).
\end{proof}

\begin{corollary}\label{cor:Jonq22BirVsPGL}
  Take the notation from Lemma~\textup{\ref{lem:GeometricDescriptionOfExorcist}}. Then
  \[
    \Bir_\k(\p^2,\pi)\simeq \left(\PGL_2(K)\ltimes \PGL_2(K(x))\right)^{\eps\Gal(K/\k)\eps^{-1}}.
  \]
\end{corollary}

\begin{proof}
  Lemma~\ref{lem:GeometricDescriptionOfExorcist} implies that $\Bir_\k(\p^2,\pi)\simeq \Bir_K(\p^1\times\p^1,\pi_1)^{\eps\Gal(K/\k)\eps^{-1}}$. The statement follows since
  \[
    \Bir_K(\p^1\times\p^1,\pi_1)\simeq \PGL_2(K)\ltimes \PGL_2(K(x)),
  \]
  using that the elements $j\in\Bir_K(\p^1\times\p^1,\pi_1)$ are of the form $j(x,y)=\left(\frac{ax+b}{cx+d},\frac{A(x)y+B(x)}{C(x)y+D(x)}\right)$
  with $M=\left(\begin{smallmatrix}
    a & b\\ c& d
  \end{smallmatrix}\right)\in\PGL_2(K)$, $M'=\left(\begin{smallmatrix}
    A(x) & B(x)\\ C(x)& D(x)
  \end{smallmatrix}\right)\in\PGL_2(K(x))$.
\end{proof}

\begin{setup}\label{setup:22}
Let $L/\k$, $L'/\k$ be two quadratic extensions with $L=\k(\two)$, $L'=\k(\two')$ and write $K=LL'$ for the composite field.
\begin{enumerate}
  \item If $L=L'$ let $g$ be the generator of $\Gal(K/\k)=\Gal(L/\k)\simeq\Z/2\Z$.
  \item If $L\neq L'$ let $h,h'$ be the generators of the Galois groups $\Gal(L/\k)$, $\Gal(L'/\k)$ and set $g=hh'$. So $g,h$ generate $\Gal(K/\k)=\Gal(L/\k)\times \Gal(L'/\k)\simeq\Z/2\Z\times\Z/2\Z$.
\end{enumerate}
We write $x^g$ for the action of $g$ on $x\in\bk$, and $A^g$ for the action of $g$ on the coefficients of $A\in\bk(x)$; similarly for $h,h'$.

Note that $\two^g=\two^h$, $\two'^g=\two'^{h'}$ for $L\neq L'$.
As in Example~\ref{ex:SpecificFibrations}, we take $\pi_{L,L'}$ to be the fibration given by the pencil of conics through the $2$-points $p'=\{p_1,p_2\}=\{[\two':1:0],[\two'^{g}:1:0]\}$ and $p=\{p_3,p_4\}=\{[\two:0:1],[\two^g:0:1]\}$.
Hence $g\in\Gal(K/\k)$ acts on $\{p_1,\ldots,p_4\}$ as the permutation $(12)(34)$. In the case of $L\neq L'$, the elements $h,h'$ act respectively as the transpositions $(34)$ and $(12)$.


We shall use Lemma~\ref{lem:GeometricDescriptionOfExorcist} with the following maps $\alpha, \beta, \gamma$, which we express in the affine charts
\begin{align*}
  \A^2&\to\p^2 & \text{ and } && \A^2&\to\p^1\times\p^1\\
  (x,y)&\mapsto [x:y:1] &&& (x,y)&\mapsto ([x:1],[y:1]).
\end{align*}
We set:
\begin{align*}
\alpha(x,y) &= (x-\two' y, x- \two'^g y), &
\alpha^{-1}(x,y) &= \left( \frac{\two'^g x - \two' y}{\two'^g - \two'},\frac{x-y}{\two'^g - \two'}\right),\\
\beta(x,y) &= \left( \frac{x-\two}{-x + \two^{g}}, \frac{y-\two^{g}}{-y + \two}\right), &
\beta^{-1}(x,y) &= \left( \frac{\two^{g} x + \two}{x + 1}, \frac{\two y+\two^{g}}{y + 1}\right), \\
\gamma(x,y) &= (xy,y), &
\gamma^{-1}(x,y) &= \left(\frac{x}y,y\right).
\end{align*}
The map $\alpha\colon \P^2 \rat \P^1 \times \P^1$ blows-up the two components $p_1,p_2$ of $p'$, and contracts the line at infinity in $\P^2$.
Moreover, $\alpha$ sends the components $p_3=(\two,0),p_4 =(\two^g,0)$ of $p$ onto the points $(\two,\two), (\two^g,\two^g)$, which are then sent by $\beta \in \Aut(\P^1 \times \P^1)$ onto $([0:1],[1:0]), ([1:0],[0:1])$.
Observe that $\eps=\gamma\circ\beta\circ\alpha$ is defined over $K$.

In the computations we shall use the notation of the following diagram for the action of $g$; similarly for $h,h'$.
\begin{equation}
\label{diag:gg'}
\tag{Dia}
\begin{tikzcd}
\A^2 \ar[rrr,bend left,"\eps"] \ar[r,"\alpha"] \ar[d,"g"]& \A^2 \ar[r,"\beta"] \ar[d,"g_\alpha"] & \A^2\ar[r,"\gamma"] \ar[d,"g_\beta"] & \A^2 \ar[d,"g_\gamma"] \\
\A^2 \ar[r,"\alpha"] & \A^2 \ar[r,"\beta"] & \A^2\ar[r,"\gamma"] & \A^2.
\end{tikzcd}
\end{equation}
\end{setup}

In the rest of this section we prove:
\begin{proposition}\label{prop:DescriptionOfJonq22}
  Assume Set-Up~\textup{\ref{setup:22}}.
  Denote by $H_{L,L'}\subset\Bir_{K}(\A^2)$ the group generated by involutions of the form
 \begin{align*}
   (x,y) & \dashrightarrow \left(\frac{1}{\mu x} , \frac{1}{\lambda y}\right)
 \end{align*}
 where $\mu=\lambda\lambda^g$ and
 \begin{enumerate}
   \item $\lambda\in K^*=L^*$ if $L= L'$,
   \item $\lambda \in K^*$ such that $\lambda\lambda^h=1$ if $L\neq L'$.
 \end{enumerate}
 Then the group $\eps^{-1} H_{L,L'}\eps\subset\Bir_\k(\p^2,\pi_{L,L'})$ surjects onto $\Aut_\k^{\pi_{L,L'}}(\p^1)$.
\end{proposition}

\subsubsection{The induced Galois action}

\begin{lemma}\label{lem:actiong}
  Assume Set-Up~\textup{\ref{setup:22}}. Then the induced Galois action $\eps\circ g\circ \eps^{-1}$ is given by
  \begin{align*}
 g_{\gamma}(x,y) & = \Bigl(x^g,\frac{x^g}{y^g}\Bigr).
\end{align*}
\end{lemma}

\begin{proof}
We compute step by step the maps on Diagram \eqref{diag:gg'}. First:
\[
g_\alpha(x,y) = \alpha g \alpha^{-1}(x,y)
    = \alpha \left(\frac{\two' x^g - \two'^g y^g}{\two' - \two'^g},\frac{x^g-y^g}{\two' - \two'^g}\right) = (y^g,x^g).
\]
We observe that $(y^g,x^g)$ commutes with $\beta$, so
$g_\beta(x,y) = \beta g_\alpha \beta^{-1}(x,y) = (y^g,x^g)$.
Finally
\[
g_\gamma(x,y) = \gamma g_\beta \gamma^{-1}(x,y) = \gamma g_\beta\Bigl(\frac{x}{y},y \Bigr)
= \gamma \Bigl(y^g,\frac{x^g}{y^g}\Bigr) = \Bigl(x^g,\frac{x^g}{y^g}\Bigr).\qedhere
\]
\end{proof}

\begin{lemma}\label{lem:actionh}
  Assume Set-Up~\textup{\ref{setup:22}} with $L\neq L'$. Then the induced Galois action $\eps\circ h\circ \eps^{-1}$ is given by
  \begin{align*}
 h_{\gamma}(x,y) & = \left(\frac{1}{x^{h}}, \frac1{y^{h}} \right).
\end{align*}
\end{lemma}
\begin{proof}
  We compute step by step the maps on Diagram \eqref{diag:gg'}:
  Since $\two'^{gh}=\two'^g$ and $\two'^h=\two'$ we have
  \begin{align*}
  h_\alpha(x,y) &= \alpha h \alpha^{-1}(x,y) =  (x^{h}, y^{h}).
  \end{align*}
On the other hand $\two^{gh}=\two$ and $\two^h=\two^g$, so we get
  \begin{multline*}
   h_\beta(x,y) = \beta h_\alpha \beta^{-1}(x,y)
  = \beta h_\alpha \left( \frac{\two^{g} x + \two}{x + 1}, \frac{\two y+\two^{g}}{y + 1}\right)\\
  = \beta \left( \frac{\two x^h+\two^g}{x^h+1}, \frac{\two^gy^h+\two}{y^h+1}\right)
  = \Bigl(\frac{1}{x^h},\frac{1}{y^h}\Bigr).
\end{multline*}
Finally since $\Bigl(\frac{1}{x^h},\frac{1}{y^h}\Bigr)$ commutes with $\gamma = (xy,y)$, we get
\[h_\gamma(x,y) = \gamma h_\beta \gamma^{-1} = \Bigl(\frac{1}{x^h}, \frac{1}{y^h}\Bigr).\qedhere\]
\end{proof}

\begin{lemma}\label{lem:22ConditionsGaloisInvarianceg}
  Assume Set-Up~\textup{\ref{setup:22}}.
  Let $M=\left(\begin{smallmatrix}
  a & b \\ c & d\\
  \end{smallmatrix}\right)\in\PGL_2(K),$
  and $M'=\left(\begin{smallmatrix}
  A & B \\ C & D\\
  \end{smallmatrix}\right)\in\PGL_2(K(x))$.
  Then $j = (M,M')\in \PGL_2(K)\ltimes \PGL_2(K(x))$ is $g_\gamma$-invariant if and only if $M=M^g$ in $\PGL_2(K)$, and the equality
  \begin{equation}
  \begin{pmatrix} A(x)x & B(x)\\ C(x)x & D(x) \end{pmatrix}
=\begin{pmatrix}
    D^g(x)(a^gx+b^g) & C^g(x)(a^gx+b^g)\\
    B^g(x)(c^gx+d^g) & A^g(x)(c^gx+d^g)
\end{pmatrix}\label{eq:gKx}
\end{equation}
   holds in $\PGL_2(K(x))$.
\end{lemma}

\begin{proof}
  Using the formula $g_{\gamma}(x,y) = (x^g,\frac{x^g}{y^g})$ from Lemma \ref{lem:actiong} we compute:
  \begin{align*}
    g_\gamma(j(x^g,y^g)) &= g_\gamma \left( \frac{ax^g+b}{cx^g+d} ,\frac{A(x^g)y^g+B(x^g)}{C(x^g)y^g+D(x^g)} \right)\\
    &= \left(\frac{a^gx+b^g}{c^g x+d^g} ,\frac{(a^gx+b^g)(C^g(x)y +D^g(x))}{(c^g x+d^g)(A^g(x)y+B^g(x))} \right), \\
    j(g_\gamma(x^g,y^g)) &= j\Bigl(x^g,\frac{x^g}{y^g}\Bigr)\\
    &= \left( \frac{ax+b}{cx+d} ,\frac{A(x)\frac{x}{y}+B(x)}{C(x)\frac{x}{y}+D(x)} \right)\\
    &= \left( \frac{ax+b}{cx+d} ,\frac{A(x){x}+B(x){y}}{C(x){x}+D(x){y}} \right).
  \end{align*}
  Therefore, $j = (M,M')$ is $g_\gamma$-invariant if and only if Equation~\eqref{eq:gKx} holds, as well as $ \smallmat{ a & b\\ c&d }
    = \smallmat{ a^g & b^g\\ c^g&d^g } $
  in $\PGL_2(K)$.
\end{proof}

\begin{remark}
  The condition that $M^g=M$ is equivalent to $M\in\PGL_3(\k)$ in the case $L=L'$, and to $M\in\PGL_3(L'')$, where $\k\subset L''\subset K$ is the fixed field of $g=hh'$ in the case of $L\neq L'$.
\end{remark}

\begin{lemma}\label{lem:22ConditionsGaloisInvarianceh}
  Assume Set-Up~\textup{\ref{setup:22}} with $L\neq L'$.
  Then $(M,M')=\left( \smallmat{a & b \\ c & d}, \smallmat{A & B \\ C & D} \right) \in \PGL_2(K)\ltimes \PGL_2(K(x))$ is $h_\gamma$-invariant if and only if the following equalities in $\PGL_2(K)$ respectively in $\PGL_2(K(x))$ are satisfied:

  \begin{align}
  \begin{pmatrix} a & b\\ c & d \end{pmatrix} &=
  \begin{pmatrix} d^{h} & c^{h}\\ b^{h} & a^{h} \end{pmatrix},\label{eq:hK}\\
  \begin{pmatrix} A(\frac1x) & B(\frac1x)\\ C(\frac1x) & D(\frac1x) \end{pmatrix} &=
  \begin{pmatrix} D^{h}(x) & C^{h}(x)\\ B^{h}(x) & A^{h}(x)\end{pmatrix}. \label{eq:hKx}
  \end{align}
\end{lemma}

\begin{proof}
  Using the formula $h_\gamma(x,y)=\left(\frac{1}{x^{h}}, \frac1{y^{h}} \right)$ from Lemma~\ref{lem:actionh} we study the action of $h_\gamma$:
  \begin{align*}
  h_{\gamma}(j(x^h,y^h)) & = \left( \frac{c^{h}x+d^{h}}{a^{h}x+b^{h}} , \frac{C^{h}(x)y+D^{h}(x)}{A^{h}(x)y+B^{h}(x)}\right),\\
  j(h_\gamma(x^h,y^h)) & = \left( \frac{a\frac{1}{x}+b}{c\frac{1}{x}+d} , \frac{A(\frac{1}{x})\frac{1}{y}+B(\frac{1}{x})}{C(\frac{1}{x})\frac{1}{y}+D(\frac{1}{x})}\right)\\
  & =  \left( \frac{b{x}+a}{d{x}+c} , \frac{B(\frac{1}{x})y+ A(\frac{1}{x})}{D(\frac{1}{x})y+ C(\frac{1}{x})} \right).
  \end{align*}
  Hence, $h_\gamma\circ (M,M')= (M,M')\circ h_\gamma$ if and only if the matrix equalities \eqref{eq:hK} and~\eqref{eq:hKx} hold in $\PGL_2(K)$ respectively $\PGL_2(K(x))$.
\end{proof}


Seeing $\PGL_2$ as a subvariety of $\p^{4}$, two matrices $\left(\begin{smallmatrix}a_1 & a_2\\ a_3 & a_4\end{smallmatrix}\right)$, $\left(\begin{smallmatrix}b_1 & b_2\\ b_3 & b_4\end{smallmatrix}\right)$ are equal in $\PGL_2$ if and only if $a_ib_j=a_jb_i$ for all $i,j=1,\ldots,4$.
If $a_i\neq0$ for a fixed $i$, then this is equivalent to $a_ib_j=a_jb_i$ for $j=1,\ldots,4$.

\begin{corollary}\label{cor:ExplicitEquationsAntidiag}
  Assume Set-Up~\textup{\ref{setup:22}}. Let $a \in K$ and $P(x) \in K(x)$.
  Then
\[\left(
  \mat{0 & 1\\ a & 0 }, \mat{ 0 & 1\\ P(x) & 0 }
  \right)\in\PGL_2(K)\ltimes\PGL_2(K(x))
  \]
  is $\eps\Gal(K/\k)\eps^{-1}$-invariant if and only if
  \begin{enumerate}
    \item\label{item:ExplicitEquationsAntidiag--g} $a^g=a$, $P(x)P^g(x)=a$,
    \item\label{item:ExplicitEquationsAntidiag--h} and also $aa^h=1$ and $P(1/x)P^h(x)=1$ if $L\neq L'$.
  \end{enumerate}
  Moreover, under these conditions also $\left(
  \smallmat{0 & 1\\ a & 0 }, \smallmat{ 0 & 1\\ P(\mu) & 0 }
  \right)$ is $\eps\Gal(K/\k)\eps^{-1}$-invariant for all $\mu\in \k^*$.
\end{corollary}

\begin{proof}
  Note that if $L\neq L'$, $g=hh'$ and $h$ generate $\Gal(K/\k)$.
  The equations come from Lemma~\ref{lem:22ConditionsGaloisInvarianceg} and Lemma~\ref{lem:22ConditionsGaloisInvarianceh}.

  For the second part, write $P(x)=A(x)/B(x)$ with two polynomials $A,B\in K[x]$ without common factor. Condition \ref{item:ExplicitEquationsAntidiag--g} implies that $B^g(x)$ is a multiple of $A(x)$. If $\mu\in \k$ would be a root of $B$, then it is also a root of $B^g$ (since $B^g(\mu)= B^g(\mu^g)= (B(\mu))^g=0$), hence of $A$, a contradiction to $A,B$ without common factors.
  Therefore, evaluating $P(x)$ at $\mu\in \k^*$ we get a constant function $P(\mu)$ satisfying again \ref{item:ExplicitEquationsAntidiag--g} and \ref{item:ExplicitEquationsAntidiag--h}.
\end{proof}

\begin{remark}
  Over $\k=\Q$ and $L=L'=\Q(i)$, taking $a=1$ and $P(x)=\frac{x+i}{x-i}$ is $\eps\Gal(K/\k)\eps^{-1}$-invariant.
  More generally, if $L=L'=\k(\theta)$ one can take $a=1$ and $P(x)=\frac{x+\theta}{x+\theta^g}$.

  In the case $L=\k(\theta) \neq \k(\theta')=L'$, one can also construct a non-constant $P\in K(x)$ with $a=1$:
  Consider $Q(x)=\frac{x+\theta}{x+\theta^g}$ and choose $P(x)=Q(x)Q(\frac{1}{x})$.
\end{remark}

\subsubsection{Proof of Proposition~\textup{\ref{prop:DescriptionOfJonq22}}}


\begin{lemma}\label{lem:22KSomeObservations}
	Assume Set-Up~\textup{\ref{setup:22}}.
  For $A,B,C,D\in K[x]$ with $AD-BC\neq0$ and $a,b,c,d\in K$ with $ad-bc\neq0$ then the following hold:
  \begin{enumerate}
    \item\label{item:22KSomeObservations--1} If $\left(\left(\begin{smallmatrix}
    a & b \\ c & d\\
  \end{smallmatrix}\right), \left(\begin{smallmatrix}
    A & B \\ C & D\\
  \end{smallmatrix}\right)\right)$ is ${\eps\Gal(K/\k)\eps^{-1}}$-invariant, then if $A\neq0$ also $\left(\left(\begin{smallmatrix}
    a & b \\ c & d\\
  \end{smallmatrix}\right), \left(\begin{smallmatrix}
    A & 0 \\ 0 & D\\
    \end{smallmatrix}\right)\right)$ is, and if $B\neq0$ then also
    $\left(\left(\begin{smallmatrix}
      a & b \\ c & d\\
    \end{smallmatrix}\right), \left(\begin{smallmatrix}
      0 & B \\ C & 0\\
      \end{smallmatrix}\right)\right)$ is.
    \item\label{item:22KSomeObservations--2} $\left(\left(\begin{smallmatrix}
    a & 0 \\ 0 & d\\
  \end{smallmatrix}\right), \left(\begin{smallmatrix}
    A & 0 \\ 0 & D\\
    \end{smallmatrix}\right)\right)$ is ${\eps\Gal(K/\k)\eps^{-1}}$-invariant if and only if $\left(\left(\begin{smallmatrix}
    0 & a \\ d & 0\\
  \end{smallmatrix}\right), \left(\begin{smallmatrix}
    0 & A \\ D & 0\\
    \end{smallmatrix}\right)\right)$ is.
\item\label{item:22KSomeObservations--3}
$\left(\left(\begin{smallmatrix}
a & 0 \\ 0 & d\\
\end{smallmatrix}\right), \left(\begin{smallmatrix}
0 & B \\ C & 0\\
\end{smallmatrix}\right)\right)$ is ${\eps\Gal(K/\k)\eps^{-1}}$-invariant if and only if $\left(\left(\begin{smallmatrix}
0 & d \\ a & 0\\
\end{smallmatrix}\right), \left(\begin{smallmatrix}
C & 0 \\ 0 & B\\
\end{smallmatrix}\right)\right)$ is.
Moreover, this is equivalent to $\left(\left(\begin{smallmatrix}
0 & d \\ a & 0\\
\end{smallmatrix}\right), \left(\begin{smallmatrix}
0 & xC(x) \\ B(x) & 0\\
\end{smallmatrix}\right)\right)$ being ${\eps\Gal(K/\k)\eps^{-1}}$-invariant.
\item\label{item:22KSomeObservations--4} $\left(\left(\begin{smallmatrix}
0 & 1 \\ 1 & 0\\
\end{smallmatrix}\right), \left(\begin{smallmatrix}
0 & 1 \\ 1 & 0\\
\end{smallmatrix}\right)\right)$ is ${\eps\Gal(K/\k)\eps^{-1}}$-invariant.
\item\label{item:22KSomeObservations--5} If $A=D=0$ or $B=C=0$, then either $a=d=0$ or $b=c=0$.
  \end{enumerate}
\end{lemma}

\begin{proof}
  We use the conditions of ${\eps\Gal(K/\k)\eps^{-1}}$-invariancy from Lemma~\ref{lem:22ConditionsGaloisInvarianceg} (for the action of $g_\gamma$) and Lemma~\ref{lem:22ConditionsGaloisInvarianceh} (for the action of $h_\gamma$ if $L\neq L'$).
  Note that \eqref{eq:gKx} implies that $A\neq0$ if and only if $D\neq 0$, and $B\neq0$ if and only if $C\neq0$.
  Assertions \ref{item:22KSomeObservations--1} and \ref{item:22KSomeObservations--4} are then immediate.
  For \ref{item:22KSomeObservations--2}, and similarly for \ref{item:22KSomeObservations--3}, one checks that the conditions given by \eqref{eq:gKx} on the two couples of matrices in the statement are the same (the other equations are clear).

  To prove \ref{item:22KSomeObservations--5}, we write $\delta(a)=0$ if $a=0$, and $\delta(a)=1$ if $a\neq0$; same for $b,c,d$.
  Since $\smallmat{A & B \\ C & D}$ is invertible, we cannot have $A = 0$ and $B = 0$ simultaneously.
  Assume that $A\neq0$, $B = C = 0$ (the other case $B\neq 0$, $A = D = 0$ is similar and left to the reader).
  Equation \eqref{eq:gKx} gives
  \begin{align*}
   A^g(x)xA^g(x)(c^gx+d^g) &= D(x)D^g(x)(a^gx+b^g).
  \end{align*}
Taking the degree in $x$, we get
\begin{align*}
  2\deg(A)+\delta(c)+1 &= 2\deg(D)+\delta(a)
\end{align*}
  Reducing this last equality modulo 2 gives two possibilities:
  Either $\delta(c)=0$ and $\delta(a)=1$,  hence $b=c=0$ and $a,d\neq 0$, or $\delta(a)=0$ and $\delta(c)=1$,  hence $a=d=0$ and $b,c\neq 0$.
\end{proof}

The previous lemma has the following corollary, which says that the diagonal and antidiagonal elements in $\left(\PGL_2(K)\ltimes \PGL_2(K(x))\right)^{\eps\Gal(K/\k)\eps^{-1}}$ surject onto $\Aut_\k^{\pi_{L,L'}}(\P^1)$, which is the image of the projection to \[\left(\PGL_2(K)\ltimes \PGL_2(K(x))\right)^{\eps\Gal(K/\k)\eps^{-1}}\to \PGL_2(K):\]

\begin{corollary}\label{c:AlwaysAntidiagOrDiag}
  Assume Set-Up~\textup{\ref{setup:22}}.
  Let $\left( \smallmat{a & b\\ c & d } , \smallmat{ A & B\\ C & D } \right) \in \PGL_2(K)\ltimes \PGL_2(K(x))$ be $\eps\Gal(K/\k)\eps^{-1}$-invariant.
  Then one of the following holds:
  \begin{enumerate}
    \item $b=c=0$ and there exist $A',D'\in K(x)$ such that the map
    $\left( \smallmat{ a & 0\\ 0 & d }, \smallmat{ A' & 0\\ 0 & D'} \right)$ is $\eps\Gal(K/\k)\eps^{-1}$-invariant.
    \item $a=d=0$ and there exist $B',C'\in K(x)$ such that the map
    $\left( \smallmat{ 0 & b\\ c & 0 }, \smallmat{ 0 & B'\\ C' & 0 }\right)$ is $\eps\Gal(K/\k)\eps^{-1}$-invariant.
  \end{enumerate}
\end{corollary}

\begin{proof}
  Up to multiplying by a common denominator we can assume that $A, B, C ,D$ are polynomials.
  If $A\neq 0$ we can assume $B=C=0$ by Lemma~\ref{lem:22KSomeObservations}\ref{item:22KSomeObservations--1}.
  By Lemma~\ref{lem:22KSomeObservations}\ref{item:22KSomeObservations--5}, either $b=c=0$ and we are in the diagonal case, or $a=d=0$.
  In the latter, Lemma~\ref{lem:22KSomeObservations}\ref{item:22KSomeObservations--3} implies that $\left( \smallmat{ 0 & b\\ c & 0 }, \smallmat{ 0 & xA\\ D & 0 } \right)$ is invariant and we are in the antidiagonal case.

  If $B\neq0$ we proceed similarly:
  We can assume $A=D=0$ and find that either $a=d=0$ and we are in the antidiagonal case, or $b=c=0$.
  In the latter, the matrices are of the form
  $\left( \smallmat{ a & 0\\ 0 & d }, \smallmat{ 0 &  B\\ C & 0 } \right)$.
  Lemma~\ref{lem:22KSomeObservations}\ref{item:22KSomeObservations--3} implies that the antidiagonal $\left( \smallmat{ 0 & d\\ a & 0 }, \smallmat{ 0 & xC \\ B & 0 } \right)$ is invariant, but so also $\left( \smallmat{
    d & 0\\ 0 & a } , \smallmat{ xC & 0\\ 0 & B} \right)$ is, by Lemma~\ref{lem:22KSomeObservations}\ref{item:22KSomeObservations--2}.
  Note that this corresponds to the birational map $f\colon(x,y)\mapsto (\frac{dx}{a}, \frac{xC(x)y}{B})$.
  Finally, conjugating this with the Galois-invariant $\left(\left(\begin{smallmatrix}
0 & 1 \\ 1 & 0\\
\end{smallmatrix}\right), \left(\begin{smallmatrix}
0 & 1 \\ 1 & 0\\
\end{smallmatrix}\right)\right)$ (Lemma~\ref{lem:22KSomeObservations}\ref{item:22KSomeObservations--4}), that is, the map $\iota\colon(x,y)\mapsto (1/x,1/y)$, we find
\begin{align*}
  \iota\circ f\circ \iota(x,y) = \iota(f(1/x,1/y))
  &=\iota \Bigl( \frac{d}{ax}, \frac{C(1/x)}{xB(1/x)y} \Bigr) = \Bigl(\frac{ax}d, \frac{xB(1/x)y}{C(1/x)} \Bigr).
\end{align*}
So $\left( \smallmat{ a & 0\\ 0 & d } , \smallmat{  B(1/x)x & 0\\ 0 & C(1/x)} \right)$ is $\eps\Gal(K/\k)\eps^{-1}$-invariant and we are in the diagonal case again.
\end{proof}

We are now ready for the proof of Proposition~\ref{prop:DescriptionOfJonq22}.

\begin{proof}[Proof of Proposition~\textup{\ref{prop:DescriptionOfJonq22}}]
  First of all, observe that maps of the form $f(x,y)=(\frac{1}{\mu x},\frac{1}{\lambda x})$ are indeed involutions, and that the conditions on $\lambda,\mu$ assert that $f$ is $g_\gamma$-invariant (Lemma~\ref{lem:actiong}), and also $h_\gamma$-invariant if $L\neq L'$ (Lemma~\ref{lem:actionh}).
  So indeed $\eps H_{L,L'}\eps^{-1}\subset \Bir_\k(\p^2,\pi_{L,L'})$.

  By Corollary \ref{cor:Jonq22BirVsPGL}, we have
    \[
      \Bir_\k(\p^2,\pi)\simeq \left(\PGL_2(K)\ltimes \PGL_2(K(x))\right)^{\eps\Gal(K/\k)\eps^{-1}}
    \]
    and $\Aut_\k^{\pi_{L,L'}}(\p^1)$ corresponds to the image of the projection onto $\PGL_2(K)$.
    Corollary~\ref{c:AlwaysAntidiagOrDiag} gives that diagonal and antidiagonal elements surject onto the image of the projection.
    The diagonal elements are generated by antidiagonal ones, namely
    \begin{align*}
      \left(\smallmat{ \mu & 0\\ 0 & 1 }, \smallmat{ P(x) & 0\\ 0 & 1 } \right) =
      \left(\smallmat{ 0 & 1\\ 1 & 0 }, \smallmat{ 0 & 1\\ 1 & 0 } \right)
      \circ \left( \smallmat{ 0 & 1\\ \mu & 0 }, \smallmat{ 0 & 1\\ P(x) & 0 } \right)
      \end{align*}
      is a composition of two $\eps\Gal(K/\k)\eps^{-1}$-invariant antidiagonal elements (Lemma~\ref{lem:22KSomeObservations}\ref{item:22KSomeObservations--2} and \ref{item:22KSomeObservations--4}).
    Corollary~\ref{cor:ExplicitEquationsAntidiag} describes the conditions on the entries of the antidiagonal matrices $\left(\smallmat{ 0 & 1\\ a & 0 }, \smallmat{ 0 & 1\\ P(x) & 0 } \right)$ to be $\eps\Gal(K/\k)\eps^{-1}$-invariant.
    In particular, it says that evaluating $P$ at any $a\in\k^*$ gives again a $\eps\Gal(K/\k)\eps^{-1}$-invariant antidiagonal matrix $\left(\smallmat{ 0 & 1\\ a & 0 }, \smallmat{ 0 & 1\\ P(a) & 0 } \right)$. Setting $P(a)=\lambda\in K$ gives the conditions $\mu^g=\mu$, $\lambda\lambda^g=\mu$, and if $L\neq L'$ also $\mu\mu^h=1=\lambda\lambda^h$. Note that $\mu=\lambda\lambda^g$ implies $\mu^g=\mu$, and so $\lambda\lambda^h=1$ implies $\mu\mu^h=1$.
    Hence $H_{L,L'}$ surjects onto the image of the projection to $\PGL_2(K)$ and the statement follows.
\end{proof}

%% file: main_problematic_cases.tex
\section{Irreducible elements of Del Pezzo type}
\label{sec:delPezzo}

In this last section we study irreducible elements of Del Pezzo type, whose possible factorization types in terms of Sarkisov links were given in Table \ref{tab:del Pezzo type}.
First we identify some ``easy cases'' where we can lower the Sarkisov length by composing by a quadratic, Geiser or Bertini involution.
Then for the remaining ``hard cases'' we produce factorizations into involutions by using elementary relations between Sarkisov links.
Most of these factorizations rely on an assumption about existence of points in general position, and in \S\ref{sec:existence} we show that these conditions are satisfied.
Finally in \S\ref{sec:proof} we put everything together and finish the proof of Theorem \ref{t:main}.

\subsection{Easy simplifications}

First, we discuss quadratic involutions, using the following result:

\begin{lemma}
\label{l:quadratic}
Let $f \in \Bir_{\alg \k}(\P^2)$ be a quadratic map with three proper base points $p_1, p_2, p_3$.
For each $\{i,j,k\} = \{1,2,3\}$, let $q_k = f(L_{ij})$ where $L_{ij} \subset \P^2$ is the line through $p_i, p_j$.
Then:
\begin{enumerate}
\item \label{quad:1}
There exists $\alpha \in \Aut_{\alg \k}(\P^2)$ that sends $q_i$ to $p_i$ for each $i = 1,2,3$.
\item \label{quad:2}
For any such automorphism $\alpha$, the quadratic map $\alpha \circ f$ is an involution.
\item \label{quad:3}
If the set $\{p_1, p_2, p_3\}$ is invariant under the Galois action, then in~\ref{quad:1} we can choose $\alpha$ in $\Aut_{\k}(\P^2)$.
\end{enumerate}
\end{lemma}

\begin{proof}
Assertion \ref{quad:1} is just the classical fact that $\PGL_3(\alg \k)$ acts transitively on triple of non colinear points. For \ref{quad:3}, see \cite[Lemma 6.11]{SchneiderRelations}.

To prove \ref{quad:2}, up to conjugation we can assume $p_1=[1:0:0]$, $p_2 = [0:1:0]$, $p_3 = [0:0:1]$.
Let $\pi\colon X \to \P^2$ be the blow-up of $p_1,p_2,p_3$.
Since by construction $\alpha \circ f$ has the same base points as its inverse, we have a commutative diagram
\[
\begin{tikzcd}
X \ar[d, "\pi",swap] \ar[r, "g"] & X  \ar[d, "\pi"]\\
\P^2 \ar[r,"\alpha \circ f", dashed] & \P^2
\end{tikzcd}
\]
with $g \in \Aut_\bk(X)$.
The automorphism group of the Del Pezzo surface $X$ of degree 6 is well known \cite[Theorem 8.4.2]{DolgachevClassical}:
\[
\Aut_\bk(X) = (\mathbb{G}_m)^2 \rtimes (S_3 \times \Z/2)
\]
where  $(\mathbb{G}_m)^2$ is the standard toric action $[ax:by:z]$ with $a,b\in {\bk}^*$, the symmetric group $S_3$ permutes the homogeneous coordinates and the generator of $\Z/2$ is the standard quadratic involution $\sigma = [yz:xz:xy]$.
In particular the action of $g$ on the six $(-1)$-curves of $X$ uniquely determines $g$ up to composition by an element of $(\mathbb{G}_m)^2$.
Finally since $p_k=\alpha\circ f(L_{ij})$, we observe that $g$ has the same action as $\sigma$, and that for any $(a,b) \in (\mathbb{G}_m)^2$ the map $[ayz:bxz:xy]$ is an involution, which ends the proof.
\end{proof}

\begin{corollary}
\label{c:quad_involution}
Let $f\colon \P^2 \rat \P^2$ be a map of minimal factorization type
\begin{enumerate}[\textup{(Q}$1$\textup{)}]
\item\label{Q33} $\P^2 \link33 \P^2$, or
\item\label{Q21} $\P^2 \link21 \Dl_8 \link12 \P^2$.
\end{enumerate}
Then there exists $\alpha \in \Aut_\k(\P^2)$ such that $\alpha \circ f$ is a quadratic involution.
\end{corollary}

\begin{proof}
In the case \ref{Q33}, we can directly apply Lemma \ref{l:quadratic}.
In the case \ref{Q21}, we take $p = \{p_1, p_2\}$ the $2$-point blown-up by the first link, and $p_3$ the rational point on $\P^2$ corresponding to the point blown up by the second link. Note that they are not collinear.
Again we conclude by Lemma \ref{l:quadratic}.
\end{proof}

Now we discuss Geiser and Bertini involutions, see \cite[8.7.2 and 8.8.2]{DolgachevClassical} for details.
Let $X$ be a Del Pezzo surface of degree $1$ (resp. of degree $2$).
Then the linear system $|-2K_X|$ (resp. $|-K_X|$) corresponds to a 2 to 1 morphism $X \to \P^2$, and the deck transformation $\sigma\colon X \to X$ is an automorphism of $X$ called a Bertini involution (resp. a Geiser involution).
The involution $\sigma$ acts on the Néron--Severi space $\Pic(X) \otimes \R$ preserving the class of $K_X$, and it acts as minus identity on $(K_X)^\perp$.

\begin{lemma}\label{l:bertini_geiser}
  Let $X$ be a Del Pezzo surface of degree $1$ \parent{resp. of degree $2$}, and let $\sigma \in \Aut(X)$ be the Bertini involution \parent{resp. the Geiser involution}.
  Then for any collection $\{C_i\}_{i\in I}$ of rational curves with $C_i^2\in \{-1,0\}$ and $C_i\cdot C_j=0$ for $i\neq j$, we have $\sigma(C_i) \neq C_j$ for all $i\in I$.
\end{lemma}

\begin{proof}
As mentioned above, the class of a curve $C \in \Pic(X)$ is fixed by $\sigma$ if and only if $C$ is a multiple of $-K_X$.
In particular, any curve that is fixed by $\sigma$ must have positive self-intersection.
So the assumption $C_i^2\leq 0$ implies directly that $\sigma(C_i)\neq C_i$ for $i=1,\ldots,r$.
Now assume by contradiction that $\sigma(C_i)=C_j$ for some indexes $i\neq j$.
Then $C_i+C_j$ is fixed by the involution $\sigma$, and since $C_i\cdot C_j=0$ by assumption, we get $(C_i+C_j)^2=C_i^2+C_j^2\leq0$, hence $C_i+C_j$ is not fixed by $\sigma$, a contradiction.
\end{proof}

\begin{corollary}
\label{c:central_symmetry}
Let $\Pl$ be the piece associated to a Del Pezzo surface of degree $1$ \parent{resp. of degree $2$}.
Then the Bertini involution \parent{resp. the Geiser involution} acts on the piece without fixing any proper face.
If the piece is $1$- or $2$-dimensional, this means that the involution acts on the piece as $-\id$ \parent{``central symmetry''}.
\end{corollary}

\begin{proof}
Let $X$ be a Del Pezzo surface of degree $1$, or $2$, of Picard rank $n\geq 2$.
So $X/\mathrm{pt}$ is a rank $n$ fibration giving rise to an $n-1$-dimensional piece.
Let $r\leq n$ and let $Y/B$ be a rank $r$ fibration that is dominated by $X$, giving rise to a $r-1$-dimensional face.
For $r=n$ there is nothing to prove.
Assume now $r<n$.
Consider all the rank $n-1$ fibrations $Z_j/B_j$ that dominate $Y$ and are dominated by $X$, with dominating maps $f_j\colon X\to Z_j$.
Consider the set $\{C_{j,i}\}_{i\in I_j}$ of curves on $X$ that are contracted by $f_j$ if $B_j=\pt$, or by $\pi_j\circ f_j$ where $\pi_j\colon Z_j\to B_j$ if $B_j=\p^1$.
So the set of $C_{j,i}$ is either a Galois orbit of pairwise disjoint $(-1)$-curves, or the transform of fibers of $Z_j\to B_j$.
Now take the union $\bigcup_j \{C_{j,i}\}_{i\in I_j}$. They satisfy $C_{ji}\cdot C_{j'i'}=0$ for $C_{ji}\neq C_{j'i'}$. So the curves $\{C_{ji}\}$ are as in Lemma~\ref{l:bertini_geiser}.
Hence, the Bertini respectively Geiser involution does not fix $Y$.
This proves that no proper face is fixed by the involution.

If the piece has dimension $n = 1$ or $2$, the combinatorial piece can be embedded into $\R^n$ as an interval or a regular polygon centered at the origin, so that any combinatorial bijection comes from the restriction of an Euclidean isometry of $\R^n$.
Since  $-\id$ is the only Euclidean involution preserving the polygon that does not fix any proper face, we get the result.
\end{proof}

We say that $\iota \in \Bir_\k(\P^2)$ is conjugate to a Bertini involution (resp. to a Geiser involution) if there exists a birational map $\phi \colon \P^2 \rat X$, where $X$ is a Del Pezzo surface of degree $1$ (resp. of degree $2$), such that
$\iota = \phi^{-1} \sigma \phi$ where $\sigma\colon X \to X$ is the Bertini involution (resp. the Geiser involution) on $X$.
The following result explains the meaning of ``Bertini simplification'' or ``Geiser simplification'' in the last column of Table \ref{tab:del Pezzo type}.

\begin{proposition}\label{prop:BertiniGeiserSimplification}
Let $f\in \Bir_\k(\P^2)$ be an irreducible map of minimal factorization type one of the following:
\begin{enumerate}[\textup{(G}$1$\textup{)}]
\item \label{G1}
$\P^2 \link77 \P^2$;
\item \label{G2}
$\P^2 \link21 \Dl_8 \link 66 \Dl_8 \link12 \P^2$;
\item \label{G3}
$\P^2 \link51 \Dl_5 \link 33 \Dl_5 \link15 \P^2$;
\item \label{G4}
$\P^2 \link21 \Dl_8 \link31 \Dl_6 \link 44 \Dl_6 \link13 \Dl_8 \link12 \P^2$;
\end{enumerate}
\begin{enumerate}[\textup{(B}$1$\textup{)}]
\item \label{B1} $\P^2 \link88 \P^2$;
\item \label{B2} $\P^2 \link21 \Dl_8 \link 77 \Dl_8 \link12 \P^2$;
\item \label{B3} $\P^2 \link51 \Dl_5 \link 44 \Dl_5 \link15 \P^2$;
\item \label{B4} $\P^2 \link21 \Dl_8 \link31 \Dl_6 \link 55 \Dl_6 \link13 \Dl_8 \link12 \P^2$.
\end{enumerate}
Then we can write $f = g  \circ \iota$, where $\iota$ is conjugate to a Geiser or Bertini involution, and $\sl(g) < \sl(f)$.
\end{proposition}

\begin{proof}
The diagram below explains the proof for the case \ref{G3}.
By the minimality of the factorization, we have $\sl(f)=3$.
\[
\begin{tikzcd}
& X_2 \ar[d,"3",swap] \ar[r,"\sigma"]
& X_2  \ar[d,"3"] \ar[r,dashed,"\phi^{-1}"]
& \P^2  \ar[dd,dashed,bend left,"\;g \text{ with } \sl(g) \le 2"] \\
\P^2 \ar[r,dashed,"5\;1"] \ar[ur,dashed,"\phi"]
\ar[urrr,dashed,bend left= 2cm,"\iota\; = \;\phi^{-1}\sigma\phi"]
 \ar[drrr,dashed,bend right,"f"]
& X_5  \ar[r,dashed,"3\;3"]
& X_5  \ar[ur,dashed,"1\;5",swap] \ar[dr,dashed,"1\;5"]\\
&&& \P^2
\end{tikzcd}
\]
Here $\sigma\colon X_2 \to X_2$ is the Geiser involution on the fibration of rank $2$ dominating the link $X_5 \link33 X_5$.
The cases \ref{G2}, \ref{G4}, \ref{B2}, \ref{B3}, \ref{B4} are similar.

For the case \ref{G1}, we have to adapt slightly the argument, since here $g$ will be an automorphism that we usually keep hidden in the definition of Sarkisov link ``up to isomorphism''. 
We use the diagram
\[
\begin{tikzcd}
X_2 \ar[d,"3",swap] \ar[r,"\sigma"]
& X_2  \ar[d,"3"]  \\
\P^2  \ar[dashed]{r}{3\;3}[swap]{\iota} \ar[rrr, bend right,dashed,"f"] 
& \P^2  \ar[rr,dashed,"g \in \Aut(\P^2)"] && \P^2
\end{tikzcd}
\]
where again $\sigma$ is the Geiser involution on the Del Pezzo surface $X_2$.
The argument for the case \ref{B1} is similar.
\end{proof}

\subsection{Remaining cases}
\label{sec:remaining}

For the remaining cases of Table \ref{tab:del Pezzo type}, we show that any map $f$ with factorization type in the following list is generated by involutions, provided that certain generality conditions are satisfied:

\begin{enumerate}[label=\rm({\roman*})]
  \item\label{item:linkD5} $\p^2\link{5}{1} \Dl_5\link{1}{5} \p^2$,
  \item\label{item:linkD6} $\p^2\link{2}{1} \Dl_8\link{3}{1} \Dl_6\link{1}{3}\Dl_8\link{1}{2}\p^2$,
  \item\label{item:linkD8-44} $\p^2\link{2}{1} \Dl_8\link{4}{4}\Dl_8\link{1}{2}\p^2$,
  \item\label{item:linkD6-33} $\p^2\link{2}{1} \Dl_8\link{3}{1} \Dl_6\link{3}{3} \Dl_6\link{1}{3}\Dl_8\link{1}{2}\p^2$,
  \item\label{item:linkD6-22} $\p^2\link{2}{1} \Dl_8\link{3}{1} \Dl_6\link{2}{2} \Dl_6\link{1}{3}\Dl_8\link{1}{2}\p^2$,
  \item\label{item:loop} $\p^2\link{5}{1} \Dl_5\link{2}{5}\Dl_8\link{1}{2}\p^2$,
  \item\label{item:linkP2-66} $\p^2\link{6}{6} \p^2$.
\end{enumerate}

If there is an $a$-point and a $b$-point on $X\in \Dl$ such that the two points are in general position, we denote by $\piece{X}ab$ the piece of the corresponding relation involving the two links starting at $X$ associated to the blow-up of the $a$-point, respectively the $b$-point.
(This is coherent with the notation used in Appendix \ref{app:relations}).

\begin{lemma}[Case~\ref{item:linkD5}]\label{lem:linkD5}
Let $f$ be a map of minimal factorization type $\p^2\link{5}{1} \Dl_5\link{1}{5} \p^2$. The following holds:
  \begin{enumerate}
    \item\label{item:linkD5--fibering} Assume that the two rational points on $X_5$ are general. Then the map $f$ is a composition of two maps of Jonqui\`eres type $1$ (and an automorphism of $\p^2$ at the end).
    \item\label{item:linkD5--delPezzo} Assume that there is a $2$-point on $X_5$ that is general with the two rational points.
    Then the map $f$ is a composition of two Geiser involutions from Del Pezzo surfaces with Picard rank $3$ and a quadratic map~\ref{Q21}.
  \end{enumerate}
\end{lemma}

\begin{proof}
Since the factorization is minimal, the two rational points on $X_5$ are distinct.
  The assumption in \ref{item:linkD5--fibering} implies that the map $f$ lies inside the piece  \hyperref[P2_15]{$\piece{X_5}11 = \piece{\p^2}15$}.
It can be glued with the piece \hyperref[P2_11]{$\piece{\p^2}11$} along the edge corresponding to the exchange of the rulings of $\F_0$.
This gives a decomposition of $f$ into two maps of Jonqui\`eres type $1$, up to an automorphism: see Figure \ref{fig:2pieces}.

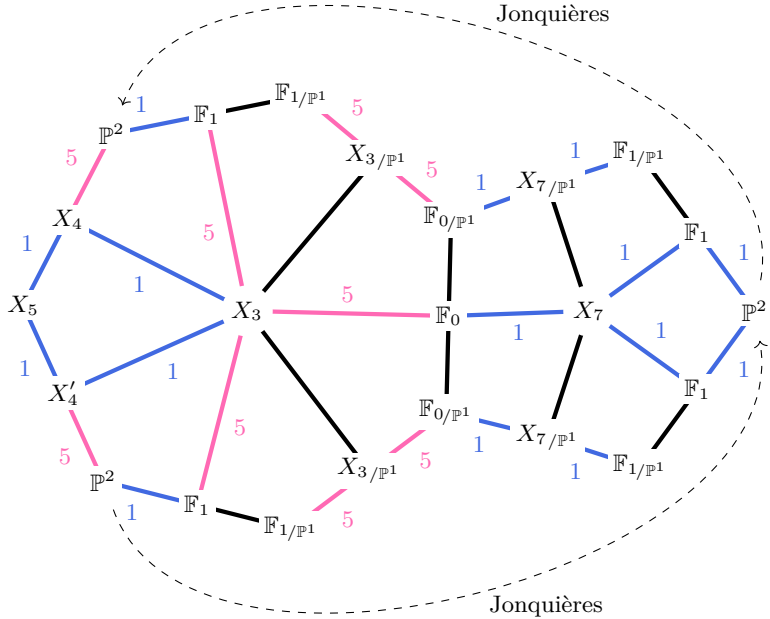
\begin{figure}
\begin{center}
\begin{tikzpicture}[font=\footnotesize]
\clip(-4,-4.2) rectangle (8,4.5);
\node[name=s, regular polygon, rotate=140, regular polygon sides=7,inner sep=1.872cm] at (0,0) {};
\draw[-,ultra thick,HotPink] (s.center) to["5"] (s.side 1);
\draw[-,ultra thick,black] (s.center) to[""] (s.side 2);
\draw[-,ultra thick,HotPink] (s.center) to["5"] (s.side 3);
\draw[-,ultra thick,black] (s.center) to[""] (s.side 4);
\draw[-,ultra thick,HotPink] (s.center) to["5"] (s.side 5);
\draw[-,ultra thick,RoyalBlue] (s.center) to["1"] (s.side 6);
\draw[-,ultra thick,RoyalBlue] (s.center) to["1"] (s.side 7);
\draw[-,ultra thick,black] (s.side 1) to[swap,""] (s.corner 2);
\draw[-,ultra thick,RoyalBlue] (s.side 1) to["1"] (s.corner 1);
\draw[-,ultra thick,HotPink] (s.side 2) to[swap,"5"] (s.corner 3);
\draw[-,ultra thick,HotPink] (s.side 2) to["5"] (s.corner 2);
\draw[-,ultra thick,black] (s.side 3) to[swap,""] (s.corner 4);
\draw[-,ultra thick,black] (s.side 3) to[""] (s.corner 3);
\draw[-,ultra thick,HotPink] (s.side 4) to[swap,"5"] (s.corner 5);
\draw[-,ultra thick,HotPink] (s.side 4) to["5"] (s.corner 4);
\draw[-,ultra thick,RoyalBlue] (s.side 5) to[swap,"1"] (s.corner 6);
\draw[-,ultra thick,black] (s.side 5) to[""] (s.corner 5);
\draw[-,ultra thick,RoyalBlue] (s.side 6) to[swap,"1"] (s.corner 7);
\draw[-,ultra thick,HotPink] (s.side 6) to["5"] (s.corner 6);
\draw[-,ultra thick,HotPink] (s.side 7) to[swap,"5"] (s.corner 1);
\draw[-,ultra thick,RoyalBlue] (s.side 7) to["1"] (s.corner 7);

\node[name=t, regular polygon, rotate=-90, regular polygon sides=5,inner sep=1.242cm] at (4.5,0) {};
\draw[-,ultra thick,RoyalBlue] (t.center) to["1"] (t.side 1);
\draw[-,ultra thick,black] (t.center) to[""] (t.side 2);
\draw[-,ultra thick,RoyalBlue] (t.center) to["1"] (s.side 3);
\draw[-,ultra thick,black] (t.center) to[""] (t.side 4);
\draw[-,ultra thick,RoyalBlue] (t.center) to["1"] (t.side 5);
\draw[-,ultra thick,black] (t.side 1) to[swap,""] (t.corner 2);
\draw[-,ultra thick,RoyalBlue] (t.side 1) to["1"] (t.corner 1);
\draw[-,ultra thick,RoyalBlue] (t.side 2) to[swap,"1"] (s.corner 4);
\draw[-,ultra thick,RoyalBlue] (t.side 2) to["1"] (t.corner 2);
\draw[-,ultra thick,RoyalBlue] (t.side 4) to[swap,"1"] (t.corner 5);
\draw[-,ultra thick,RoyalBlue] (t.side 4) to["1"] (s.corner 3);
\draw[-,ultra thick,RoyalBlue] (t.side 5) to[swap,"1"] (t.corner 1);
\draw[-,ultra thick,black] (t.side 5) to[""] (t.corner 5);

\node[fill=white,circle,inner sep=2] at (s.center) {${X_3}$};
\node[fill=white,inner sep=2] at (s.corner 1) {${\p^2}$};
\node[fill=white,inner sep=2] at (s.side 1) {${\F_1}$};
\node[fill=white,inner sep=2] at (s.corner 2) {${\F_1}_{\!/\p^1}$};
\node[fill=white,inner sep=2] at (s.side 2) {${X_3}_{/\p^1}$};
\node[fill=white,inner sep=2] at (s.corner 3) {${\F_0}_{\!/\p^1}$};
\node[fill=white,inner sep=2] at (s.side 3) {${\F_0}$};
\node[fill=white,inner sep=2] at (s.corner 4) {${\F_0}_{\!/\p^1}$};
\node[fill=white,inner sep=2] at (s.side 4) {${X_3}_{/\p^1}$};
\node[fill=white,inner sep=2] at (s.corner 5) {${\F_1}_{\!/\p^1}$};
\node[fill=white,inner sep=2] at (s.side 5) {${\F_1}$};
\node[fill=white,inner sep=2] at (s.corner 6) {${\p^2}$};
\node[fill=white,inner sep=2] at (s.side 6) {${X_4}$};
\node[fill=white,inner sep=2] at (s.corner 7) {${X_5}$};
\node[fill=white,inner sep=2] at (s.side 7) {${X_4'}$};

\node[fill=white,circle,inner sep=2] at (t.center) {${X_7}$};
\node[fill=white,inner sep=2] at (t.corner 1) {${\p^2}$};
\node[fill=white,inner sep=2] at (t.side 1) {${\F_1}$};
\node[fill=white,inner sep=2] at (t.corner 2) {${\F_1}_{\!/\p^1}$};
\node[fill=white,inner sep=2] at (t.side 2) {${X_7}_{/\p^1}$};
\node[fill=white,inner sep=2] at (t.side 4) {${X_7}_{/\p^1}$};
\node[fill=white,inner sep=2] at (t.corner 5) {${\F_1}_{\!/\p^1}$};
\node[fill=white,inner sep=2] at (t.side 5) {${\F_1}$};

\draw[dashed, bend right=3cm,->,shorten <= 4mm, shorten >= 4mm] (s.corner 1) to["Jonquières",swap] (t.corner 1);
\draw[dashed, bend left=3cm,<-, shorten <= 4mm, shorten >= 4mm] (s.corner 6) to["Jonquières"] (t.corner 1);
\end{tikzpicture}
\end{center}
\caption{The two pieces of Lemma \ref{lem:linkD5}\ref{item:linkD5--fibering}}
\label{fig:2pieces}
\end{figure}

  By the assumption in \ref{item:linkD5--delPezzo} the map $f$ fits into two pieces \hyperref[P2_25]{$\piece{X_5}12=\piece{\p^2}25$} glued along the edge $\Dl_5\link{2}{5}\Dl_8$ that meets the factorization of $f$.
By Corollary \ref{c:central_symmetry}, the central symmetry of \hyperref[P2_25]{$\piece{\p^2}25$} corresponds to a Geiser involution. Hence, the map $f$ is given by the Geiser involution of the first piece, followed by a quadratic map \ref{Q21} followed by the Geiser involution of the second piece.
\end{proof}

\begin{lemma}[Case~\ref{item:linkD6}]\label{lem:linkD6}
Let $f$ be a map of minimal factorization type $\p^2\link{2}{1} \Dl_8\link{3}{1} \Dl_6\link{1}{3}\Dl_8\link{1}{2}\p^2$.
Assume that the two rational points on $X_6$ are general. Then $f$ is the composition of two quadratic maps~\ref{Q21} and one quadratic map~\ref{Q33}.
\end{lemma}

\begin{proof}
By minimality of the factorization, the two rational points on $X_6$ are distinct, hence general by the assumption. This implies that the middle part of $f$, namely $\Dl_8\link{3}{1} \Dl_6\link{1}{3}\Dl_8$, is inside the piece \hyperref[P2_23]{$\piece{X_6}11=\piece{\p^2}23$}.
  Hence, the map $f$ is given by a quadratic map \ref{Q21}, followed by a quadratic map \ref{Q33}, followed by another quadratic map \ref{Q21}.
\end{proof}

\begin{lemma}[Case~\ref{item:linkD8-44}]\label{lem:linkD8-44}
Let $f$ be a map of minimal factorization type $\p^2\link{2}{1} \Dl_8\link{4}{4}\Dl_8\link{1}{2}\p^2$.
  \begin{enumerate}
    \item\label{item:linkD8-44--fibering} Assume that there exists a rational point on $X_8$ that is general with the $4$-point.
    Then the map $f$ is the product of (at most) two quadratic maps~\ref{Q21} and a map of fibering type in $\Bir_\k(\p^2,\pi_{F})$, where $F/\k$ is the residue field of the $4$-point.
    \item\label{item:linkD8-44--fibering22} Assume that there is a $2$-point that is general with the $4$-point on $X_8$. Then the map $f$ is the product of a Geiser involution from a Del Pezzo surface of Picard rank $3$ and a map of Jonqui\`eres type $2+2$ in $\Bir_\k(\p^2,\pi_{L,L'})$ (and an automorphism of $\p^2$), where $L/\k$ is the residue field of the $2$-point on $\p^2$ and $L'/\k$ is the residue field of the $2$-point on $X_8$.
    \item\label{item:linkD8-44--delPezzo} Assume that there exists a $3$-point on $X_8$ that is general with the $4$-point.
    Then the map $f$ is the product of a Bertini involution from a Del Pezzo surface of Picard rank $3$ and a Geiser simplification \ref{G4}.
  \end{enumerate}
\end{lemma}

\begin{proof}
  For~\ref{item:linkD8-44--fibering} observe that the assumption implies that the middle part of $f$, namely $\Dl_8\link{4}{4}\Dl_8$ fits into a piece $\hyperref[P2_24]{\piece{X_8}14=\piece{\p^2}24}$. Two opposite vertices of this piece are $\p^2$, and the map between these two $\p^2$ is of fibering type 4 with residue field $F$.
  The beginning and the end of $f$, that is $\p^2\link{2}{1} \Dl_8$ and $\Dl_8\link{1}{2}\p^2$, are not necessarily contained in the piece. However, composing it with a quadratic map \ref{Q21} at the beginning and/or the end gives the statement.

  By the assumption in~\ref{item:linkD8-44--fibering22} the middle part of $f$ is part of a piece $\hyperref[D8_24]{\piece{X_8}24}$, whose central symmetry is a Geiser involution (Corollary \ref{c:central_symmetry}).
  So $f$ can be written as a Geiser involution followed by a map $\p^2\link{2}{1}X_8\linkI{2}\Dl_6\link{4}{4}\linkIII{2}X_8\link{1}{2}\p^2$, which is a map of Jonqui\`eres type $2+2$ in $\Bir_\k(\p^2,\pi_{L,L'})$ (after composing with an automorphism of $\p^2$).

  By the assumption in~\ref{item:linkD8-44--delPezzo} the middle part of $f$ is part of a piece $\hyperref[D8_34]{\piece{X_8}34}$. 
Observe that by Corollary \ref{c:central_symmetry}, its central symmetry is a Bertini involution.
  So $f$ can be decomposed into this Bertini involution and a Geiser simplification \ref{G4}.
\end{proof}

\begin{lemma}[Case~\ref{item:linkD6-33}]\label{lem:linkD6-33}
Let $f$ be a map of minimal factorization type
\[\p^2\link{2}{1} \Dl_8\link{3}{1} \Dl_6\link{3}{3} \Dl_6\link{1}{3}\Dl_8\link{1}{2}\p^2.\]
  Assume that there exists a rational point on $X_6$ that is general with the $3$-point on $X_6$.
  Then the map $f$ is the composition of two maps of the form~\ref{item:linkD6} and a Geiser involution.
\end{lemma}

\begin{proof}
  The assumption implies that the middle part of $f$, namely $\Dl_6\link{3}{3} \Dl_6$, is part of a piece $\hyperref[D6_13]{\piece{X_6}13}$. 
By Corollary \ref{c:central_symmetry} the central symmetry of the piece is a Geiser involution.
  Choosing a rational point on a $\Dl_8$ in the piece gives an edge $\Dl_8\link{1}{2}\p^2$ going out of the piece, and the central Geiser involution gives an opposite edge to $\p^2$.
  Therefore, the map $f$ is the product of a Geiser involution and maps of the form~\ref{item:linkD6} at the beginning and/or at the end.
\end{proof}

\begin{lemma}[Case~\ref{item:linkD6-22}]\label{lem:linkD6-22}
Let $f$ be a map of minimal factorization type
\[\p^2\link{2}{1} \Dl_8\link{3}{1} \Dl_6\link{2}{2} \Dl_6\link{1}{3}\Dl_8\link{1}{2}\p^2.\]
  \begin{enumerate}
    \item \label{item:linkD6-22--fibering} Assume that there is a rational point on $X_6$ that is general with the $2$-point on $X_6$. Then the map $f$ is the composition of (at most) two maps of the form \ref{item:linkD6} and a map of fibering type, more precisely a map in $\Bir_\k(\p^2,\pi_{L,L'})$ where $L/\k$ is the residue field of the $2$-point in $X_6$
    and $L'/\k$ is the degree $2$ extension coming from the link $\p^2\rat X_8$.
    \item\label{item:linkD6-22--delPezzo} Assume that there is a $3$-point on $X_6$ that is general with the $2$-point on $X_6$. Then the map $f$ is the composition of a Bertini involution of a Del Pezzo surface of Picard rank $3$ and (if needed) a map of the form~\ref{item:linkD6-33}.
  \end{enumerate}
\end{lemma}

\begin{proof}
  The assumption in~\ref{item:linkD6-22--fibering} implies that the middle part of $f$, namely $\Dl_6\link{2}{2} \Dl_6$, is part of a piece $\hyperref[D8_23]{\piece{X_6}12}=\piece{X_8}23$. We choose a rational point on one of the $\Dl_8$ in the piece, obtaining an edge to $\p^2$ going out of the piece. On the opposite side of the piece we can choose an edge to $\p^2$ such that the induced map of the piece between the opposite $\p^2$'s is of fibering type.
  Hence, by composing this map if needed with a map of the form \ref{item:linkD6} at the beginning and/or the end, statement~\ref{item:linkD6-22--fibering} follows.

  The assumption in~\ref{item:linkD6-22--delPezzo} implies that the middle part of $f$ is part of a piece $\hyperref[D6_23]{\piece{X_6}23}$. 
By Corollary \ref{c:central_symmetry} the central symmetry of the piece is a Bertini involution. Hence, the map $f$ is the composition of the central Bertini involution, followed by a map of the form~\ref{item:linkD6-33}.
\end{proof}

\begin{lemma}[Case~\ref{item:loop}]\label{lem:loop}
Let $f$ be a map of factorization type $\p^2\link{5}{1} \Dl_5\link{2}{5}\Dl_8\link{1}{2}\p^2$.
  Assume that there is a rational point on $X_5$ that is general with the $2$-point on $X_5$.
  Then $f$ is the product of a Geiser involution of a Del Pezzo surface of Picard rank $3$ and (if needed) maps of the form~\ref{Q21} and~\ref{item:linkD5}.
\end{lemma}

\begin{proof}
  The assumption implies that the middle part of $f$, namely $\Dl_5\link{2}{5}\Dl_8$, is part of a piece $\hyperref[P2_25]{\piece{X_5}12=\piece{\p^2}25}$. 
By Corollary \ref{c:central_symmetry} the central symmetry of the piece is a Geiser involution. Hence, $f$ can be written as the product of a Geiser involution and, if needed, a map of the form~\ref{item:linkD5} at the beginning and a quadratic map \ref{Q33} at the end.
\end{proof}

\begin{lemma}[Case~\ref{item:linkP2-66}]\label{lem:linkP2-66}
  Let $f$ be a link $\p^2\link{6}{6}\p^2$.
\begin{enumerate}
\item\label{item:linkP2-66--fibering} Assume that there exists a rational point on $\p^2$ that is general with the $6$-point on $\p^2$. Then $f$ is the composition of a Geiser involution, a map of Jonqui\`eres type~$1$ and an automorphism of $\p^2$.
\item\label{item:linkP2-66--delPezzo} Assume that there exists a $2$-point on $\p^2$ that is general with the $6$-point on $\p^2$.
  Then, the map $f$ is the composition of a Bertini involution of a Del Pezzo surface of Picard rank $3$ and a Geiser simplification \ref{G2}.
\end{enumerate}
\end{lemma}

\begin{proof}
The assumption in \ref{item:linkP2-66--fibering} implies that $f$ is part of a piece $\hyperref[P2_16]{\piece{\p^2}16}$, whose central symmetry is a Geiser involution.
Hence, $f$ can be decomposed into a Geiser involution and a map of Jonqui\`eres type 1 (and possibly an automorphism of $\p^2$).

  The assumption in \ref{item:linkP2-66--delPezzo} implies that $f$ is part of a piece $\hyperref[P2_26]{\piece{\p^2}26}$. Note that the central symmetry is a Bertini involution.
  Hence, doing first the central Bertini involution and then the Geiser simplification \ref{G2} along the edge $\Dl_8\link{6}{6}\Dl_8$ gives the statement.
\end{proof}

\subsection{General position}

\label{sec:existence}

In this section we show that the assumptions of general position in the Lemmas in Section~\ref{sec:remaining} are satisfied, at least for the first part of the respective lemmas.
As always the ground field $\k$ is assumed to be an arbitrary perfect field.

%

\begin{lemma}[Cases~\ref{item:linkD5} and~\ref{item:linkD6}]\label{lem:ExistenceXgeneral-11}
Let $X\in\Dl_5\cup\Dl_6$. Then any two rational points on $X$ are general.
\end{lemma}

\begin{proof}
See the appendix, Lemmas \ref{l:D5-11} and \ref{l:D6-11}.
\end{proof}

\begin{lemma}[Case~\ref{item:linkD8-44}, Lemma \ref{lem:linkD8-44}\ref{item:linkD8-44--fibering}]
\label{lem:ExistenceX8general-14}
Let $p\in X_8$ be a general $4$-point, where $X_8\in\Dl_8$. Then, any rational point $q\in X_8$ is general with $p$.
\end{lemma}
\begin{proof}
  Recall that five points on $\p^1\times\p^1$ are general if and only if no two lie on the same vertical or horizontal curve, and no four lie on a diagonal (see for example \cite[Lemma 4.3]{SchneiderF2}).
  Denote by $s$ the horizontal curve through $q$, and $f$ the vertical curve through $q$. Observe that $s$ and $f$ form a Galois orbit. If $p_1$ would lie on $s\cup f$, then all four $p_i$ would lie on $s\cup f$, hence (at least) two of the $p_i$ lie on the same vertical or horizontal line, contradicting the generality of $p$.

  If $q$ would lie on the diagonal through three of the $p_i$, then $q$ lies on all the diagonals through three of the four $p_i$. Since the diagonal through three points is unique, this means that all these diagonals are equal and so that $p$ lies on a diagonal, again contradicting the generality of $p$.
\end{proof}

\begin{lemma}[Case~\ref{item:linkD6-33}, Lemma \ref{lem:linkD6-33}]
\label{lem:ExistenceX6general-13}
Assume the ground field $\k$ is not the field with two elements, and let $X_6\in \Dl_6$. For any general $3$-point $p\in X_6$, there exists a rational point $q$ on $X_6$ that is general with $p$.
\end{lemma}

\begin{proof}
Consider $X_3\to X_6$ the blow-up of $X_6$ at $p$.
It is enough to see that there is a rational point that does not lie on any of the $27$ $(-1)$-curves of the cubic surface $X_3$.
Since each $(-1)$-curve contains at most one rational point (if it would contain two it is defined over $\k$ since two lines intersect at most once), this is clear if $\k$ is infinite.

Assume now that $\k$ is finite.
We have (see Remark~\ref{r:at_least_3})
\[|X_3(\k)| = |X_6(\k)| = |\k|^2-|\k|+1.\]
In the following we give a bound on the number of rational points that lie on the $27$ $(-1)$-curves of $X_3$ and see that this number is smaller than $|X_3(\k)|$.
Assume that $p\in X_3(\k)$ lies on a $(-1)$-curve that is defined over $\bk$, and let $E_1,\ldots, E_d$ be its Galois orbit.
Since $p$ is a rational point, it lies on each $E_i$.
The $27$ $(-1)$-curves on $X_3$ are such that at most three lines intersect in one point. Hence, $d\leq 3$.
  Since the link given by $p$ is of the form $X_6\link{3}{3}X_6'$ the only Galois orbits of pairwise disjoint $(-1)$-curves on $X_3$ are two Galois orbits of size $3$.
In particular, there is no $(-1)$-curve defined over $\k$ and so $d\geq2$.
  The exceptional divisors of the two points of degree $3$ that come from the link give six $(-1)$-curves that do not contain a rational point.
  The hexagon of the six $(-1)$-curves on $X_6$ is one Galois orbit (in \cite[Figure 1, Section 4.1]{SchneiderZimmermann} all configurations of the hexagon give rise to a birational morphism except when the hexagon forms one Galois orbit).
Similarly, the hexagon of the six $(-1)$-curves on $X_6'$ form also one Galois orbit, and its strict transform on $X_3$ is different from the hexagon on $X_6$.

  Hence, at most $27-18=9$ $(-1)$-curves contain a rational point.
  Since each $(-1)$-curve contains at most one rational point and each rational point lies on at least two $(-1)$-curves, the the total number of rational points lying on the set of $(-1)$-curves is at most $\frac{9}{2}$, so it is at most $4$.
  We find
  \[
    |\k|^2-|\k|+1 -4\geq 3.\qedhere
  \]
\end{proof}

\begin{remark}\label{rem:NotForF2}
  Lemma~\ref{lem:ExistenceX6general-13} does not hold for $\k=\F_2$: The blow-up of the two general $3$-points on $X_6$ found in \cite[Lemma 4.33, (2)]{SchneiderF2} gives a cubic surface such that all three rational points of $X_6$ lie on $(-1)$-curves. Hence Lemma~\ref{lem:linkD6-33} is not applicable.
  Moreover, there is also no $2$-point that is general with the $3$-point, and so also Lemma~\ref{lem:linkD6-22}\ref{item:linkD6-22--delPezzo} is not applicable, which would have allowed to decompose a link of the form $\Dl_6\link33\Dl_6$ into a Bertini involution and a link $\Dl_6\link22\Dl_6$.
  Therefore, this approach fails to give a decomposition of maps of the form \ref{item:linkD6-33} into involutions.
  Luckily, it was already proved in \cite[Corollary 1.3]{SchneiderF2} that over $\FF_2$  these maps are in fact involutions.
\end{remark}

\begin{lemma}\label{lem:X8-2And3General}
Let $p\in X_8$ be a general $2$-point, and let $q\in X_8$ be a general $3$-point.
Assume that the diagonal through $q$ does not contain any of the components of $p$.
Then $q$ is general with $p$.
\end{lemma}

\begin{proof}
  The Galois orbit of the horizontal and vertical curves containing $p$ has four components. Therefore, none of the vertical or horizontal curves through one of the $p_i$ contains any of the $q_i$, since the latter is a component of a $3$-point.

  Assume that two points of $q$, say $q_2,q_3$, lie on one diagonal with $p$, call it $D$.
  Note that since $\Gal(L/\k)$ is a transitive subgroup of $S_3$, it contains the cyclic group, where $L/\k$ is the splitting field of $q$.
  Let $\sigma$ be the generator of the cyclic group in $\Gal(L/\k)$. So $\sigma(D)$ contains $p$ as well as $q_1,q_3$, and therefore $D=\sigma(D)$ is the diagonal through $q$, a contradiction.
\end{proof}

\begin{lemma}[Case~\ref{item:linkD6-22}, Lemma \ref{lem:linkD6-22}\ref{item:linkD6-22--fibering}]
\label{lem:ExistenceX6general-12}
  Let $p$ be a general $2$-point on $X_6\in\Dl_6$. Then any rational point on $X_6$ is general with $p$.
\end{lemma}

\begin{proof}
  Let $q\in X_6$ be a rational point, and denote the image of $p$ under the link $X_6\link{1}{3}X_8$ again by $p$. Write $q=\{q_1,q_2,q_3\}\subset X_8$ for the contracted $3$-point. In particular, $q$ as well as $p$ are general points in $X_8$. Note that the diagonal through $q$ is the exceptional divisor of the link.
  Since the diagonal through the three geometric points of $q$ is contracted by the link, it does not contain any of the $p_i$.
  Lemma~\ref{lem:X8-2And3General} implies the statement.
\end{proof}

\begin{lemma}[Case~\ref{item:loop}, Lemma \ref{lem:loop}]
\label{lem:ExistenceX5general-12}
Let $p\in X_5$ be a general $2$-point, where $X_5\in\Dl_5$.
Then any rational point $q\in X_5$ is general with $p$.
\end{lemma}

\begin{proof}
  Consider the link $X_5\link{1}{5}\p^2$ given by the rational point $q$, denote by $r$ the $5$-point in $\p^2$ (this is a general point), and denote the image of $p$ in $\p^2$ again by $p$.
  Note that $p$ on $X_5$ is general with $q$ if and only if $p$ on $\p^2$ is general with $r$, because the surface obtained by blowing up $p$ and $q$ on $X_5$ is isomorphic to the one obtained by blowing up $p$ and $r$ on $\p^2$.
  Note that since the exceptional divisor of $q$ is sent onto the conic through the five geometric components of $r$, $p$ does not lie on this conic.
  Lemma~\ref{lem:P2-5And1Or2} implies the statement.
\end{proof}

\begin{lemma}[Case~\ref{item:linkP2-66}, Lemma \ref{lem:linkP2-66}\ref{item:linkP2-66--fibering}]\label{lem:P2-6WithRationalPoint}
Let $p\in\p^2$ be a general $6$-point.
Then, all rational points except finitely many are general with $p$.
Moreover, if the ground field $\k$ is finite, then all rational points except at most one are general with $p$.
\end{lemma}

\begin{proof}
  Let $r$ be a rational point. We first show that it does not lie on any conic through five components of $p$.
  Assume that $r$ lies on the conic through all components of $p$ except $p_i$, and call it $C_i$.
  Let $\sigma\in\Gal(\bk/\k)$ such that $\sigma(p_i)=p_j\neq p_i$. So $\sigma(C_i)=C_j$ contains $q$ and all components of $p$ except $p_j$. So $C_i$ and $C_j$ have the five points $q$ and $p\setminus\{p_i,p_j\}$ in common and are therefore equal. This is a contradiction to the generality of $p$.

  The lines $L_{ij}$ through two components $p_i$, $p_j$ of $p$ contain at most one rational point (if a line would contain two then it is defined over $\k$, not possible).
  Therefore, there are only finitely many rational points that are collinear with two components of $p$. (In fact, $\binom{6}{2}=15$ gives an upper bound.)

  Assume now that $\k$ is finite. We order the components of $p$ cyclically and denote by $\sigma$ the generator of $\Gal(L/\k)$ where $L$ is the extension of $\k$ of degree $6$.
  The lines $L_{ij}$ form three orbits, namely the ones of $L_{12}$ and $L_{13}$ are both of size $6$, and the one of $L_{14}$ is of size $3$.
  If $q\in L_{12}$ then $q\in\sigma(L_{12})=L_{23}$, so $L_{12}=L_{23}$ and so $p_1$, $p_2$ and $p_3$ are collinear, a contradiction to the generality of $p$.
  If $q\in L_{13}$ then $q\in\sigma^2(L_{13})=L_{35}$ and so $p_1$, $p_3$ and $p_5$ are collinear, again a contradiction.
  If $q\in L_{14}$, then $q\in L_{25}\cap L_{36}$, and so $q$ is the unique intersection point of these three lines.
\end{proof}

\subsection{Proof of Theorem \ref{t:main}}
\label{sec:proof}

We start from Proposition \ref{prop:CremonaGeneratedByDelPezzoAndFiberTypes}, and we want to show that each of the generators are contained in the subgroup of $\Bir_\k(\P^2)$ generated by involutions.

This is Lemma \ref{l:PGL3_and_jonq} for $\Aut_\k(\P^2)$ and $\Bir_\k(\P^2, \pi_\times)$, and Corollary \ref{c:Jonq_type_4} for a Jonquières group $\Bir_\k(\p^2,\pi_F)$ of type $4$.

For a Jonquières group $\Bir_\k(\p^2,\pi_{L,L'})$ of type $2+2$,
this follows from Propositions \ref{prop:KernelGeneratedByInvo} and \ref{prop:DescriptionOfJonq22}.

We are left with the case of an irreducible elements $f$ of Del Pezzo type with $1 \le \sl(f) \le 5$.
By contradiction, assume there exists such an element that is not a composition of involutions, and assume that $\sl(f)$ is minimal for this property.

Then the factorization type of a minimal factorization of $f$ is not one of the 10 cases covered by Corollary \ref{c:quad_involution} and Proposition \ref{prop:BertiniGeiserSimplification}, and we are left with the seven cases studied in sections \ref{sec:remaining} and \ref{sec:existence}:

\ref{item:linkD5} $\p^2\link{5}{1} \Dl_5\link{1}{5} \p^2$, see Lemmas~\ref{lem:linkD5}\ref{item:linkD5--fibering} and \ref{lem:ExistenceXgeneral-11}.

\ref{item:linkD6} $\p^2\link{2}{1} \Dl_8\link{3}{1} \Dl_6\link{1}{3}\Dl_8\link{1}{2}\p^2$, see Lemmas \ref{lem:linkD6} and \ref{lem:ExistenceXgeneral-11}.

\ref{item:linkD8-44} $\p^2\link{2}{1} \Dl_8\link{4}{4}\Dl_8\link{1}{2}\p^2$, see Lemmas \ref{lem:linkD8-44}\ref{item:linkD8-44--fibering} and \ref{lem:ExistenceX8general-14}.

\ref{item:linkD6-33} $\p^2\link{2}{1} \Dl_8\link{3}{1} \Dl_6\link{3}{3} \Dl_6\link{1}{3}\Dl_8\link{1}{2}\p^2$,
see Lemmas \ref{lem:linkD6-33} and \ref{lem:ExistenceX6general-13} if $|\k|\geq 3$. For $\k=\F_2$ see \cite[Corollary 1.3]{SchneiderF2}.

\ref{item:linkD6-22} $\p^2\link{2}{1} \Dl_8\link{3}{1} \Dl_6\link{2}{2} \Dl_6\link{1}{3}\Dl_8\link{1}{2}\p^2$,
see Lemmas \ref{lem:linkD6-22}\ref{item:linkD6-22--fibering} and \ref{lem:ExistenceX6general-12}.

\ref{item:loop} $\p^2\link{5}{1} \Dl_5\link{2}{5}\Dl_8\link{1}{2}\p^2$,
see Lemmas \ref{lem:loop} and \ref{lem:ExistenceX5general-12}

\ref{item:linkP2-66} $\p^2\link{6}{6} \p^2$, see Lemmas \ref{lem:linkP2-66}\ref{item:linkP2-66--fibering} and \ref{lem:P2-6WithRationalPoint}.

In each case we obtain a factorization into involutions (using automorphisms, elements of Jonqui\`eres type, quadratic, Geiser and Bertini involutions), which gives the expected contradiction and finishes the proof of the theorem.
\qed
\bigskip

We can be slightly more precise and give the following involutive generating set:

\begin{proposition}
  $\Bir_\k(\p^2)$ is generated by linear involutions, involutions of Jonquières type $1$, $2+2$ and $4$, quadratic involutions, and maps conjugated to a Geiser or Bertini involution on a Del Pezzo surface of Picard rank $2$ or $3$.
\end{proposition}

\begin{proof}
  In the proof of Theorem~\ref{t:main} it is enough to give the composition into involutions of the irreducible maps of Del Pezzo type, which the mentioned lemmas provide except in the case of \ref{G1}-\ref{B4}.

  For the Geiser and Bertini simplifications $f=g\circ \iota$ in Proposition~\ref{prop:BertiniGeiserSimplification}, the map $\iota\in\Bir_\k(\p^2)$ is conjugate to a a Geiser \parent{or Bertini} involution on a del Pezzo surface of degree $2$ or $1$ that has Picard rank $2$, and $g$ has factorization type~\ref{Q21},~\ref{item:linkD5}, or~\ref{item:linkD6}.
  It is enough to observe that $g$ is either an automorphism or has a minimal factorization \ref{Q33}, \ref{Q21}, \ref{item:linkD5}, or \ref{item:linkD6}:

  First, assume $g$ has $\sl(g)\neq0$ and a factorization type \ref{Q21} or \ref{item:linkD5}. If it were not minimal, then $\sl(g)=1$ and hence $g$ has a minimal factorization $\p^2\link dd\p^2$ with a $d$-point as base point.
  By Figure~\ref{diag:links}, we have $d\in\{ 3,6,7,8\}$, but $g$ does not have any such $d$-point as base point, a contradiction.
  Moreover, if $g$ with $\sl(g)\neq 0$ has factorization type \ref{item:linkD6}, then $g$ has as base points at most one $2$-point, one $3$-point, and two rational points. Considering Figure~\ref{diag:links}, the factorization is either minimal, or $g$ has a minimal factorization as either \ref{Q33} or as \ref{Q21}.
\end{proof}

\begin{figure}[t]
\[
\includegraphics[scale=1]{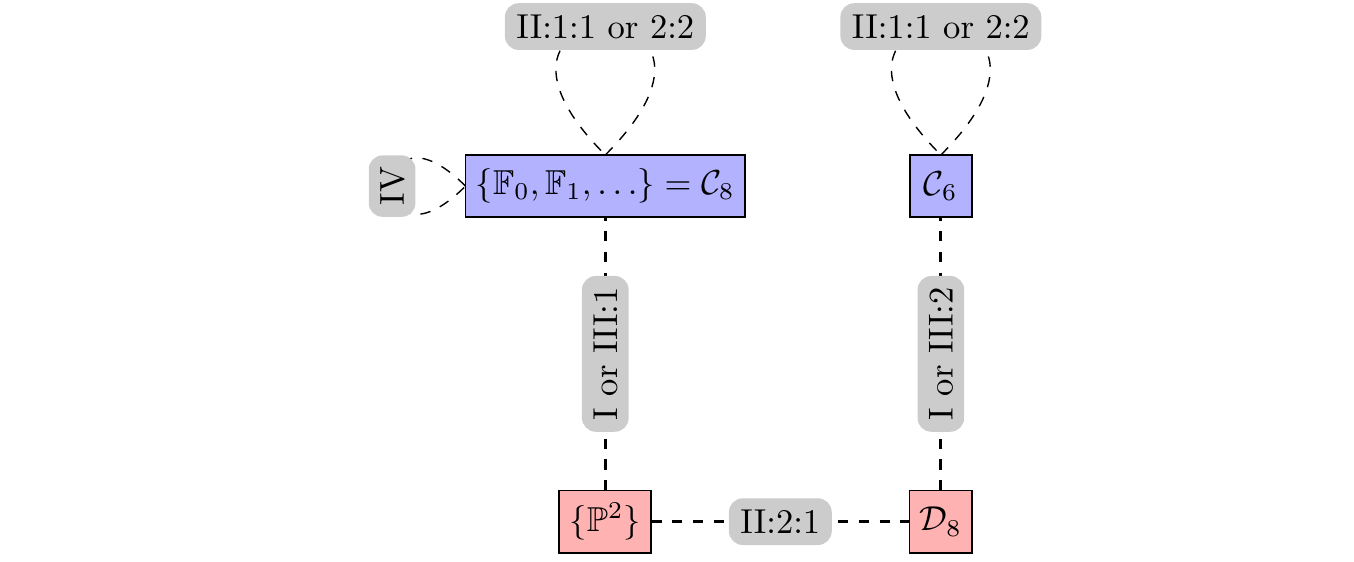}
\]
\caption{Sarkisov links between rational surfaces over $\R$}
\label{diag:linksR}
\end{figure}

\begin{remark}
 For $\k = \R$, or more generally for any field $\k$ with $[\alg \k:\k] = 2$, the graph of Sarkisov links reduces to Figure \ref{diag:linksR}, and Proposition \ref{prop:CremonaGeneratedByDelPezzoAndFiberTypes} becomes (compare with \cite[Corollary 4.2]{Zimmermann}): The Cremona group is generated by $\Aut_\k(\P^2)$, by the Jonquières group of type 1 and by the Jonquières group $\Bir_\k(\p^2,\pi_{L,L})$ of type 2 + 2, with $L = \alg \k$.
So to prove that $\Bir_\k(\P^2)$ is generated by involutions there is no need to consider sporadic cases of Del Pezzo type, and then our proof reduces essentially to the one given in \cite[Corollary II.4.12]{ZimmermannPhD}.
\end{remark}

\begin{remark}
 Over $\k = \FF_2$, some miracles occur with the sporadic cases.
For instance it is shown in \cite{SchneiderF2} that over $\FF_2$, a link $\P^2 \link66 \P^2$ is always an involution (up to $\PGL_3$), but this is certainly not the case for  general fields.
In fact it can be observed from \cite{DolgachevDuncan} that if the link is an involution then the cubic surface dominating the link must admit an Eckardt point (three $(-1)$-curve meeting at a same point), and it turns out that over $\FF_2$ all rational cubic surfaces of Picard rank 2 admit Eckardt points.
\end{remark}

\FloatBarrier

%% file: appendix_russianbis.tex
\section{Sarkisov links between rational surfaces over a perfect field}
\label{app:russian_bis}

In this appendix we provide a mostly self-contained proof of Theorem \ref{t:russian}.
Precisely, we use some references such as \cite{SchneiderRelations, SchneiderZimmermann}, but we avoid any reference to the paper \cite{Iskovskikh96}.

One should keep in mind that the theorem does not say anything about the existence of the links.
In particular:
\begin{itemize}
\item Over an algebraically closed field, the graph of Figure \ref{diag:links} reduces to only two vertices $\{\P^2\}$, $\Cl_8$.
\item Over the field $\R$, or more generally a field such that $[\alg \k : \k] = 2$, the graph of Figure \ref{diag:links} reduces to only four vertices $\{\P^2\}$, $\Cl_8$, $\Dl_8$, $\Cl_6$ (see Figure \ref{diag:linksR}).
\item Over the field $\F_2$, it was shown in \cite{SchneiderF2}  that all vertices appear, and all edges appear except the loop $\Dl_8 \link44  \Dl_8$.
\end{itemize}

The following result holds in greater generality, with a proof almost as short \cite[A.6]{KollarSzabo}.
However our argument allows a precise counting in the case of a finite field.

\begin{lemma}
\label{l:one_rational_point}
Let $X \rat Y$ be a birational map between surfaces defined over $\k$.
Then $X$ admits at least one rational point if and only if the same holds for $Y$.
\end{lemma}

\begin{proof}
Since any birational map is a composition of blow-ups, it is sufficient to consider the case of the blow-up of a $d$-point $X \stackrel{_d}{\to} Y$.
If $d \ge 2$ the rational points of $X$ and $Y$ are in bijection, and if $d = 1$ the exceptional divisor is isomorphic to $\P^1$ over $\k$ and so contains rational points.
\end{proof}

\begin{remark}
\label{r:at_least_3}
Over a finite field $\k = \FF_q$, the above argument shows that the blow-up of a rational point produces an exceptional divisor with $q+1$ rational points. In particular, given a link $X \link{d}1 X'$ of type II with $d \ge 2$, we have $|X'(\k)| = |X(\k)| - q$.
This shows
\begin{align*}
|\P^2(\k)| &= q^2 + q + 1, & |X_8(\k)| &= q^2 + 1, \\
 |X_6(\k)| &= q^2 - q + 1, & |X_5(\k)| &= q^2 + 1,
\end{align*}
for any $X_i \in \Dl_i$.
Observe also that all these numbers are greater or equal than $3$.
\end{remark}

In the next proposition we gather some characterizations of the rank 1 fibrations involved in our links (except for the class $\Dl_5$ where we could not find a reference):

\begin{proposition}
\label{p:delpezzo}
Let $X$ be a rational Del Pezzo surface defined over a perfect field $\k$.
Assume $X$ has degree $d$ and Picard rank $1$.
\begin{enumerate}
\item \label{delpezzoP2}
If $d = 9$ then $X = \P^2$.
\item \label{delpezzoD8}
If $d = 8$ then $X \in \Dl_8$.
\item \label{delpezzoD6}
If $d = 6$ then $X \in \Dl_6$.
\end{enumerate}
Let $X/\P^1$ be a rational conic bundle defined over a perfect field $\k$.
Assume $K_X^2 = d$ and the relative Picard rank of $X/\P^1$ is $1$.
\begin{enumerate}[resume]
\item \label{delpezzoC5}
If $d = 5$ then $X \in \Cl_5$.
\item \label{delpezzoC6}
If $d = 6$ then $X \in \Cl_6$.
\item \label{delpezzoC8}
If $d = 8$ then $X \in \Cl_8$.
\end{enumerate}
\end{proposition}

\begin{proof}
\begin{enumerate}
\item This is a result of Châtelet, see e.g. \cite[Theorem 5.1.3]{GilleSzamuely} or \cite[Corollary~13]{Kollar}.
\item \cite[Lemma 3.2]{SchneiderZimmermann}.
\item \cite[\S4]{SchneiderZimmermann}.
\item
We follow \cite[Lemma 6.13]{SchneiderRelations}.
$X$ admits $8 - K_X^2 = 3$ singular fibers.
Since the Picard rank of $X/\P^1$ is 1, for each singular fiber $C + C'$ there is an element $g \in \Gal(\alg\k/\k)$ that sends $C$ on $C'$, and since the intersection form is preserved the same $g$ sends $C'$ on $C$.
Since $X$ is rational, by Lemma \ref{l:one_rational_point} it contains a rational point, and so we can apply \cite[Lemma 6.5]{SchneiderRelations} and conclude that $X$ is the blow-up of four points in $\P^2$, invariant under $\Gal(\k^a/\k)$.
Moreover since the Picard rank of $X$ is 2, the four points form a single Galois orbit.

\item
$X$ admits $8 - K_X^2 = 2$ singular fibers.
A singular fiber contains at most one rational point (the meeting point of the two components), so by Remark \ref{r:at_least_3} we can pick a rational point $p$ outside the singular fibers.
Let $Y \to X$ be the blow-up of $p$.
Then as in the previous case we can apply \cite[Lemma 6.5]{SchneiderRelations} (with \cite[Observation~6.9]{SchneiderRelations}) to $Y$ and conclude that $Y$ is the blow-up of $\P^2$ at four points.
Since the Picard rank of $Y$ is $3$, the four points are divided in two Galois orbits.
The singular fibers correspond on $\p^2$ to pairs of lines going through all four points. If one of the points would be rational, none of these lines is defined over $\k$.
This is not possible since one of them corresponds to the exceptional curve $E$ of $p$, which is rational.
Therefore $Y\to\p^2$ is the blow-up at two $2$-points.
Let $F=\{F_1,F_2\}$ and $F'=\{F_1',F_2'\}$ be the orbits of $(-1)$-curves on $Y$ that can be contracted to $\p^2$.
By \cite[Lemma~6.5\parent{3}]{SchneiderRelations}, exactly one of them is pairwise disjoint with $E$, say $F$.
We obtain a commutative diagram
\[
\begin{tikzcd}
Y \ar[r, "F"] \ar[d, "E",swap] & X_7 \ar[r, "F'"] \ar[d, "E"] & \p^2 \\
X \ar[r,"F"] & X_8
\end{tikzcd}
\]
In particular, $X\in\Cl_6$.
\item
Over $\alg \k$, $X$ is isomorphic to a Hirzebruch surface $\F_n$ for some $n\geq0$.
If $n=1$, $X=\F_1$ because the $(-1)$-curve is rational and hence can be contracted to a Del Pezzo surface of degree $9$ with a rational point, which is $\p^2$.
For $n\geq 2$ we proceed by induction.
Since $X\to\p^1$ is defined over $\k$, there exist fibers defined over $\k$.
Pick a rational point on such a fiber, which is not the intersection point with the $(-n)$-curve.
Blowing it up and contracting the transform of the fiber gives a birational map $X\rat X'$ such that $X'$ is isomorphic to $\F_{n-1}$ over $\alg \k$, and with a morphism $X'\to\p^1$ over $\k$.
By induction, we arrive at $\F_1$ after $n-1$ steps. Therefore, each step was a map $\F_{i}\rat\F_{i-1}$ and so $X=\F_n$.
Similarly, in the case of $n = 0$, by blowing-up a rational point on a fiber defined over $\k$, and contracting the transform of the fiber, we obtain $\F_1$ and we conclude as above.
\qedhere
\end{enumerate}
\end{proof}

We say that a Sarkisov link of type II $X \link{d}{d'} Y$ is \emph{auto-similar} if $X$ and $Y$ are of the same type.
Comparing the Picard numbers of $X$ and $Y$ over $\alg \k$, we see that in this case $d = d'$.
On the graph of Figure \ref{diag:links}, a loop from one vertex to itself should be understood as the assertion that any corresponding Sarkisov link is auto-similar.
For instance, the loop $\P^2 \link88 \P^2$ means that blowing-up a general $8$-point on $\P^2$, we should always come back to $\P^2$ after contracting an orbit of $8$ disjoint $(-1)$-curves.
In the following lemma we check this assertion and other similar ones.

\begin{lemma}
Let $X \in \Dl$, $d \ge 1$  and $p \in X$ a general $d$-point.
Assume that $X$ and $d$ are such that the corresponding edge on the graph of Figure \ref{diag:links} is a loop.
Then the Sarkisov link starting with the blow-up of $p$ is indeed an auto-similar link.
\end{lemma}

\begin{proof}
Let $Z \to X$ be the blow-up of the $d$-point $p$.
If $Z$ is a Del Pezzo surface of degree 1 or 2, then by Lemma~\ref{l:bertini_geiser} up to an automorphism the link is induced by the Bertini or Geiser involution on $Z$ and so the conclusion is clear (and in fact we get a link $X \rat X$ between two isomorphic surfaces).

The Bertini case corresponds to the links:
\begin{align*}
\P^2 \link{8}{8} \P^2, &&
\Dl_8 \link{7}{7} \Dl_8, &&
\Dl_5 \link{4}{4} \Dl_5, &&
\Dl_6 \link{5}{5} \Dl_6.
\end{align*}

The Geiser case corresponds to the links:
\begin{align*}
\P^2 \link{7}{7} \P^2, &&
\Dl_8 \link{6}{6} \Dl_8, &&
\Dl_5 \link{3}{3} \Dl_5, &&
\Dl_6 \link{4}{4} \Dl_6.
\end{align*}

We now do a case by case analysis for the five remaining loops:

\begin{itemize}
\item
$\P^2 \link33 \P^2$.
We blow-down the orbit of three lines through two of the $p_i$, by Proposition~\ref{p:delpezzo}\ref{delpezzoP2} the resulting surface is $\P^2$.
\item
$\P^2 \link66 \P^2$.
We blow-down the orbit of six conics through five of the $p_i$, again by Proposition~\ref{p:delpezzo}\ref{delpezzoP2} the resulting surface is $\P^2$.
\item
$\Dl_8 \link44 \Dl_8$.
We blow-down the orbit of four diagonals through three of the $p_i$, by Proposition~\ref{p:delpezzo}\ref{delpezzoD8} the resulting surface is in $\Dl_8$.


\item
$\Dl_6 \link22 \Dl_6$.
On a Del Pezzo surface $X$ of degree 4, given two disjoint $(-1)$-curves $E_1, E_2$, there exists a unique pair of two other disjoint $(-1)$-curves $E_3, E_4$ such that $\sum_{i=1}^4 E_i = - K_X$.
We apply this remark to the exceptional divisor of the blow-up of the $2$-point, and we find an orbit of two curves that we can contract.
By Proposition~\ref{p:delpezzo}\ref{delpezzoD6} the resulting surface is in $\Dl_6$.


\item
$\Dl_6 \link33 \Dl_6$.
On a Del Pezzo surface $X$ of degree 3, given three pairwise disjoint $(-1)$-curves $E_1$, $E_2$, $E_3$, there exists a unique triple of other pairwise disjoint $(-1)$-curves $E_4$, $E_5$, $E_6$ such that $\sum_{i=1}^6 E_i = -2K_X$: intersect the cubic surface $X$ with the unique quadric surface containing $E_1, E_2, E_3$.
We apply this remark to the exceptional divisor of the blow-up of the $3$-point, and we find an orbit of three curves that we can contract.
Again, by Proposition~\ref{p:delpezzo}\ref{delpezzoD6} the resulting surface is in $\Dl_6$.
\qedhere
\end{itemize}
\end{proof}

We consider now the remaining edges between the Del Pezzo type vertices.

\begin{lemma}
\label{lem:P2-5And1Or2}
Let $r\in\p^2$ be a general $5$-point and let $p\in\p^2$ be either a rational point, or a $2$-point. Assume that $p$ does not lie on the conic through $r$.
Then $p$ is general with~$r$.
\end{lemma}

\begin{proof}
  In \cite[Lemma 4.9]{SchneiderF2} it was proved that any general $5$-point on $\p^2$ is either general with a $2$-point, or all seven components of the two points lie on the same conic.
  This implies the statement when $p$ is a $2$-point.

  We show now that a rational point on $\p^2$ is never collinear with two geometric components of $r$, which then implies the other part of the statement.
  The lines $L_{ij}$ through $r_i, r_j$ do not contain any rational point: There exists $\sigma\in \Gal(\k'/\k)$ such that $\sigma(r_i)=r_j$ and $\sigma(r_j)\neq r_i$, where $\k'/\k$ is the splitting field of $r$. Hence, if a rational point $p$ would lie on $L_{ij}$ then $p=\sigma(p)\in L_{j,\sigma(j)}$ and so $L_{ij}=L_{j,\sigma(j)}$ contains three components of $r$, a contradiction to the generality of $r$.
\end{proof}

\begin{lemma}
\label{l:D5-11}
Let $X \in \Dl_5$, $p \in X$ a rational point such that the blow-up of $p$ corresponds to a link $X \link15 \P^2$, and $q \in X$ an arbitrary rational point, distinct from $p$. Then:
\begin{enumerate}
\item \label{D5-11:1}
The points $p$ and $q$ are in general position.
\item \label{D5-11:2}
The Sarkisov link associated to the blow-up of $q$ also is of type $\Dl_5 \link15 \P^2$.
\end{enumerate}
In particular, any two rational points on $X$ are in general position.
\end{lemma}

\begin{proof}
Let $r \in \P^2$ be the general $5$-point that is contracted by the Sarkisov link  $X \link15 \P^2$ associated with $p$.
The point $q \in X$ corresponds to a rational point $q' \in \P^2$, which does not lie on the conic through $r$.
By Lemma \ref{lem:P2-5And1Or2} the blow-up of $r$ and $q'$ is a Del Pezzo surface $Z$, and since this is the same as the blow-up of $X$ at $p$ and $q$, this gives \ref{D5-11:1}.

The piece associated to the Del Pezzo surface $Z$ is the piece
$\hyperref[D5_11]{\piece{\Dl_5}11} = \hyperref[P2_15]{\piece{\P^2}15}$, which is described in Lemma \ref{l:pieceD511}.
In particular we get that the blow-up of $q$ also gives a link $X \link15 \P^2$.

The last assertion is immediate.
\end{proof}

\begin{lemma}
\label{l:D5-12}
Let $X \in \Dl_5$, $p \in X$ a rational point and $q \in X$ a $2$-point. Then:
\begin{enumerate}
\item \label{D5-12:1}
The points $p$ and $q$ are in general position.
\item \label{D5-12:2}
The Sarkisov link associated to the blow-up of $q$ is of type $\Dl_5 \link25 \Dl_8$.
\end{enumerate}
\end{lemma}

\begin{proof}
By Lemma \ref{l:D5-11} the blow-up of $p$ yields a link $X \link15 \P^2$.
Let $r \in \P^2$ be the general $5$-point that is blown-up by the inverse of this link.
The $2$-point $q \in X$ corresponds to a $2$-point $q' \in \P^2$, which does not lie on the conic through $r$.
By Lemma \ref{lem:P2-5And1Or2} the blow-up of $r$ and $q'$ is a Del Pezzo surface $Z$, and since this is the same as the blow-up of $X$ at $p$ and $q$, this gives \ref{D5-12:1}.

The piece associated to the Del Pezzo surface $Z$ is the piece
$\hyperref[D5_12]{\piece{\Dl_5}12} = \hyperref[P2_25]{\piece{\P^2}25}$, which is described in Lemma \ref{l:pieceP225}.
In particular we get that the blow-up of $q$ gives a link $X \link25 X' \in \Dl_8$.
\end{proof}

\begin{lemma}
\label{l:D6-11}
Let $X \in \Dl_6$, $p \in X$ a rational point such that the blow-up of $p$ corresponds to a link $X \link13 X'$ with $X' \in \Dl_8$, and $q \in X$ an arbitrary rational point, distinct from $p$. Then:
\begin{enumerate}
\item \label{D6-11:1}
The points $p$ and $q$ are in general position.
\item \label{D6-11:2}
The Sarkisov link associated to the blow-up of $q$ also is of type $\Dl_6 \link13 \Dl_8$.
\end{enumerate}
In particular, any two rational points on $X$ are in general position.
\end{lemma}

\begin{proof}
The proof is similar to the one of Lemma \ref{l:D5-11}, using now the piece
$\hyperref[D6_11]{\piece{\Dl_6}11} = \hyperref[D8_13]{\piece{\Dl_8}13}$, which is described in Lemma \ref{l:pieceD611}.
\end{proof}

\begin{lemma}
Let $X \in \Dl$, $d \ge 1$  and $p \in X$ a general $d$-point.
Assume the corresponding edge on the graph of Figure \textup{\ref{diag:links}} is of type II and is not a loop.
Then the Sarkisov link starting with the blow-up of $p$ is as prescribed by Figure \textup{\ref{diag:links}}.
\end{lemma}

\begin{proof}
For the edges $\P^2 \link21 \Dl_8$, $\P^2 \link51 \Dl_5$ and $\Dl_8 \link31 \Dl_6$, this follows from the definition of the classes $\Dl_i$.
Now we consider the inverses of these links.

If $X \in \Dl_8$ and $p \in X$ is a rational point, after the blow-up of $X$ at $p$ we can contract the transform of the horizontal and vertical rulings through $p$.
By Proposition \ref{p:delpezzo}\ref{delpezzoP2}, the resulting surface is $\P^2$, so we get a link $\Dl_8 \link12 \P^2$.

If $X \in \Dl_5$ and $p \in X$ is a rational point, the existence of a link $\Dl_5 \link15 \P^2$ starting with the blow-up of $p$ is given by Lemma \ref{l:D5-11}.
If $X \in \Dl_6$ and $p \in X$ is a rational point, the existence of a link $\Dl_6 \link13 \Dl_8$ is given by Lemma \ref{l:D6-11}.
Similarly, if $X \in \Dl_5$ and $p \in X$ is a general $2$-point,
the existence of a link $\Dl_5 \link25 \Dl_8$ is given by Lemma \ref{l:D5-12}.

Finally let $X \in \Dl_8$ and $p \in X$ a general $5$-point.
We claim that any rational point $r \in X$ is general with $p$.
Indeed, let $C$ be the curve of bidegree $(2,2)$ passing through $p,r$, and with a double point at $r$.
Since $p$ is general, $C$ is irreducible.
Consider the link $X \link12 \P^2$ starting with the blow-up of $r$, and let $q \in \P^2$ be the $2$-point image of the horizontal and vertical rulings through $r$.
Then the image of $C$ is a conic which does not contain $q$.
We know by Lemma \ref{lem:P2-5And1Or2} that $q$ and the image of $p$ are general, so $q$ and $r$ are also general as claimed.
We conclude to the existence of a link $\Dl_8 \link52 \Dl_5$ by using the piece $\hyperref[D8_15]{\piece{\Dl_8}15} = \hyperref[P2_25]{\piece{\P^2}25}$, which is described in Lemma \ref{l:pieceP225}
\end{proof}

We now turn to links involving the classes $\Cl_i$.
First recall the following lemma from \cite[Lemma 6.12]{SchneiderRelations}:

\begin{lemma}
\label{l:Julia_4pts}
Let $\Pl$, $\Ql \subset \P^2(\k^a)$ be two sets of four points, and $X_\Pl \to \P^2$, $X_\Ql \to \P^2$ the corresponding blow-ups.
Assume $\Pl$ is either a general $4$-point, or the union of two $2$-points that are general.
If there exists a birational map $X_\Pl \rat X_\Ql$ preserving
the fibrations associated with the pencils of conics through $\Pl$ and $\Ql$ respectively, then there exists an automorphism $\alpha \in \Aut_\k(\P^2)$ such that $\alpha(\Pl) = \Ql$, and in particular $X_\Pl$ and $X_\Ql$ are isomorphic.
\end{lemma}

Let $p, p'$ be two distinct $2$-points in general position in $\P^2$, and let $Y \to \P^2$ be the blow-up of $p$ and $p'$.
Then the transform of the pencil of conics through $p,p'$ gives a structure of rank $2$ fibration $Y/\P^1$.
Contracting either the transform of the line through $p$ or through $p'$, we obtain two distinct surfaces $X$, $X' \in \Cl_6$.
We say that $X$ and $X'$ are \emph{twin} elements in $\Cl_6$, and that $Y$ is their \emph{parent}.
Observe that $X$ and $X'$ are not necessarily isomorphic, and that they are uniquely defined by $Y$.

\begin{lemma}
\label{l:XisoX'}
Let $X \in \Cl_5 \cup \Cl_6$, and $p \in X$ a general $d$-point
\parent{here $d$ can be arbitrary large}.
Let $\chi\colon X \link{d}{d} X'$ be the Sarkisov link of type II over $\P^1$ constructed from the blow-up of $p$.
Then
\begin{enumerate}
\item \label{XisoX':C5}
If $X \in \Cl_5$, then $X'$ is isomorphic to $X$.
\item \label{XisoX':C6}
If $X \in \Cl_6$, then $X'$ is isomorphic to $X$ or to the twin of $X$.
\end{enumerate}
\end{lemma}

\begin{proof}
We know from Proposition \ref{p:delpezzo} that $X$ and $X'$ belong to the same class $\Cl_i$, $i = 5$ or $6$.
Then Lemma \ref{l:Julia_4pts} gives the result when $X \in \Cl_5$, and when $X \in \Cl_6$ it gives that $X, X'$ have the same parent, hence the result.
\end{proof}

\begin{lemma}
Let $\F_n \in \Cl_8$ be a Hirzebruch surface.
Then any Sarkisov link from $\F_n$ is one of the following:
\begin{itemize}
\item A link $\F_n \rat \F_m$ of type II over $\P^1$.
\item When $n = 1$, a link $\F_1 \to \P^2$ of type III.
\item When $n = 0$, a link $\F_0 \to \F_0$ of type IV.
\end{itemize}
\end{lemma}

\begin{proof}
If $\F_n \link{d}{d} X$ is a link of type II over $\P^1$, the surface $X/\P^1$ is again a Hirzebruch surface by \ref{p:delpezzo}\ref{delpezzoC8}.
Since $\F_n$ has Picard rank 2, it admits exactly two extremal rays, with one corresponding to the given structure of rank 1 fibration $\F_n/\P^1$.
If $n \ge 2$, the second extremal ray is generated by the exceptional section, which cannot be contracted to a smooth point.
So the only possibilities for a link of type III or IV are the cases $n \in \{0,1\}$, as expected.
\end{proof}

\begin{proof}[Proof of Theorem \textup{\ref{t:russian}}]
\ref{russian:1}.
Let $(X/B, \phi)$ be a marked rank 1 fibration.
By Proposition~\ref{p:Sarkisov}, we know that the birational map $\phi\colon X \rat \P^2$ admits a factorization into Sarkisov links.
In the above lemmas, we systematically explored all possible sequences of Sarkisov links starting from $\P^2$, and we showed that we never leave the seven classes $\{\P^2\}$, $\Dl_5$, $\Dl_6$, $\Dl_8$, $\Cl_5$, $\Cl_6$, $\Cl_8$.

Similarly, \ref{russian:2} and \ref{russian:3} follow from the exhaustive description of links.
A link of type IV must occur on a rank 1 fibration $X/\P^1$ such that the second extremal ray also corresponds to a fibration, and $\F_0$ is the only candidate.
\end{proof}

%% file: appendix_relations.tex
\section{Elementary relations}
\label{app:relations}

In this appendix we describe all elementary relations associated with a rank $3$ fibrations $X_d/\pt$, where $X_d$ is a Del Pezzo surface of Picard rank 3 and degree $d$.
In some sense this gives a modern version of the results in \cite{IKT}.
However the point of view is slightly different, since \cite{IKT} describes a set of relations for the Cremona group $\Bir_\k(\P^2)$, whereas by \cite[Theorem 3.1]{LamyZimmermann} our elementary relations constitute a set of relations for the groupoid $\BirMori_\k(\P^2)$, with respect to the Sarkisov links as generators.

Recall from page~\pageref{def:piece} that to any rank $3$ fibration we can associate a $2$-piece, which we represent as a polygon with vertices corresponding to some Mori fiber spaces, and edges to some Sarkisov links.
There are two possibilities for the surface $X_d$:
\begin{itemize}
\item Either $X_d$  is the blow-up of general $a$-point and $b$-point on $X_i \in \Dl_i \subset \Dl$, with $a + b + d = i$. We write $\piece{\Dl_i}ab$ the corresponding $2$-piece;
\item Or $X_d$  is the blow-up of a general $a$-point on $\F_0$, with $a + d = 8$.
We write $\pieceF{a}$ the corresponding $2$-piece.
\end{itemize}
In total there are $27$ distinct such $2$-pieces, see Table \ref{tab:pieces} and the figures below.
We do not include here the elementary relations associated to a rank 3 fibrations $X_d/\P^1$, since these are always of the same form: the corresponding piece is a square with all four vertices corresponding to conic bundles in the same subset $\Cl_i \subset \Cl$, $i = 5,6$ or $8$.

In the pictures, we use the following convention.
An edge labeled with $d$ is the blow-up of a general $d$-point, with one color associated to each $d$ (from $d = 1$ to $d = 7$).
A black edge without label is a change of base, as in Lemma \ref{l:2_domination_type}\ref{domination:base}.
A surface $X_i$ is a Del Pezzo surface of degree $i$.
In particular when a surface $X_i$ corresponds to a vertex, it has Picard rank $1$ and so belongs to the class $\Dl_i$, with its structure of fibration to the point. Similarly, a surface  $X_i/\P^1$ at a vertex denotes a surface in $\Cl_i$, with its structure of conic bundle.
We put some prime such as $X_i'$ when several surfaces with the same degree appear in the diagram, and we see no good reason why they should be isomorphic (a typical good reason is that they are related by a Geiser or Bertini involution, or that we can apply Lemma \ref{l:XisoX'}\ref{XisoX':C5}).

The proof that each piece has the form given in the pictures below can be done as follows.
Over $\alg \k$, we know the number of $(-1)$-curves and of rational fibrations on $X_d$.
Then we can study the Galois action on these curves, and find which orbits correspond to pairwise disjoint $(-1)$-curves and so correspond to a blow-down defined over $\k$.
We now give detailed statements for the pieces  $\piece{\P^2}15$, $\piece{\P^2}23$, $\piece{\P^2}25$, since these were used in our proof of Theorem \ref{t:russian} given in Appendix \ref{app:russian_bis}.

\begin{table}[t]
\begin{minipage}[t]{0.32\linewidth}
\strut\vspace*{-\baselineskip}\newline
\begin{tabular}{Cl}
\toprule
\text{Piece} & Figure \\
\midrule
\piece{\p^2}11 & Fig. \ref{P2_11}  \\
\piece{\p^2}12 & Fig. \ref{P2_12}  \\
\piece{\p^2}13 & Fig. \ref{P2_13}  \\
\piece{\p^2}14 & Fig. \ref{P2_14}  \\
\piece{\p^2}15 & Fig. \ref{P2_15}  \\
\piece{\p^2}16 & Fig. \ref{P2_16}  \\
\piece{\p^2}17 & Fig. \ref{P2_17}  \\
\piece{\p^2}22 & Fig. \ref{P2_22}  \\
\piece{\p^2}23 & Fig. \ref{P2_23}  \\
\piece{\p^2}24 & Fig. \ref{P2_24}  \\
\piece{\p^2}25 & Fig. \ref{P2_25}  \\
\piece{\p^2}26 & Fig. \ref{P2_26}  \\
\piece{\p^2}33 & Fig. \ref{P2_33}  \\
\piece{\p^2}34 & Fig. \ref{P2_34}  \\
\piece{\p^2}35 & Fig. \ref{P2_35}  \\
\piece{\p^2}44 & Fig. \ref{P2_44}  \\
\bottomrule \\
\end{tabular}
\end{minipage}
\begin{minipage}[t]{0.32\linewidth}
\strut\vspace*{-\baselineskip}\newline
\begin{tabular}{Cl}
\toprule
\text{Piece} & Figure \\
\midrule
\piece{\Dl_8}11 & Fig. \ref{P2_12} \\
\piece{\Dl_8}12 & Fig. \ref{P2_22} \\
\piece{\Dl_8}13 & Fig. \ref{P2_23}  \\
\piece{\Dl_8}14 & Fig. \ref{P2_24}  \\
\piece{\Dl_8}15 & Fig. \ref{P2_25}  \\
\piece{\Dl_8}16 & Fig. \ref{P2_26}  \\
\piece{\Dl_8}22 & Fig. \ref{D8_22}  \\
\piece{\Dl_8}23 & Fig. \ref{D8_23}  \\
\piece{\Dl_8}24 & Fig. \ref{D8_24}  \\
\piece{\Dl_8}25 & Fig. \ref{D8_25}  \\
\piece{\Dl_8}33 & Fig. \ref{D8_33}  \\
\piece{\Dl_8}34 & Fig. \ref{D8_34}  \\
\midrule
\piece{\Dl_5}11 & Fig. \ref{P2_15}  \\
\piece{\Dl_5}12 & Fig. \ref{P2_25}  \\
\piece{\Dl_5}13 & Fig. \ref{P2_35}  \\
\piece{\Dl_5}22 & Fig. \ref{D8_25}  \\
\bottomrule \\
\end{tabular}
\end{minipage}
\begin{minipage}[t]{0.32\linewidth}
\strut\vspace*{-\baselineskip}\newline
\begin{tabular}{Cl}
\toprule
\text{Piece} & Figure \\
\midrule
\pieceF{1} & Fig. \ref{P2_12} \\
\pieceF{2} & Fig. \ref{F0_2} \\
\pieceF{3} & Fig. \ref{P2_13} \\
\pieceF{4} & Fig. \ref{F0_4} \\
\pieceF{5} & Fig. \ref{P2_15} \\
\pieceF{6} & Fig. \ref{F0_6} \\
\pieceF{7} & Fig. \ref{P2_17} \\
\midrule
\piece{\Dl_6}11 & Fig. \ref{P2_23} \\
\piece{\Dl_6}12 & Fig. \ref{D8_23}  \\
\piece{\Dl_6}13 & Fig. \ref{D8_33}  \\
\piece{\Dl_6}14 & Fig. \ref{D8_34} \\
\piece{\Dl_6}22 & Fig. \ref{D6_22}  \\
\piece{\Dl_6}23 & Fig. \ref{D6_23} \\
\bottomrule \\
\end{tabular}
\end{minipage}
\caption{The 27 elementary relations over a point}
\label{tab:pieces}
\end{table}

\begin{lemma}[{piece $\hyperref[P2_15]{\piece{\P^2}{1}{5}}$}]
\label{l:pieceD511}
Let $p \in \P^2$ be a rational point, and $q \in \P^2$ a $5$-point such that $p,q$ are in general position.
Let $X_3 \to \P^2$ be the blow-up of $p$ and $q$.
Then $X_3$ is a Del Pezzo surface of degree $3$ and Picard rank $3$, and admits exactly seven  extremal rays described as follows \parent{we order them such that the intersection product of two consecutive rays is zero}:
\begin{itemize}
\item The exceptional divisor $E_1$ from $p$;
\item The exceptional divisor $E_5$ from $q$;
\item The pencil of lines through $p$, giving a structure of rank $2$ fibration $X_3/\P^1$;
\item  The transform $L_5$ of the five lines  through $p$ and one of the $q_i$;
\item The pencil of cubics through $p, q$ and with a double point at $p$, corresponding to a second structure of rank $2$ fibration $X_3/\P^1$;
\item The transform $O_5$ of the five conics through $p$ and four of the $q_i$;
\item The transform $O_1$ of the conic through $q$.
\end{itemize}
In consequence, the piece $\hyperref[P2_15]{\piece{\P^2}15}$ does not depend on the particular choice of the blown-up points, and coincides with the piece $\hyperref[D5_11]{\piece{\Dl_5}11}$.
\end{lemma}

\begin{proof}
It is routine to check that the listed curves become smooth rational curves on $X_3$:
\begin{itemize}
\item Either of self-intersection $0$, and so moving in a pencil;
\item Or a disjoint union of $(-1)$-curves, forming a single Galois orbit.
\end{itemize}
At this point we know that the piece is a heptagon, and that the surfaces corresponding to vertices have respective degree $9$, $5$, $9$, $8$, $8$, $8$, $8$.
By Proposition \ref{p:delpezzo}, the surfaces of degree 9 are $\P^2$.
Then since the surface of degree $5$ is obtained by a Sarkisov link from $\P^2$, it belongs to $\Dl_5$.
For the four vertices corresponding to surfaces in $\Cl_8$, two of them come from the blow-up of a rational point on $\P^2$ and so are $\F_1$, and the two other are related by a link of type IV and so are $\F_0$.
\end{proof}

Proofs for the next two lemmas are similar, so we omit them.

\begin{lemma}[{$\hyperref[P2_23]{\piece{\P^2}23} = \hyperref[D8_13]{\piece{\Dl_8}13} = \hyperref[D6_11]{\piece{\Dl_6}11}$}]
\label{l:pieceD611} \label{l:pieceP223} \label{l:pieceD813}
Let $X_8 \in \Dl_8$, and $q \in X_8$ a general $3$-point.
Then:
\begin{enumerate}
\item
For any rational point $p$ not on the diagonal passing through $q$, the points $p,q$ are in general position.
\item Let $X_4 \to X_8$ be the blow-up of such points $p,q$.
Then $X_4$ is a Del Pezzo surface of degree $4$ and Picard rank $3$, and admits exactly five  extremal rays described as follows \parent{we order them such that the intersection product of two consecutive rays is zero}:
\begin{itemize}
\item The exceptional divisor $E_3$ from $q$;
\item The exceptional divisor $E_1$ from $p$;
\item The transform $D_1$ of the diagonal through $q$;
\item The transform $D_3$ of the three diagonals through $p$ and two of the $q_i$;
\item The transform $R_2$ of the horizontal and vertical rulings through $p$.
\end{itemize}
\end{enumerate}
As a consequence the piece $\hyperref[D8_13]{\piece{\Dl_8}13}$ does not depend on the particular choice of the blown-up points, and coincides with the pieces $\hyperref[P2_23]{\piece{\P^2}23}$ and $\hyperref[D6_11]{\piece{\Dl_6}11}$.
\end{lemma}

\begin{lemma}[{$\hyperref[P2_25]{\piece{\P^2}{2}{5}}$}]
\label{l:pieceD815} \label{l:pieceP225} \label{l:pieceD515}
Let $q \in \P^2$ be a $2$-point.
Let $p\in \P^2$ be a $5$-point in general position with $q$, and let
$X_2 \to \P^2$ be the blow-up of $p,q$.
Then $X_2$ is a Del Pezzo surface of degree $2$ and Picard rank $3$, and admits exactly six extremal rays described as follows
\begin{itemize}
\item The exceptional divisor $E_5$ from $p$;
\item The line $L_1$ through $q$;
\item the cubics $C_2$ passing through $p,q$ and singular at one of the $q_i$;
\item the cubics $C_5$ passing through $p,q$ and singular at one of the $p_i$;
\item the conic $O_1$ through $p$;
\item the exceptional divisor $E_2$.
\end{itemize}
Let $\P^2 \link21 X_8$ be the Sarkisov link starting with the blow-up of $q$, let $r \in X_8$ be the rational point image of the line through $q$, and still denote by $p$ the image of the $5$-point on $X_8$. Then we can see the surface $X_2$ as the blow-up of $r$ and $p$ on $X_8$, and in term of transform of curves coming from $X_8$ the above list of extremal rays becomes:
\begin{itemize}
\item The exceptional divisor $E_5$ from $p$;
\item the exceptional divisor $E_1$ from $r$;
\item the curves $C_2$ of bidegree $(1,2)$ and $(2,1)$ through $p$;
\item the curves $C_5$ of bidegree $(2,2)$ through $p$ and $r$, and singular at one of the $p_i$;
\item the curves $O_1$ of bidegree $(2,2)$ through $p$ and $r$, and singular at $r$;
\item the vertical and horizontal rulings through $r$.
\end{itemize}
As a consequence, the piece $\hyperref[P2_25]{\piece{\P^2}25}$ does not depend on the particular choice of the blown-up points, and coincides with the pieces $\hyperref[D8_25]{\piece{\Dl_8}25}$.
\end{lemma}

\begin{remark}
Over the field $\k = \R$, a similar list of $2$-pieces was used by Zimmermann in \cite{ZimmermannAmalgam}, where she calls ``disc of type 1 to 6'' our pieces involving only rational points or $2$-points. The correspondence is as follows:
\begin{itemize}
\item Disc of type 1: $\hyperref[P2_12]{\piece{\P^2}12}$;
\item Disc of type 2: square relation between conic bundles in $\Cl_6$;
\item Disc of type 3: $\hyperref[P2_22]{\piece{\P^2}22}$;
\item Disc of type 4: $\hyperref[D8_22]{\piece{\Dl_8}22}$;
\item Disc of type 5: $\hyperref[P2_11]{\piece{\P^2}11}$;
\item Disc of type 6: square relation between Hirzebruch surfaces.
\end{itemize}
\end{remark}

\clearpage
\newgeometry{bottom=1cm}
\renewcommand\thefigure{\thesection.\arabic{figure}}
\setcounter{figure}{0}

\begin{figure}[ht]
\begin{minipage}[c]{0.49\linewidth}
\[
\input{tikz/P2_11}
\]
\caption{$\piece{\P^2}11$, $\pieceF1$}
\label{P2_11}
\end{minipage}
\begin{minipage}[c]{0.49\linewidth}
\[
\input{tikz/P2_12}
\]
\caption{$\piece{\P^2}12$, $\piece{\Dl_8}11$}
\label{P2_12}
\end{minipage}
\end{figure}

\begin{figure}[ht]
\begin{minipage}[b]{0.49\linewidth}
\[
\input{tikz/P2_13}
\]
\caption{$\piece{\P^2}13$, $\pieceF3$}
\label{P2_13}
\end{minipage}
\begin{minipage}[b]{0.49\linewidth}
\[
\input{tikz/P2_14}
\]
\caption{$\piece{\P^2}14$}
\label{P2_14}
\end{minipage}
\end{figure}

\begin{figure}[ht]
\begin{minipage}[c]{0.50\linewidth}
\mbox{}\\[-\baselineskip]
\[
\input{tikz/P2_15}
\]
\end{minipage}
\begin{minipage}[c]{0.49\linewidth}
\mbox{}\\[-\baselineskip]
\caption{$\piece{\p^2}15$, $\piece{\Dl_5}11$, $\pieceF5$}
\label{P2_15} \label{D5_11}
\end{minipage}
\end{figure}

\begin{figure}
\begin{minipage}[t]{0.49\linewidth}
\[
\input{tikz/P2_16}
\]
\caption{$\piece{\p^2}16$}
\label{P2_16}
\end{minipage}
\begin{minipage}[t]{0.49\linewidth}
\[
\input{tikz/P2_26}
\]
\caption{$\piece{\P^2}26$, $\piece{\Dl_8}16$}
\label{P2_26} \label{D8_16}
\end{minipage}
\end{figure}

\begin{figure}
\[
\input{tikz/P2_17}
\]
\caption{$\piece{\p^2}17$, $\pieceF7$}
\label{P2_17}
\end{figure}

\begin{figure}[ht]
\begin{minipage}[b]{0.49\linewidth}
\[
\input{tikz/P2_22}
\]
\caption{$\piece{\P^2}22$, $\piece{\Dl_8}12$}
\label{P2_22}
\end{minipage}
\begin{minipage}[b]{0.49\linewidth}
\[
\input{tikz/P2_23}
\]
\caption{$\piece{\P^2}23$, $\piece{\Dl_8}13$, $\piece{\Dl_6}11$}
\label{P2_23} \label{D8_13} \label{D6_11}
\end{minipage}
\end{figure}

\begin{figure}[ht]
\begin{minipage}[b]{0.49\linewidth}
\[
\input{tikz/P2_24}
\]
\caption{$\piece{\P^2}24$, $\piece{\Dl_8}14$}
\label{P2_24}
\end{minipage}
\begin{minipage}[b]{0.49\linewidth}
\[
\input{tikz/P2_25}
\]
\caption{$\piece{\P^2}25$, $\piece{\Dl_8}15$, $\piece{\Dl_5}12$}
\label{P2_25} \label{D8_15} \label{D5_12}
\end{minipage}
\end{figure}

\begin{figure}[ht]
\begin{minipage}[t]{0.49\linewidth}
\[
\input{tikz/P2_33}
\]
\caption{$\piece{\P^2}33$}
\label{P2_33}
\end{minipage}
\begin{minipage}[t]{0.49\linewidth}
\[
\input{tikz/X8_33}
\]
\caption{$\piece{\Dl_8}33$, $\piece{\Dl_6}13$}
\label{D8_33} \label{D6_13}
\end{minipage}
\end{figure}

\begin{figure}
\begin{minipage}[t]{0.49\linewidth}
\[
\input{tikz/P2_34}
\]
\caption{$\piece{\P^2}34$}
\label{P2_34}
\end{minipage}
\begin{minipage}[t]{0.49\linewidth}
\[
\input{tikz/P2_35}
\]
\caption{$\piece{\P^2}35$, $\piece{\Dl_5}13$}
\label{P2_35}
\end{minipage}
\end{figure}

\begin{figure}
\[
\input{tikz/P2_44}
\]
\caption{$\piece{\P^2}44$}
\label{P2_44}
\end{figure}

\begin{figure}[ht]
\begin{minipage}[t]{0.49\linewidth}
\[
\input{tikz/X8_22}
\]
\caption{$\piece{\Dl_8}22$}
\label{D8_22}
\end{minipage}
\begin{minipage}[t]{0.49\linewidth}
\[
\input{tikz/X8_23}
\]
\caption{$\piece{\Dl_8}23$, $\piece{\Dl_6}12$}
\label{D8_23} \label{D6_12}
\end{minipage}
\end{figure}

\begin{figure}[ht]
\begin{minipage}[b]{0.49\linewidth}
\[
\input{tikz/X6_22}
\]
\caption{$\piece{\Dl_6}22$}
\label{D6_22}
\end{minipage}
\begin{minipage}[b]{0.49\linewidth}
\[
\input{tikz/X6_23}
\]
\caption{$\piece{\Dl_6}23$}
\label{D6_23}
\end{minipage}
\end{figure}

\begin{figure}
\begin{minipage}[t]{0.49\linewidth}
\[
\input{tikz/X8_24}
\]
\caption{$\piece{\Dl_8}24$}
\label{D8_24}
\end{minipage}
\begin{minipage}[t]{0.49\linewidth}
\[
\input{tikz/X8_34}
\]
\caption{$\piece{\Dl_8}34$, $\piece{\Dl_6}14$}
\label{D8_34} \label{D6_14}
\end{minipage}
\end{figure}

\begin{figure}
\[
\input{tikz/X8_25}
\]
\caption{$\piece{\Dl_8}25$, $\piece{\Dl_5}22$}
\label{D8_25} \label{D5_22}
\end{figure}

\begin{figure}[ht]
\begin{minipage}[c]{0.50\linewidth}
\mbox{}\\[-\baselineskip]
\[
\input{tikz/F0_02}
\]
\end{minipage}
\begin{minipage}[c]{0.49\linewidth}
\mbox{}\\[-\baselineskip]
\caption{$\pieceF2$}
\label{F0_2}
\end{minipage}
\end{figure}

\begin{figure}[ht]
\begin{minipage}[c]{0.49\linewidth}
\[
\input{tikz/F0_04}
\]
\caption{$\pieceF4$}
\label{F0_4}
\end{minipage}
\begin{minipage}[c]{0.49\linewidth}
\[
\input{tikz/F0_06}
\]
\caption{$\pieceF6$}
\label{F0_6}
\end{minipage}
\end{figure}

\FloatBarrier

%% file: tikz/P2_11.tex
\begin{tikzpicture}[font=\footnotesize]
\node[name=s, regular polygon, rotate=180, regular polygon sides=5,inner sep=1.242cm] at (0,0) {};
\draw[-,ultra thick,RoyalBlue] (s.center) to["1"] (s.side 1);
\draw[-,ultra thick,black] (s.center) to[""] (s.side 2);
\draw[-,ultra thick,RoyalBlue] (s.center) to["1"] (s.side 3);
\draw[-,ultra thick,black] (s.center) to[""] (s.side 4);
\draw[-,ultra thick,RoyalBlue] (s.center) to["1"] (s.side 5);
\draw[-,ultra thick,black] (s.side 1) to[swap,""] (s.corner 2);
\draw[-,ultra thick,RoyalBlue] (s.side 1) to["1"] (s.corner 1);
\draw[-,ultra thick,RoyalBlue] (s.side 2) to[swap,"1"] (s.corner 3);
\draw[-,ultra thick,RoyalBlue] (s.side 2) to["1"] (s.corner 2);
\draw[-,ultra thick,black] (s.side 3) to[swap,""] (s.corner 4);
\draw[-,ultra thick,black] (s.side 3) to[""] (s.corner 3);
\draw[-,ultra thick,RoyalBlue] (s.side 4) to[swap,"1"] (s.corner 5);
\draw[-,ultra thick,RoyalBlue] (s.side 4) to["1"] (s.corner 4);
\draw[-,ultra thick,RoyalBlue] (s.side 5) to[swap,"1"] (s.corner 1);
\draw[-,ultra thick,black] (s.side 5) to[""] (s.corner 5);
\node[fill=white,circle,inner sep=2] at (s.center) {${X_7}$};
\node[fill=white,inner sep=2] at (s.corner 1) {${\p^2}$};
\node[fill=white,inner sep=2] at (s.side 1) {${\F_1}$};
\node[fill=white,inner sep=2] at (s.corner 2) {${\F_1}_{\!/\p^1}$};
\node[fill=white,inner sep=2] at (s.side 2) {${X_7}_{/\p^1}$};
\node[fill=white,inner sep=2] at (s.corner 3) {${\F_0}_{\!/\p^1}$};
\node[fill=white,inner sep=2] at (s.side 3) {${\F_0}$};
\node[fill=white,inner sep=2] at (s.corner 4) {${\F_0}_{\!/\p^1}$};
\node[fill=white,inner sep=2] at (s.side 4) {${X_7}_{/\p^1}$};
\node[fill=white,inner sep=2] at (s.corner 5) {${\F_1}_{\!/\p^1}$};
\node[fill=white,inner sep=2] at (s.side 5) {${\F_1}$};
\end{tikzpicture}

%% file: tikz/P2_12.tex
\begin{tikzpicture}[font=\footnotesize]
\node[name=s, regular polygon, rotate=180, regular polygon sides=5,inner sep=1.242cm] at (0,0) {};
\draw[-,ultra thick,FireBrick] (s.center) to["2"] (s.side 1);
\draw[-,ultra thick,black] (s.center) to[""] (s.side 2);
\draw[-,ultra thick,FireBrick] (s.center) to["2"] (s.side 3);
\draw[-,ultra thick,RoyalBlue] (s.center) to["1"] (s.side 4);
\draw[-,ultra thick,RoyalBlue] (s.center) to["1"] (s.side 5);
\draw[-,ultra thick,black] (s.side 1) to[swap,""] (s.corner 2);
\draw[-,ultra thick,RoyalBlue] (s.side 1) to["1"] (s.corner 1);
\draw[-,ultra thick,FireBrick] (s.side 2) to[swap,"2"] (s.corner 3);
\draw[-,ultra thick,FireBrick] (s.side 2) to["2"] (s.corner 2);
\draw[-,ultra thick,RoyalBlue] (s.side 3) to[swap,"1"] (s.corner 4);
\draw[-,ultra thick,black] (s.side 3) to[""] (s.corner 3);
\draw[-,ultra thick,RoyalBlue] (s.side 4) to[swap,"1"] (s.corner 5);
\draw[-,ultra thick,FireBrick] (s.side 4) to["2"] (s.corner 4);
\draw[-,ultra thick,FireBrick] (s.side 5) to[swap,"2"] (s.corner 1);
\draw[-,ultra thick,RoyalBlue] (s.side 5) to["1"] (s.corner 5);
\node[fill=white,circle,inner sep=2] at (s.center) {${X_6}$};
\node[fill=white,inner sep=2] at (s.corner 1) {${\p^2}$};
\node[fill=white,inner sep=2] at (s.side 1) {${\F_1}$};
\node[fill=white,inner sep=2] at (s.corner 2) {${\F_1}_{\!/\p^1}$};
\node[fill=white,inner sep=2] at (s.side 2) {${X_6}_{/\p^1}$};
\node[fill=white,inner sep=2] at (s.corner 3) {${\F_1}_{\!/\p^1}$};
\node[fill=white,inner sep=2] at (s.side 3) {${\F_1}$};
\node[fill=white,inner sep=2] at (s.corner 4) {${\p^2}$};
\node[fill=white,inner sep=2] at (s.side 4) {${X_7}$};
\node[fill=white,inner sep=2] at (s.corner 5) {${X_8}$};
\node[fill=white,inner sep=2] at (s.side 5) {${X_7'}$};
\end{tikzpicture}

%% file: tikz/P2_13.tex
\begin{tikzpicture}[font=\footnotesize]
\node[name=s, regular polygon, rotate=180, regular polygon sides=6,inner sep=1.557cm] at (0,0) {};
\draw[-,ultra thick,ForestGreen] (s.center) to["3"] (s.side 1);
\draw[-,ultra thick,black] (s.center) to[""] (s.side 2);
\draw[-,ultra thick,ForestGreen] (s.center) to["3"] (s.side 3);
\draw[-,ultra thick,black] (s.center) to[""] (s.side 4);
\draw[-,ultra thick,ForestGreen] (s.center) to["3"] (s.side 5);
\draw[-,ultra thick,RoyalBlue] (s.center) to["1"] (s.side 6);
\draw[-,ultra thick,black] (s.side 1) to[swap,""] (s.corner 2);
\draw[-,ultra thick,RoyalBlue] (s.side 1) to["1"] (s.corner 1);
\draw[-,ultra thick,ForestGreen] (s.side 2) to[swap,"3"] (s.corner 3);
\draw[-,ultra thick,ForestGreen] (s.side 2) to["3"] (s.corner 2);
\draw[-,ultra thick,black] (s.side 3) to[swap,""] (s.corner 4);
\draw[-,ultra thick,black] (s.side 3) to[""] (s.corner 3);
\draw[-,ultra thick,ForestGreen] (s.side 4) to[swap,"3"] (s.corner 5);
\draw[-,ultra thick,ForestGreen] (s.side 4) to["3"] (s.corner 4);
\draw[-,ultra thick,RoyalBlue] (s.side 5) to[swap,"1"] (s.corner 6);
\draw[-,ultra thick,black] (s.side 5) to[""] (s.corner 5);
\draw[-,ultra thick,ForestGreen] (s.side 6) to[swap,"3"] (s.corner 1);
\draw[-,ultra thick,ForestGreen] (s.side 6) to["3"] (s.corner 6);
\node[fill=white,circle,inner sep=2] at (s.center) {${X_5}$};
\node[fill=white,inner sep=2] at (s.corner 1) {${\p^2}$};
\node[fill=white,inner sep=2] at (s.side 1) {${\F_1}$};
\node[fill=white,inner sep=2] at (s.corner 2) {${\F_1}_{\!/\p^1}$};
\node[fill=white,inner sep=2] at (s.side 2) {${X_5}_{/\p^1}$};
\node[fill=white,inner sep=2] at (s.corner 3) {${\F_0}_{\!/\p^1}$};
\node[fill=white,inner sep=2] at (s.side 3) {${\F_0}$};
\node[fill=white,inner sep=2] at (s.corner 4) {${\F_0}_{\!/\p^1}$};
\node[fill=white,inner sep=2] at (s.side 4) {${X_5}_{/\p^1}$};
\node[fill=white,inner sep=2] at (s.corner 5) {${\F_1}_{\!/\p^1}$};
\node[fill=white,inner sep=2] at (s.side 5) {${\F_1}$};
\node[fill=white,inner sep=2] at (s.corner 6) {${\p^2}$};
\node[fill=white,inner sep=2] at (s.side 6) {${X_6}$};
\end{tikzpicture}

%% file: tikz/P2_14.tex
\begin{tikzpicture}[font=\footnotesize]
\node[name=s, regular polygon, rotate=180, regular polygon sides=6,inner sep=1.557cm] at (0,0) {};
\draw[-,ultra thick,Orange] (s.center) to["4"] (s.side 1);
\draw[-,ultra thick,black] (s.center) to[""] (s.side 2);
\draw[-,ultra thick,Orange] (s.center) to["4"] (s.side 3);
\draw[-,ultra thick,RoyalBlue] (s.center) to["1"] (s.side 4);
\draw[-,ultra thick,black] (s.center) to[""] (s.side 5);
\draw[-,ultra thick,RoyalBlue] (s.center) to["1"] (s.side 6);
\draw[-,ultra thick,black] (s.side 1) to[swap,""] (s.corner 2);
\draw[-,ultra thick,RoyalBlue] (s.side 1) to["1"] (s.corner 1);
\draw[-,ultra thick,Orange] (s.side 2) to[swap,"4"] (s.corner 3);
\draw[-,ultra thick,Orange] (s.side 2) to["4"] (s.corner 2);
\draw[-,ultra thick,RoyalBlue] (s.side 3) to[swap,"1"] (s.corner 4);
\draw[-,ultra thick,black] (s.side 3) to[""] (s.corner 3);
\draw[-,ultra thick,black] (s.side 4) to[swap,""] (s.corner 5);
\draw[-,ultra thick,Orange] (s.side 4) to["4"] (s.corner 4);
\draw[-,ultra thick,RoyalBlue] (s.side 5) to[swap,"1"] (s.corner 6);
\draw[-,ultra thick,RoyalBlue] (s.side 5) to["1"] (s.corner 5);
\draw[-,ultra thick,Orange] (s.side 6) to[swap,"4"] (s.corner 1);
\draw[-,ultra thick,black] (s.side 6) to[""] (s.corner 6);
\node[fill=white,circle,inner sep=2] at (s.center) {${X_4}$};
\node[fill=white,inner sep=2] at (s.corner 1) {${\p^2}$};
\node[fill=white,inner sep=2] at (s.side 1) {${\F_1}$};
\node[fill=white,inner sep=2] at (s.corner 2) {${\F_1}_{\!/\p^1}$};
\node[fill=white,inner sep=2] at (s.side 2) {${X_4}_{/\p^1}$};
\node[fill=white,inner sep=2] at (s.corner 3) {${\F_1}_{\!/\p^1}$};
\node[fill=white,inner sep=2] at (s.side 3) {${\F_1}$};
\node[fill=white,inner sep=2] at (s.corner 4) {${\p^2}$};
\node[fill=white,inner sep=2] at (s.side 4) {${X_5}$};
\node[fill=white,inner sep=2] at (s.corner 5) {${X_5}_{\!/\p^1}$};
\node[fill=white,inner sep=2] at (s.side 5) {${X_4}_{/\p^1}$};
\node[fill=white,inner sep=2] at (s.corner 6) {${X_5}_{\!/\p^1}$};
\node[fill=white,inner sep=2] at (s.side 6) {${X_5}$};
\end{tikzpicture}

%% file: tikz/P2_15.tex
\begin{tikzpicture}[font=\footnotesize]
\node[name=s, regular polygon, rotate=180, regular polygon sides=7,inner sep=1.872cm] at (0,0) {};
\draw[-,ultra thick,HotPink] (s.center) to["5"] (s.side 1);
\draw[-,ultra thick,black] (s.center) to[""] (s.side 2);
\draw[-,ultra thick,HotPink] (s.center) to["5"] (s.side 3);
\draw[-,ultra thick,black] (s.center) to[""] (s.side 4);
\draw[-,ultra thick,HotPink] (s.center) to["5"] (s.side 5);
\draw[-,ultra thick,RoyalBlue] (s.center) to["1"] (s.side 6);
\draw[-,ultra thick,RoyalBlue] (s.center) to["1"] (s.side 7);
\draw[-,ultra thick,black] (s.side 1) to[swap,""] (s.corner 2);
\draw[-,ultra thick,RoyalBlue] (s.side 1) to["1"] (s.corner 1);
\draw[-,ultra thick,HotPink] (s.side 2) to[swap,"5"] (s.corner 3);
\draw[-,ultra thick,HotPink] (s.side 2) to["5"] (s.corner 2);
\draw[-,ultra thick,black] (s.side 3) to[swap,""] (s.corner 4);
\draw[-,ultra thick,black] (s.side 3) to[""] (s.corner 3);
\draw[-,ultra thick,HotPink] (s.side 4) to[swap,"5"] (s.corner 5);
\draw[-,ultra thick,HotPink] (s.side 4) to["5"] (s.corner 4);
\draw[-,ultra thick,RoyalBlue] (s.side 5) to[swap,"1"] (s.corner 6);
\draw[-,ultra thick,black] (s.side 5) to[""] (s.corner 5);
\draw[-,ultra thick,RoyalBlue] (s.side 6) to[swap,"1"] (s.corner 7);
\draw[-,ultra thick,HotPink] (s.side 6) to["5"] (s.corner 6);
\draw[-,ultra thick,HotPink] (s.side 7) to[swap,"5"] (s.corner 1);
\draw[-,ultra thick,RoyalBlue] (s.side 7) to["1"] (s.corner 7);
\node[fill=white,circle,inner sep=2] at (s.center) {${X_3}$};
\node[fill=white,inner sep=2] at (s.corner 1) {${\p^2}$};
\node[fill=white,inner sep=2] at (s.side 1) {${\F_1}$};
\node[fill=white,inner sep=2] at (s.corner 2) {${\F_1}_{\!/\p^1}$};
\node[fill=white,inner sep=2] at (s.side 2) {${X_3}_{/\p^1}$};
\node[fill=white,inner sep=2] at (s.corner 3) {${\F_0}_{\!/\p^1}$};
\node[fill=white,inner sep=2] at (s.side 3) {${\F_0}$};
\node[fill=white,inner sep=2] at (s.corner 4) {${\F_0}_{\!/\p^1}$};
\node[fill=white,inner sep=2] at (s.side 4) {${X_3}_{/\p^1}$};
\node[fill=white,inner sep=2] at (s.corner 5) {${\F_1}_{\!/\p^1}$};
\node[fill=white,inner sep=2] at (s.side 5) {${\F_1}$};
\node[fill=white,inner sep=2] at (s.corner 6) {${\p^2}$};
\node[fill=white,inner sep=2] at (s.side 6) {${X_4}$};
\node[fill=white,inner sep=2] at (s.corner 7) {${X_5}$};
\node[fill=white,inner sep=2] at (s.side 7) {${X_4'}$};
\end{tikzpicture}

%% file: tikz/P2_16.tex
\begin{tikzpicture}[font=\footnotesize]
\node[name=s, regular polygon, rotate=180, regular polygon sides=8,inner sep=2.169cm] at (0,0) {};
\draw[-,ultra thick,Turquoise] (s.center) to["6"] (s.side 1);
\draw[-,ultra thick,black] (s.center) to[""] (s.side 2);
\draw[-,ultra thick,Turquoise] (s.center) to["6"] (s.side 3);
\draw[-,ultra thick,RoyalBlue] (s.center) to["1"] (s.side 4);
\draw[-,ultra thick,Turquoise] (s.center) to["6"] (s.side 5);
\draw[-,ultra thick,black] (s.center) to[""] (s.side 6);
\draw[-,ultra thick,Turquoise] (s.center) to["6"] (s.side 7);
\draw[-,ultra thick,RoyalBlue] (s.center) to["1"] (s.side 8);
\draw[-,ultra thick,black] (s.side 1) to[swap,""] (s.corner 2);
\draw[-,ultra thick,RoyalBlue] (s.side 1) to["1"] (s.corner 1);
\draw[-,ultra thick,Turquoise] (s.side 2) to[swap,"6"] (s.corner 3);
\draw[-,ultra thick,Turquoise] (s.side 2) to["6"] (s.corner 2);
\draw[-,ultra thick,RoyalBlue] (s.side 3) to[swap,"1"] (s.corner 4);
\draw[-,ultra thick,black] (s.side 3) to[""] (s.corner 3);
\draw[-,ultra thick,Turquoise] (s.side 4) to[swap,"6"] (s.corner 5);
\draw[-,ultra thick,Turquoise] (s.side 4) to["6"] (s.corner 4);
\draw[-,ultra thick,black] (s.side 5) to[swap,""] (s.corner 6);
\draw[-,ultra thick,RoyalBlue] (s.side 5) to["1"] (s.corner 5);
\draw[-,ultra thick,Turquoise] (s.side 6) to[swap,"6"] (s.corner 7);
\draw[-,ultra thick,Turquoise] (s.side 6) to["6"] (s.corner 6);
\draw[-,ultra thick,RoyalBlue] (s.side 7) to[swap,"1"] (s.corner 8);
\draw[-,ultra thick,black] (s.side 7) to[""] (s.corner 7);
\draw[-,ultra thick,Turquoise] (s.side 8) to[swap,"6"] (s.corner 1);
\draw[-,ultra thick,Turquoise] (s.side 8) to["6"] (s.corner 8);
\node[fill=white,circle,inner sep=2] at (s.center) {${X_2}$};
\node[fill=white,inner sep=2] at (s.corner 1) {${\p^2}$};
\node[fill=white,inner sep=2] at (s.side 1) {${\F_1}$};
\node[fill=white,inner sep=2] at (s.corner 2) {${\F_1}_{\!/\p^1}$};
\node[fill=white,inner sep=2] at (s.side 2) {${X_2}_{/\p^1}$};
\node[fill=white,inner sep=2] at (s.corner 3) {${\F_1}_{\!/\p^1}$};
\node[fill=white,inner sep=2] at (s.side 3) {${\F_1}$};
\node[fill=white,inner sep=2] at (s.corner 4) {${\p^2}$};
\node[fill=white,inner sep=2] at (s.side 4) {${X_3}$};
\node[fill=white,inner sep=2] at (s.corner 5) {${\p^2}$};
\node[fill=white,inner sep=2] at (s.side 5) {${\F_1}$};
\node[fill=white,inner sep=2] at (s.corner 6) {${\F_1}_{\!/\p^1}$};
\node[fill=white,inner sep=2] at (s.side 6) {${X_2}_{/\p^1}$};
\node[fill=white,inner sep=2] at (s.corner 7) {${\F_1}_{\!/\p^1}$};
\node[fill=white,inner sep=2] at (s.side 7) {${\F_1}$};
\node[fill=white,inner sep=2] at (s.corner 8) {${\p^2}$};
\node[fill=white,inner sep=2] at (s.side 8) {${X_3}$};
\end{tikzpicture}

%% file: tikz/P2_26.tex
\begin{tikzpicture}[font=\footnotesize]
\node[name=s, regular polygon, rotate=180, regular polygon sides=8,inner sep=2.169cm] at (0,0) {};
\draw[-,ultra thick,Turquoise] (s.center) to["6"] (s.side 1);
\draw[-,ultra thick,RoyalBlue] (s.center) to["1"] (s.side 2);
\draw[-,ultra thick,Turquoise] (s.center) to["6"] (s.side 3);
\draw[-,ultra thick,FireBrick] (s.center) to["2"] (s.side 4);
\draw[-,ultra thick,Turquoise] (s.center) to["6"] (s.side 5);
\draw[-,ultra thick,RoyalBlue] (s.center) to["1"] (s.side 6);
\draw[-,ultra thick,Turquoise] (s.center) to["6"] (s.side 7);
\draw[-,ultra thick,FireBrick] (s.center) to["2"] (s.side 8);
\draw[-,ultra thick,RoyalBlue] (s.side 1) to[swap,"1"] (s.corner 2);
\draw[-,ultra thick,FireBrick] (s.side 1) to["2"] (s.corner 1);
\draw[-,ultra thick,Turquoise] (s.side 2) to[swap,"6"] (s.corner 3);
\draw[-,ultra thick,Turquoise] (s.side 2) to["6"] (s.corner 2);
\draw[-,ultra thick,FireBrick] (s.side 3) to[swap,"2"] (s.corner 4);
\draw[-,ultra thick,RoyalBlue] (s.side 3) to["1"] (s.corner 3);
\draw[-,ultra thick,Turquoise] (s.side 4) to[swap,"6"] (s.corner 5);
\draw[-,ultra thick,Turquoise] (s.side 4) to["6"] (s.corner 4);
\draw[-,ultra thick,RoyalBlue] (s.side 5) to[swap,"1"] (s.corner 6);
\draw[-,ultra thick,FireBrick] (s.side 5) to["2"] (s.corner 5);
\draw[-,ultra thick,Turquoise] (s.side 6) to[swap,"6"] (s.corner 7);
\draw[-,ultra thick,Turquoise] (s.side 6) to["6"] (s.corner 6);
\draw[-,ultra thick,FireBrick] (s.side 7) to[swap,"2"] (s.corner 8);
\draw[-,ultra thick,RoyalBlue] (s.side 7) to["1"] (s.corner 7);
\draw[-,ultra thick,Turquoise] (s.side 8) to[swap,"6"] (s.corner 1);
\draw[-,ultra thick,Turquoise] (s.side 8) to["6"] (s.corner 8);
\node[fill=white,circle,inner sep=2] at (s.center) {${X_1}$};
\node[fill=white,inner sep=2] at (s.corner 1) {${\p^2}$};
\node[fill=white,inner sep=2] at (s.side 1) {${X_7}$};
\node[fill=white,inner sep=2] at (s.corner 2) {${X_8}$};
\node[fill=white,inner sep=2] at (s.side 2) {${X_2}$};
\node[fill=white,inner sep=2] at (s.corner 3) {${X_8}$};
\node[fill=white,inner sep=2] at (s.side 3) {${X_7}$};
\node[fill=white,inner sep=2] at (s.corner 4) {${\p^2}$};
\node[fill=white,inner sep=2] at (s.side 4) {${X_3}$};
\node[fill=white,inner sep=2] at (s.corner 5) {${\p^2}$};
\node[fill=white,inner sep=2] at (s.side 5) {${X_7}$};
\node[fill=white,inner sep=2] at (s.corner 6) {${X_8}$};
\node[fill=white,inner sep=2] at (s.side 6) {${X_2}$};
\node[fill=white,inner sep=2] at (s.corner 7) {${X_8}$};
\node[fill=white,inner sep=2] at (s.side 7) {${X_7}$};
\node[fill=white,inner sep=2] at (s.corner 8) {${\p^2}$};
\node[fill=white,inner sep=2] at (s.side 8) {${X_3}$};
\end{tikzpicture}

%% file: tikz/P2_17.tex
\begin{tikzpicture}[font=\footnotesize]
\node[name=s, regular polygon, rotate=180, regular polygon sides=12,inner sep=3.357cm] at (0,0) {};
\draw[-,ultra thick,Salmon] (s.center) to["7"] (s.side 1);
\draw[-,ultra thick,black] (s.center) to[""] (s.side 2);
\draw[-,ultra thick,Salmon] (s.center) to["7"] (s.side 3);
\draw[-,ultra thick,black] (s.center) to[""] (s.side 4);
\draw[-,ultra thick,Salmon] (s.center) to["7"] (s.side 5);
\draw[-,ultra thick,RoyalBlue] (s.center) to["1"] (s.side 6);
\draw[-,ultra thick,Salmon] (s.center) to["7"] (s.side 7);
\draw[-,ultra thick,black] (s.center) to[""] (s.side 8);
\draw[-,ultra thick,Salmon] (s.center) to["7"] (s.side 9);
\draw[-,ultra thick,black] (s.center) to[""] (s.side 10);
\draw[-,ultra thick,Salmon] (s.center) to["7"] (s.side 11);
\draw[-,ultra thick,RoyalBlue] (s.center) to["1"] (s.side 12);
\draw[-,ultra thick,black] (s.side 1) to[swap,""] (s.corner 2);
\draw[-,ultra thick,RoyalBlue] (s.side 1) to["1"] (s.corner 1);
\draw[-,ultra thick,Salmon] (s.side 2) to[swap,"7"] (s.corner 3);
\draw[-,ultra thick,Salmon] (s.side 2) to["7"] (s.corner 2);
\draw[-,ultra thick,black] (s.side 3) to[swap,""] (s.corner 4);
\draw[-,ultra thick,black] (s.side 3) to[""] (s.corner 3);
\draw[-,ultra thick,Salmon] (s.side 4) to[swap,"7"] (s.corner 5);
\draw[-,ultra thick,Salmon] (s.side 4) to["7"] (s.corner 4);
\draw[-,ultra thick,RoyalBlue] (s.side 5) to[swap,"1"] (s.corner 6);
\draw[-,ultra thick,black] (s.side 5) to[""] (s.corner 5);
\draw[-,ultra thick,Salmon] (s.side 6) to[swap,"7"] (s.corner 7);
\draw[-,ultra thick,Salmon] (s.side 6) to["7"] (s.corner 6);
\draw[-,ultra thick,black] (s.side 7) to[swap,""] (s.corner 8);
\draw[-,ultra thick,RoyalBlue] (s.side 7) to["1"] (s.corner 7);
\draw[-,ultra thick,Salmon] (s.side 8) to[swap,"7"] (s.corner 9);
\draw[-,ultra thick,Salmon] (s.side 8) to["7"] (s.corner 8);
\draw[-,ultra thick,black] (s.side 9) to[swap,""] (s.corner 10);
\draw[-,ultra thick,black] (s.side 9) to[""] (s.corner 9);
\draw[-,ultra thick,Salmon] (s.side 10) to[swap,"7"] (s.corner 11);
\draw[-,ultra thick,Salmon] (s.side 10) to["7"] (s.corner 10);
\draw[-,ultra thick,RoyalBlue] (s.side 11) to[swap,"1"] (s.corner 12);
\draw[-,ultra thick,black] (s.side 11) to[""] (s.corner 11);
\draw[-,ultra thick,Salmon] (s.side 12) to[swap,"7"] (s.corner 1);
\draw[-,ultra thick,Salmon] (s.side 12) to["7"] (s.corner 12);
\node[fill=white,circle,inner sep=2] at (s.center) {${X_1}$};
\node[fill=white,inner sep=2] at (s.corner 1) {${\p^2}$};
\node[fill=white,inner sep=2] at (s.side 1) {${\F_1}$};
\node[fill=white,inner sep=2] at (s.corner 2) {${\F_1}_{\!/\p^1}$};
\node[fill=white,inner sep=2] at (s.side 2) {${X_1}_{/\p^1}$};
\node[fill=white,inner sep=2] at (s.corner 3) {${\F_0}_{\!/\p^1}$};
\node[fill=white,inner sep=2] at (s.side 3) {${\F_0}$};
\node[fill=white,inner sep=2] at (s.corner 4) {${\F_0}_{\!/\p^1}$};
\node[fill=white,inner sep=2] at (s.side 4) {${X_1}_{/\p^1}$};
\node[fill=white,inner sep=2] at (s.corner 5) {${\F_1}_{\!/\p^1}$};
\node[fill=white,inner sep=2] at (s.side 5) {${\F_1}$};
\node[fill=white,inner sep=2] at (s.corner 6) {${\p^2}$};
\node[fill=white,inner sep=2] at (s.side 6) {${X_2}$};
\node[fill=white,inner sep=2] at (s.corner 7) {${\p^2}$};
\node[fill=white,inner sep=2] at (s.side 7) {${\F_1}$};
\node[fill=white,inner sep=2] at (s.corner 8) {${\F_1}_{\!/\p^1}$};
\node[fill=white,inner sep=2] at (s.side 8) {${X_1}_{/\p^1}$};
\node[fill=white,inner sep=2] at (s.corner 9) {${\F_0}_{\!/\p^1}$};
\node[fill=white,inner sep=2] at (s.side 9) {${\F_0}$};
\node[fill=white,inner sep=2] at (s.corner 10) {${\F_0}_{\!/\p^1}$};
\node[fill=white,inner sep=2] at (s.side 10) {${X_1}_{/\p^1}$};
\node[fill=white,inner sep=2] at (s.corner 11) {${\F_1}_{\!/\p^1}$};
\node[fill=white,inner sep=2] at (s.side 11) {${\F_1}$};
\node[fill=white,inner sep=2] at (s.corner 12) {${\p^2}$};
\node[fill=white,inner sep=2] at (s.side 12) {${X_2}$};
\end{tikzpicture}

%% file: tikz/P2_22.tex
\begin{tikzpicture}[font=\footnotesize]
\node[name=s, regular polygon, rotate=180, regular polygon sides=5,inner sep=1.242cm] at (0,0) {};
\draw[-,ultra thick,FireBrick] (s.center) to["2"] (s.side 1);
\draw[-,ultra thick,RoyalBlue] (s.center) to["1"] (s.side 2);
\draw[-,ultra thick,black] (s.center) to[""] (s.side 3);
\draw[-,ultra thick,RoyalBlue] (s.center) to["1"] (s.side 4);
\draw[-,ultra thick,FireBrick] (s.center) to["2"] (s.side 5);
\draw[-,ultra thick,RoyalBlue] (s.side 1) to[swap,"1"] (s.corner 2);
\draw[-,ultra thick,FireBrick] (s.side 1) to["2"] (s.corner 1);
\draw[-,ultra thick,black] (s.side 2) to[swap,""] (s.corner 3);
\draw[-,ultra thick,FireBrick] (s.side 2) to["2"] (s.corner 2);
\draw[-,ultra thick,RoyalBlue] (s.side 3) to[swap,"1"] (s.corner 4);
\draw[-,ultra thick,RoyalBlue] (s.side 3) to["1"] (s.corner 3);
\draw[-,ultra thick,FireBrick] (s.side 4) to[swap,"2"] (s.corner 5);
\draw[-,ultra thick,black] (s.side 4) to[""] (s.corner 4);
\draw[-,ultra thick,FireBrick] (s.side 5) to[swap,"2"] (s.corner 1);
\draw[-,ultra thick,RoyalBlue] (s.side 5) to["1"] (s.corner 5);
\node[fill=white,circle,inner sep=2] at (s.center) {${X_5}$};
\node[fill=white,inner sep=2] at (s.corner 1) {${\p^2}$};
\node[fill=white,inner sep=2] at (s.side 1) {${X_7}$};
\node[fill=white,inner sep=2] at (s.corner 2) {${X_8}$};
\node[fill=white,inner sep=2] at (s.side 2) {${X_6}$};
\node[fill=white,inner sep=2] at (s.corner 3) {${X_6}_{\!/\p^1}$};
\node[fill=white,inner sep=2] at (s.side 3) {${X_5}_{/\p^1}$};
\node[fill=white,inner sep=2] at (s.corner 4) {${X_6'}_{\!/\p^1}$};
\node[fill=white,inner sep=2] at (s.side 4) {${X_6'}$};
\node[fill=white,inner sep=2] at (s.corner 5) {${X_8'}$};
\node[fill=white,inner sep=2] at (s.side 5) {${X_7'}$};
\end{tikzpicture}

%% file: tikz/P2_23.tex
\begin{tikzpicture}[font=\footnotesize]
\node[name=s, regular polygon, rotate=180, regular polygon sides=5,inner sep=1.242cm] at (0,0) {};
\draw[-,ultra thick,ForestGreen] (s.center) to["3"] (s.side 1);
\draw[-,ultra thick,RoyalBlue] (s.center) to["1"] (s.side 2);
\draw[-,ultra thick,RoyalBlue] (s.center) to["1"] (s.side 3);
\draw[-,ultra thick,ForestGreen] (s.center) to["3"] (s.side 4);
\draw[-,ultra thick,FireBrick] (s.center) to["2"] (s.side 5);
\draw[-,ultra thick,RoyalBlue] (s.side 1) to[swap,"1"] (s.corner 2);
\draw[-,ultra thick,FireBrick] (s.side 1) to["2"] (s.corner 1);
\draw[-,ultra thick,RoyalBlue] (s.side 2) to[swap,"1"] (s.corner 3);
\draw[-,ultra thick,ForestGreen] (s.side 2) to["3"] (s.corner 2);
\draw[-,ultra thick,ForestGreen] (s.side 3) to[swap,"3"] (s.corner 4);
\draw[-,ultra thick,RoyalBlue] (s.side 3) to["1"] (s.corner 3);
\draw[-,ultra thick,FireBrick] (s.side 4) to[swap,"2"] (s.corner 5);
\draw[-,ultra thick,RoyalBlue] (s.side 4) to["1"] (s.corner 4);
\draw[-,ultra thick,ForestGreen] (s.side 5) to[swap,"3"] (s.corner 1);
\draw[-,ultra thick,ForestGreen] (s.side 5) to["3"] (s.corner 5);
\node[fill=white,circle,inner sep=2] at (s.center) {${X_4}$};
\node[fill=white,inner sep=2] at (s.corner 1) {${\p^2}$};
\node[fill=white,inner sep=2] at (s.side 1) {${X_7}$};
\node[fill=white,inner sep=2] at (s.corner 2) {${X_8}$};
\node[fill=white,inner sep=2] at (s.side 2) {${X_5}$};
\node[fill=white,inner sep=2] at (s.corner 3) {${X_6}$};
\node[fill=white,inner sep=2] at (s.side 3) {${X_5}$};
\node[fill=white,inner sep=2] at (s.corner 4) {${X_8}$};
\node[fill=white,inner sep=2] at (s.side 4) {${X_7}$};
\node[fill=white,inner sep=2] at (s.corner 5) {${\p^2}$};
\node[fill=white,inner sep=2] at (s.side 5) {${X_6}$};
\end{tikzpicture}

%% file: tikz/P2_24.tex
\begin{tikzpicture}[font=\footnotesize]
\node[name=s, regular polygon, rotate=180, regular polygon sides=6,inner sep=1.557cm] at (0,0) {};
\draw[-,ultra thick,Orange] (s.center) to["4"] (s.side 1);
\draw[-,ultra thick,RoyalBlue] (s.center) to["1"] (s.side 2);
\draw[-,ultra thick,Orange] (s.center) to["4"] (s.side 3);
\draw[-,ultra thick,FireBrick] (s.center) to["2"] (s.side 4);
\draw[-,ultra thick,black] (s.center) to[""] (s.side 5);
\draw[-,ultra thick,FireBrick] (s.center) to["2"] (s.side 6);
\draw[-,ultra thick,RoyalBlue] (s.side 1) to[swap,"1"] (s.corner 2);
\draw[-,ultra thick,FireBrick] (s.side 1) to["2"] (s.corner 1);
\draw[-,ultra thick,Orange] (s.side 2) to[swap,"4"] (s.corner 3);
\draw[-,ultra thick,Orange] (s.side 2) to["4"] (s.corner 2);
\draw[-,ultra thick,FireBrick] (s.side 3) to[swap,"2"] (s.corner 4);
\draw[-,ultra thick,RoyalBlue] (s.side 3) to["1"] (s.corner 3);
\draw[-,ultra thick,black] (s.side 4) to[swap,""] (s.corner 5);
\draw[-,ultra thick,Orange] (s.side 4) to["4"] (s.corner 4);
\draw[-,ultra thick,FireBrick] (s.side 5) to[swap,"2"] (s.corner 6);
\draw[-,ultra thick,FireBrick] (s.side 5) to["2"] (s.corner 5);
\draw[-,ultra thick,Orange] (s.side 6) to[swap,"4"] (s.corner 1);
\draw[-,ultra thick,black] (s.side 6) to[""] (s.corner 6);
\node[fill=white,circle,inner sep=2] at (s.center) {${X_3}$};
\node[fill=white,inner sep=2] at (s.corner 1) {${\p^2}$};
\node[fill=white,inner sep=2] at (s.side 1) {${X_7}$};
\node[fill=white,inner sep=2] at (s.corner 2) {${X_8}$};
\node[fill=white,inner sep=2] at (s.side 2) {${X_4}$};
\node[fill=white,inner sep=2] at (s.corner 3) {${X_8'}$};
\node[fill=white,inner sep=2] at (s.side 3) {${X_7'}$};
\node[fill=white,inner sep=2] at (s.corner 4) {${\p^2}$};
\node[fill=white,inner sep=2] at (s.side 4) {${X_5}$};
\node[fill=white,inner sep=2] at (s.corner 5) {${X_5}_{\!/\p^1}$};
\node[fill=white,inner sep=2] at (s.side 5) {${X_3}_{/\p^1}$};
\node[fill=white,inner sep=2] at (s.corner 6) {${X_5}_{\!/\p^1}$};
\node[fill=white,inner sep=2] at (s.side 6) {${X_5}$};
\end{tikzpicture}

%% file: tikz/P2_25.tex
\begin{tikzpicture}[font=\footnotesize]
\node[name=s, regular polygon, rotate=180, regular polygon sides=6,inner sep=1.557cm] at (0,0) {};
\draw[-,ultra thick,HotPink] (s.center) to["5"] (s.side 1);
\draw[-,ultra thick,RoyalBlue] (s.center) to["1"] (s.side 2);
\draw[-,ultra thick,FireBrick] (s.center) to["2"] (s.side 3);
\draw[-,ultra thick,HotPink] (s.center) to["5"] (s.side 4);
\draw[-,ultra thick,RoyalBlue] (s.center) to["1"] (s.side 5);
\draw[-,ultra thick,FireBrick] (s.center) to["2"] (s.side 6);
\draw[-,ultra thick,RoyalBlue] (s.side 1) to[swap,"1"] (s.corner 2);
\draw[-,ultra thick,FireBrick] (s.side 1) to["2"] (s.corner 1);
\draw[-,ultra thick,FireBrick] (s.side 2) to[swap,"2"] (s.corner 3);
\draw[-,ultra thick,HotPink] (s.side 2) to["5"] (s.corner 2);
\draw[-,ultra thick,HotPink] (s.side 3) to[swap,"5"] (s.corner 4);
\draw[-,ultra thick,RoyalBlue] (s.side 3) to["1"] (s.corner 3);
\draw[-,ultra thick,RoyalBlue] (s.side 4) to[swap,"1"] (s.corner 5);
\draw[-,ultra thick,FireBrick] (s.side 4) to["2"] (s.corner 4);
\draw[-,ultra thick,FireBrick] (s.side 5) to[swap,"2"] (s.corner 6);
\draw[-,ultra thick,HotPink] (s.side 5) to["5"] (s.corner 5);
\draw[-,ultra thick,HotPink] (s.side 6) to[swap,"5"] (s.corner 1);
\draw[-,ultra thick,RoyalBlue] (s.side 6) to["1"] (s.corner 6);
\node[fill=white,circle,inner sep=2] at (s.center) {${X_2}$};
\node[fill=white,inner sep=2] at (s.corner 1) {${\p^2}$};
\node[fill=white,inner sep=2] at (s.side 1) {${X_7}$};
\node[fill=white,inner sep=2] at (s.corner 2) {${X_8}$};
\node[fill=white,inner sep=2] at (s.side 2) {${X_3}$};
\node[fill=white,inner sep=2] at (s.corner 3) {${X_5}$};
\node[fill=white,inner sep=2] at (s.side 3) {${X_4}$};
\node[fill=white,inner sep=2] at (s.corner 4) {${\p^2}$};
\node[fill=white,inner sep=2] at (s.side 4) {${X_7}$};
\node[fill=white,inner sep=2] at (s.corner 5) {${X_8}$};
\node[fill=white,inner sep=2] at (s.side 5) {${X_3}$};
\node[fill=white,inner sep=2] at (s.corner 6) {${X_5}$};
\node[fill=white,inner sep=2] at (s.side 6) {${X_4}$};
\end{tikzpicture}

%% file: tikz/P2_33.tex
\begin{tikzpicture}[font=\footnotesize]
\node[name=s, regular polygon, rotate=180, regular polygon sides=6,inner sep=1.557cm] at (0,0) {};
\draw[-,ultra thick,ForestGreen] (s.center) to["3"] (s.side 1);
\draw[-,ultra thick,ForestGreen] (s.center) to["3"] (s.side 2);
\draw[-,ultra thick,ForestGreen] (s.center) to["3"] (s.side 3);
\draw[-,ultra thick,ForestGreen] (s.center) to["3"] (s.side 4);
\draw[-,ultra thick,ForestGreen] (s.center) to["3"] (s.side 5);
\draw[-,ultra thick,ForestGreen] (s.center) to["3"] (s.side 6);
\draw[-,ultra thick,ForestGreen] (s.side 1) to[swap,"3"] (s.corner 2);
\draw[-,ultra thick,ForestGreen] (s.side 1) to["3"] (s.corner 1);
\draw[-,ultra thick,ForestGreen] (s.side 2) to[swap,"3"] (s.corner 3);
\draw[-,ultra thick,ForestGreen] (s.side 2) to["3"] (s.corner 2);
\draw[-,ultra thick,ForestGreen] (s.side 3) to[swap,"3"] (s.corner 4);
\draw[-,ultra thick,ForestGreen] (s.side 3) to["3"] (s.corner 3);
\draw[-,ultra thick,ForestGreen] (s.side 4) to[swap,"3"] (s.corner 5);
\draw[-,ultra thick,ForestGreen] (s.side 4) to["3"] (s.corner 4);
\draw[-,ultra thick,ForestGreen] (s.side 5) to[swap,"3"] (s.corner 6);
\draw[-,ultra thick,ForestGreen] (s.side 5) to["3"] (s.corner 5);
\draw[-,ultra thick,ForestGreen] (s.side 6) to[swap,"3"] (s.corner 1);
\draw[-,ultra thick,ForestGreen] (s.side 6) to["3"] (s.corner 6);
\node[fill=white,circle,inner sep=2] at (s.center) {${X_3}$};
\node[fill=white,inner sep=2] at (s.corner 1) {${\p^2}$};
\node[fill=white,inner sep=2] at (s.side 1) {${X_6}$};
\node[fill=white,inner sep=2] at (s.corner 2) {${\p^2}$};
\node[fill=white,inner sep=2] at (s.side 2) {${X_6}$};
\node[fill=white,inner sep=2] at (s.corner 3) {${\p^2}$};
\node[fill=white,inner sep=2] at (s.side 3) {${X_6}$};
\node[fill=white,inner sep=2] at (s.corner 4) {${\p^2}$};
\node[fill=white,inner sep=2] at (s.side 4) {${X_6}$};
\node[fill=white,inner sep=2] at (s.corner 5) {${\p^2}$};
\node[fill=white,inner sep=2] at (s.side 5) {${X_6}$};
\node[fill=white,inner sep=2] at (s.corner 6) {${\p^2}$};
\node[fill=white,inner sep=2] at (s.side 6) {${X_6}$};
\end{tikzpicture}

%% file: tikz/X8_33.tex
\begin{tikzpicture}[font=\footnotesize]
\node[name=s, regular polygon, rotate=180, regular polygon sides=6,inner sep=1.557cm] at (0,0) {};
\draw[-,ultra thick,ForestGreen] (s.center) to["3"] (s.side 1);
\draw[-,ultra thick,RoyalBlue] (s.center) to["1"] (s.side 2);
\draw[-,ultra thick,ForestGreen] (s.center) to["3"] (s.side 3);
\draw[-,ultra thick,ForestGreen] (s.center) to["3"] (s.side 4);
\draw[-,ultra thick,RoyalBlue] (s.center) to["1"] (s.side 5);
\draw[-,ultra thick,ForestGreen] (s.center) to["3"] (s.side 6);
\draw[-,ultra thick,RoyalBlue] (s.side 1) to[swap,"1"] (s.corner 2);
\draw[-,ultra thick,ForestGreen] (s.side 1) to["3"] (s.corner 1);
\draw[-,ultra thick,ForestGreen] (s.side 2) to[swap,"3"] (s.corner 3);
\draw[-,ultra thick,ForestGreen] (s.side 2) to["3"] (s.corner 2);
\draw[-,ultra thick,ForestGreen] (s.side 3) to[swap,"3"] (s.corner 4);
\draw[-,ultra thick,RoyalBlue] (s.side 3) to["1"] (s.corner 3);
\draw[-,ultra thick,RoyalBlue] (s.side 4) to[swap,"1"] (s.corner 5);
\draw[-,ultra thick,ForestGreen] (s.side 4) to["3"] (s.corner 4);
\draw[-,ultra thick,ForestGreen] (s.side 5) to[swap,"3"] (s.corner 6);
\draw[-,ultra thick,ForestGreen] (s.side 5) to["3"] (s.corner 5);
\draw[-,ultra thick,ForestGreen] (s.side 6) to[swap,"3"] (s.corner 1);
\draw[-,ultra thick,RoyalBlue] (s.side 6) to["1"] (s.corner 6);
\node[fill=white,circle,inner sep=2] at (s.center) {${X_2}$};
\node[fill=white,inner sep=2] at (s.corner 1) {${X_8}$};
\node[fill=white,inner sep=2] at (s.side 1) {${X_5}$};
\node[fill=white,inner sep=2] at (s.corner 2) {${X_6}$};
\node[fill=white,inner sep=2] at (s.side 2) {${X_3}$};
\node[fill=white,inner sep=2] at (s.corner 3) {${X_6'}$};
\node[fill=white,inner sep=2] at (s.side 3) {${X_5'}$};
\node[fill=white,inner sep=2] at (s.corner 4) {${X_8}$};
\node[fill=white,inner sep=2] at (s.side 4) {${X_5}$};
\node[fill=white,inner sep=2] at (s.corner 5) {${X_6}$};
\node[fill=white,inner sep=2] at (s.side 5) {${X_3}$};
\node[fill=white,inner sep=2] at (s.corner 6) {${X_6'}$};
\node[fill=white,inner sep=2] at (s.side 6) {${X_5'}$};
\end{tikzpicture}

%% file: tikz/P2_34.tex
\begin{tikzpicture}[font=\footnotesize]
\node[name=s, regular polygon, rotate=180, regular polygon sides=8,inner sep=2.169cm] at (0,0) {};
\draw[-,ultra thick,Orange] (s.center) to["4"] (s.side 1);
\draw[-,ultra thick,ForestGreen] (s.center) to["3"] (s.side 2);
\draw[-,ultra thick,black] (s.center) to[""] (s.side 3);
\draw[-,ultra thick,ForestGreen] (s.center) to["3"] (s.side 4);
\draw[-,ultra thick,Orange] (s.center) to["4"] (s.side 5);
\draw[-,ultra thick,ForestGreen] (s.center) to["3"] (s.side 6);
\draw[-,ultra thick,black] (s.center) to[""] (s.side 7);
\draw[-,ultra thick,ForestGreen] (s.center) to["3"] (s.side 8);
\draw[-,ultra thick,ForestGreen] (s.side 1) to[swap,"3"] (s.corner 2);
\draw[-,ultra thick,ForestGreen] (s.side 1) to["3"] (s.corner 1);
\draw[-,ultra thick,black] (s.side 2) to[swap,""] (s.corner 3);
\draw[-,ultra thick,Orange] (s.side 2) to["4"] (s.corner 2);
\draw[-,ultra thick,ForestGreen] (s.side 3) to[swap,"3"] (s.corner 4);
\draw[-,ultra thick,ForestGreen] (s.side 3) to["3"] (s.corner 3);
\draw[-,ultra thick,Orange] (s.side 4) to[swap,"4"] (s.corner 5);
\draw[-,ultra thick,black] (s.side 4) to[""] (s.corner 4);
\draw[-,ultra thick,ForestGreen] (s.side 5) to[swap,"3"] (s.corner 6);
\draw[-,ultra thick,ForestGreen] (s.side 5) to["3"] (s.corner 5);
\draw[-,ultra thick,black] (s.side 6) to[swap,""] (s.corner 7);
\draw[-,ultra thick,Orange] (s.side 6) to["4"] (s.corner 6);
\draw[-,ultra thick,ForestGreen] (s.side 7) to[swap,"3"] (s.corner 8);
\draw[-,ultra thick,ForestGreen] (s.side 7) to["3"] (s.corner 7);
\draw[-,ultra thick,Orange] (s.side 8) to[swap,"4"] (s.corner 1);
\draw[-,ultra thick,black] (s.side 8) to[""] (s.corner 8);
\node[fill=white,circle,inner sep=2] at (s.center) {${X_2}$};
\node[fill=white,inner sep=2] at (s.corner 1) {${\p^2}$};
\node[fill=white,inner sep=2] at (s.side 1) {${X_6}$};
\node[fill=white,inner sep=2] at (s.corner 2) {${\p^2}$};
\node[fill=white,inner sep=2] at (s.side 2) {${X_5}$};
\node[fill=white,inner sep=2] at (s.corner 3) {${X_5}_{\!/\p^1}$};
\node[fill=white,inner sep=2] at (s.side 3) {${X_2}_{/\p^1}$};
\node[fill=white,inner sep=2] at (s.corner 4) {${X_5}_{\!/\p^1}$};
\node[fill=white,inner sep=2] at (s.side 4) {${X_5}$};
\node[fill=white,inner sep=2] at (s.corner 5) {${\p^2}$};
\node[fill=white,inner sep=2] at (s.side 5) {${X_6}$};
\node[fill=white,inner sep=2] at (s.corner 6) {${\p^2}$};
\node[fill=white,inner sep=2] at (s.side 6) {${X_5}$};
\node[fill=white,inner sep=2] at (s.corner 7) {${X_5}_{\!/\p^1}$};
\node[fill=white,inner sep=2] at (s.side 7) {${X_2}_{/\p^1}$};
\node[fill=white,inner sep=2] at (s.corner 8) {${X_5}_{\!/\p^1}$};
\node[fill=white,inner sep=2] at (s.side 8) {${X_5}$};
\end{tikzpicture}

%% file: tikz/P2_35.tex
\begin{tikzpicture}[font=\footnotesize]
\node[name=s, regular polygon, rotate=180, regular polygon sides=8,inner sep=2.169cm] at (0,0) {};
\draw[-,ultra thick,HotPink] (s.center) to["5"] (s.side 1);
\draw[-,ultra thick,ForestGreen] (s.center) to["3"] (s.side 2);
\draw[-,ultra thick,RoyalBlue] (s.center) to["1"] (s.side 3);
\draw[-,ultra thick,ForestGreen] (s.center) to["3"] (s.side 4);
\draw[-,ultra thick,HotPink] (s.center) to["5"] (s.side 5);
\draw[-,ultra thick,ForestGreen] (s.center) to["3"] (s.side 6);
\draw[-,ultra thick,RoyalBlue] (s.center) to["1"] (s.side 7);
\draw[-,ultra thick,ForestGreen] (s.center) to["3"] (s.side 8);
\draw[-,ultra thick,ForestGreen] (s.side 1) to[swap,"3"] (s.corner 2);
\draw[-,ultra thick,ForestGreen] (s.side 1) to["3"] (s.corner 1);
\draw[-,ultra thick,RoyalBlue] (s.side 2) to[swap,"1"] (s.corner 3);
\draw[-,ultra thick,HotPink] (s.side 2) to["5"] (s.corner 2);
\draw[-,ultra thick,ForestGreen] (s.side 3) to[swap,"3"] (s.corner 4);
\draw[-,ultra thick,ForestGreen] (s.side 3) to["3"] (s.corner 3);
\draw[-,ultra thick,HotPink] (s.side 4) to[swap,"5"] (s.corner 5);
\draw[-,ultra thick,RoyalBlue] (s.side 4) to["1"] (s.corner 4);
\draw[-,ultra thick,ForestGreen] (s.side 5) to[swap,"3"] (s.corner 6);
\draw[-,ultra thick,ForestGreen] (s.side 5) to["3"] (s.corner 5);
\draw[-,ultra thick,RoyalBlue] (s.side 6) to[swap,"1"] (s.corner 7);
\draw[-,ultra thick,HotPink] (s.side 6) to["5"] (s.corner 6);
\draw[-,ultra thick,ForestGreen] (s.side 7) to[swap,"3"] (s.corner 8);
\draw[-,ultra thick,ForestGreen] (s.side 7) to["3"] (s.corner 7);
\draw[-,ultra thick,HotPink] (s.side 8) to[swap,"5"] (s.corner 1);
\draw[-,ultra thick,RoyalBlue] (s.side 8) to["1"] (s.corner 8);
\node[fill=white,circle,inner sep=2] at (s.center) {${X_1}$};
\node[fill=white,inner sep=2] at (s.corner 1) {${\p^2}$};
\node[fill=white,inner sep=2] at (s.side 1) {${X_6}$};
\node[fill=white,inner sep=2] at (s.corner 2) {${\p^2}$};
\node[fill=white,inner sep=2] at (s.side 2) {${X_4}$};
\node[fill=white,inner sep=2] at (s.corner 3) {${X_5}$};
\node[fill=white,inner sep=2] at (s.side 3) {${X_2}$};
\node[fill=white,inner sep=2] at (s.corner 4) {${X_5}$};
\node[fill=white,inner sep=2] at (s.side 4) {${X_4}$};
\node[fill=white,inner sep=2] at (s.corner 5) {${\p^2}$};
\node[fill=white,inner sep=2] at (s.side 5) {${X_6}$};
\node[fill=white,inner sep=2] at (s.corner 6) {${\p^2}$};
\node[fill=white,inner sep=2] at (s.side 6) {${X_4}$};
\node[fill=white,inner sep=2] at (s.corner 7) {${X_5}$};
\node[fill=white,inner sep=2] at (s.side 7) {${X_2}$};
\node[fill=white,inner sep=2] at (s.corner 8) {${X_5}$};
\node[fill=white,inner sep=2] at (s.side 8) {${X_4}$};
\end{tikzpicture}

%% file: tikz/P2_44.tex
\begin{tikzpicture}[font=\footnotesize]
\node[name=s, regular polygon, rotate=180, regular polygon sides=12,inner sep=3.357cm] at (0,0) {};
\draw[-,ultra thick,Orange] (s.center) to["4"] (s.side 1);
\draw[-,ultra thick,black] (s.center) to[""] (s.side 2);
\draw[-,ultra thick,Orange] (s.center) to["4"] (s.side 3);
\draw[-,ultra thick,Orange] (s.center) to["4"] (s.side 4);
\draw[-,ultra thick,black] (s.center) to[""] (s.side 5);
\draw[-,ultra thick,Orange] (s.center) to["4"] (s.side 6);
\draw[-,ultra thick,Orange] (s.center) to["4"] (s.side 7);
\draw[-,ultra thick,black] (s.center) to[""] (s.side 8);
\draw[-,ultra thick,Orange] (s.center) to["4"] (s.side 9);
\draw[-,ultra thick,Orange] (s.center) to["4"] (s.side 10);
\draw[-,ultra thick,black] (s.center) to[""] (s.side 11);
\draw[-,ultra thick,Orange] (s.center) to["4"] (s.side 12);
\draw[-,ultra thick,black] (s.side 1) to[swap,""] (s.corner 2);
\draw[-,ultra thick,Orange] (s.side 1) to["4"] (s.corner 1);
\draw[-,ultra thick,Orange] (s.side 2) to[swap,"4"] (s.corner 3);
\draw[-,ultra thick,Orange] (s.side 2) to["4"] (s.corner 2);
\draw[-,ultra thick,Orange] (s.side 3) to[swap,"4"] (s.corner 4);
\draw[-,ultra thick,black] (s.side 3) to[""] (s.corner 3);
\draw[-,ultra thick,black] (s.side 4) to[swap,""] (s.corner 5);
\draw[-,ultra thick,Orange] (s.side 4) to["4"] (s.corner 4);
\draw[-,ultra thick,Orange] (s.side 5) to[swap,"4"] (s.corner 6);
\draw[-,ultra thick,Orange] (s.side 5) to["4"] (s.corner 5);
\draw[-,ultra thick,Orange] (s.side 6) to[swap,"4"] (s.corner 7);
\draw[-,ultra thick,black] (s.side 6) to[""] (s.corner 6);
\draw[-,ultra thick,black] (s.side 7) to[swap,""] (s.corner 8);
\draw[-,ultra thick,Orange] (s.side 7) to["4"] (s.corner 7);
\draw[-,ultra thick,Orange] (s.side 8) to[swap,"4"] (s.corner 9);
\draw[-,ultra thick,Orange] (s.side 8) to["4"] (s.corner 8);
\draw[-,ultra thick,Orange] (s.side 9) to[swap,"4"] (s.corner 10);
\draw[-,ultra thick,black] (s.side 9) to[""] (s.corner 9);
\draw[-,ultra thick,black] (s.side 10) to[swap,""] (s.corner 11);
\draw[-,ultra thick,Orange] (s.side 10) to["4"] (s.corner 10);
\draw[-,ultra thick,Orange] (s.side 11) to[swap,"4"] (s.corner 12);
\draw[-,ultra thick,Orange] (s.side 11) to["4"] (s.corner 11);
\draw[-,ultra thick,Orange] (s.side 12) to[swap,"4"] (s.corner 1);
\draw[-,ultra thick,black] (s.side 12) to[""] (s.corner 12);
\node[fill=white,circle,inner sep=2] at (s.center) {${X_1}$};
\node[fill=white,inner sep=2] at (s.corner 1) {${\p^2}$};
\node[fill=white,inner sep=2] at (s.side 1) {${X_5}$};
\node[fill=white,inner sep=2] at (s.corner 2) {${X_5}_{\!/\p^1}$};
\node[fill=white,inner sep=2] at (s.side 2) {${X_1}_{/\p^1}$};
\node[fill=white,inner sep=2] at (s.corner 3) {${X_5}_{\!/\p^1}$};
\node[fill=white,inner sep=2] at (s.side 3) {${X_5}$};
\node[fill=white,inner sep=2] at (s.corner 4) {${\p^2}$};
\node[fill=white,inner sep=2] at (s.side 4) {${X_5}$};
\node[fill=white,inner sep=2] at (s.corner 5) {${X_5}_{\!/\p^1}$};
\node[fill=white,inner sep=2] at (s.side 5) {${X_1}_{/\p^1}$};
\node[fill=white,inner sep=2] at (s.corner 6) {${X_5}_{\!/\p^1}$};
\node[fill=white,inner sep=2] at (s.side 6) {${X_5}$};
\node[fill=white,inner sep=2] at (s.corner 7) {${\p^2}$};
\node[fill=white,inner sep=2] at (s.side 7) {${X_5}$};
\node[fill=white,inner sep=2] at (s.corner 8) {${X_5}_{\!/\p^1}$};
\node[fill=white,inner sep=2] at (s.side 8) {${X_1}_{/\p^1}$};
\node[fill=white,inner sep=2] at (s.corner 9) {${X_5}_{\!/\p^1}$};
\node[fill=white,inner sep=2] at (s.side 9) {${X_5}$};
\node[fill=white,inner sep=2] at (s.corner 10) {${\p^2}$};
\node[fill=white,inner sep=2] at (s.side 10) {${X_5}$};
\node[fill=white,inner sep=2] at (s.corner 11) {${X_5}_{\!/\p^1}$};
\node[fill=white,inner sep=2] at (s.side 11) {${X_1}_{/\p^1}$};
\node[fill=white,inner sep=2] at (s.corner 12) {${X_5}_{\!/\p^1}$};
\node[fill=white,inner sep=2] at (s.side 12) {${X_5}$};
\end{tikzpicture}

%% file: tikz/X8_22.tex
\begin{tikzpicture}[font=\footnotesize]
\node[name=s, regular polygon, rotate=180, regular polygon sides=6,inner sep=1.557cm] at (0,0) {};
\draw[-,ultra thick,FireBrick] (s.center) to["2"] (s.side 1);
\draw[-,ultra thick,black] (s.center) to[""] (s.side 2);
\draw[-,ultra thick,FireBrick] (s.center) to["2"] (s.side 3);
\draw[-,ultra thick,FireBrick] (s.center) to["2"] (s.side 4);
\draw[-,ultra thick,black] (s.center) to[""] (s.side 5);
\draw[-,ultra thick,FireBrick] (s.center) to["2"] (s.side 6);
\draw[-,ultra thick,black] (s.side 1) to[swap,""] (s.corner 2);
\draw[-,ultra thick,FireBrick] (s.side 1) to["2"] (s.corner 1);
\draw[-,ultra thick,FireBrick] (s.side 2) to[swap,"2"] (s.corner 3);
\draw[-,ultra thick,FireBrick] (s.side 2) to["2"] (s.corner 2);
\draw[-,ultra thick,FireBrick] (s.side 3) to[swap,"2"] (s.corner 4);
\draw[-,ultra thick,black] (s.side 3) to[""] (s.corner 3);
\draw[-,ultra thick,black] (s.side 4) to[swap,""] (s.corner 5);
\draw[-,ultra thick,FireBrick] (s.side 4) to["2"] (s.corner 4);
\draw[-,ultra thick,FireBrick] (s.side 5) to[swap,"2"] (s.corner 6);
\draw[-,ultra thick,FireBrick] (s.side 5) to["2"] (s.corner 5);
\draw[-,ultra thick,FireBrick] (s.side 6) to[swap,"2"] (s.corner 1);
\draw[-,ultra thick,black] (s.side 6) to[""] (s.corner 6);
\node[fill=white,circle,inner sep=2] at (s.center) {${X_4}$};
\node[fill=white,inner sep=2] at (s.corner 1) {${X_8}$};
\node[fill=white,inner sep=2] at (s.side 1) {${X_6}$};
\node[fill=white,inner sep=2] at (s.corner 2) {${X_6}_{\!/\p^1}$};
\node[fill=white,inner sep=2] at (s.side 2) {${X_4}_{/\p^1}$};
\node[fill=white,inner sep=2] at (s.corner 3) {${X_6'}_{\!/\p^1}$};
\node[fill=white,inner sep=2] at (s.side 3) {${X_6'}$};
\node[fill=white,inner sep=2] at (s.corner 4) {${X_8'}$};
\node[fill=white,inner sep=2] at (s.side 4) {${X_6''}$};
\node[fill=white,inner sep=2] at (s.corner 5) {${X_6''}_{\!/\p^1}$};
\node[fill=white,inner sep=2] at (s.side 5) {${X_4'}$};
\node[fill=white,inner sep=2] at (s.corner 6) {${X_6'''}_{\!/\p^1}$};
\node[fill=white,inner sep=2] at (s.side 6) {${X_6'''}$};
\end{tikzpicture}

%% file: tikz/X8_23.tex
\begin{tikzpicture}[font=\footnotesize]
\node[name=s, regular polygon, rotate=180, regular polygon sides=6,inner sep=1.557cm] at (0,0) {};
\draw[-,ultra thick,ForestGreen] (s.center) to["3"] (s.side 1);
\draw[-,ultra thick,black] (s.center) to[""] (s.side 2);
\draw[-,ultra thick,ForestGreen] (s.center) to["3"] (s.side 3);
\draw[-,ultra thick,FireBrick] (s.center) to["2"] (s.side 4);
\draw[-,ultra thick,RoyalBlue] (s.center) to["1"] (s.side 5);
\draw[-,ultra thick,FireBrick] (s.center) to["2"] (s.side 6);
\draw[-,ultra thick,black] (s.side 1) to[swap,""] (s.corner 2);
\draw[-,ultra thick,FireBrick] (s.side 1) to["2"] (s.corner 1);
\draw[-,ultra thick,ForestGreen] (s.side 2) to[swap,"3"] (s.corner 3);
\draw[-,ultra thick,ForestGreen] (s.side 2) to["3"] (s.corner 2);
\draw[-,ultra thick,FireBrick] (s.side 3) to[swap,"2"] (s.corner 4);
\draw[-,ultra thick,black] (s.side 3) to[""] (s.corner 3);
\draw[-,ultra thick,RoyalBlue] (s.side 4) to[swap,"1"] (s.corner 5);
\draw[-,ultra thick,ForestGreen] (s.side 4) to["3"] (s.corner 4);
\draw[-,ultra thick,FireBrick] (s.side 5) to[swap,"2"] (s.corner 6);
\draw[-,ultra thick,FireBrick] (s.side 5) to["2"] (s.corner 5);
\draw[-,ultra thick,ForestGreen] (s.side 6) to[swap,"3"] (s.corner 1);
\draw[-,ultra thick,RoyalBlue] (s.side 6) to["1"] (s.corner 6);
\node[fill=white,circle,inner sep=2] at (s.center) {${X_3}$};
\node[fill=white,inner sep=2] at (s.corner 1) {${X_8}$};
\node[fill=white,inner sep=2] at (s.side 1) {${X_6}$};
\node[fill=white,inner sep=2] at (s.corner 2) {${X_6}_{\!/\p^1}$};
\node[fill=white,inner sep=2] at (s.side 2) {${X_3}_{/\p^1}$};
\node[fill=white,inner sep=2] at (s.corner 3) {${X_6'}_{\!/\p^1}$};
\node[fill=white,inner sep=2] at (s.side 3) {${X_6'}$};
\node[fill=white,inner sep=2] at (s.corner 4) {${X_8'}$};
\node[fill=white,inner sep=2] at (s.side 4) {${X_5'}$};
\node[fill=white,inner sep=2] at (s.corner 5) {${X_6''}$};
\node[fill=white,inner sep=2] at (s.side 5) {${X_4}$};
\node[fill=white,inner sep=2] at (s.corner 6) {${X_6'''}$};
\node[fill=white,inner sep=2] at (s.side 6) {${X_5}$};
\end{tikzpicture}

%% file: tikz/X6_22.tex
\begin{tikzpicture}[font=\footnotesize]
\node[name=s, regular polygon, rotate=180, regular polygon sides=4,inner sep=0.9cm] at (0,0) {};
\draw[-,ultra thick,FireBrick] (s.center) to["2"] (s.side 1);
\draw[-,ultra thick,FireBrick] (s.center) to["2"] (s.side 2);
\draw[-,ultra thick,FireBrick] (s.center) to["2"] (s.side 3);
\draw[-,ultra thick,FireBrick] (s.center) to["2"] (s.side 4);
\draw[-,ultra thick,FireBrick] (s.side 1) to[swap,"2"] (s.corner 2);
\draw[-,ultra thick,FireBrick] (s.side 1) to["2"] (s.corner 1);
\draw[-,ultra thick,FireBrick] (s.side 2) to[swap,"2"] (s.corner 3);
\draw[-,ultra thick,FireBrick] (s.side 2) to["2"] (s.corner 2);
\draw[-,ultra thick,FireBrick] (s.side 3) to[swap,"2"] (s.corner 4);
\draw[-,ultra thick,FireBrick] (s.side 3) to["2"] (s.corner 3);
\draw[-,ultra thick,FireBrick] (s.side 4) to[swap,"2"] (s.corner 1);
\draw[-,ultra thick,FireBrick] (s.side 4) to["2"] (s.corner 4);
\node[fill=white,circle,inner sep=2] at (s.center) {${X_2}$};
\node[fill=white,inner sep=2] at (s.corner 1) {${X_6}$};
\node[fill=white,inner sep=2] at (s.side 1) {${X_4}$};
\node[fill=white,inner sep=2] at (s.corner 2) {${X_6'}$};
\node[fill=white,inner sep=2] at (s.side 2) {${X_4'}$};
\node[fill=white,inner sep=2] at (s.corner 3) {${X_6}$};
\node[fill=white,inner sep=2] at (s.side 3) {${X_4}$};
\node[fill=white,inner sep=2] at (s.corner 4) {${X_6'}$};
\node[fill=white,inner sep=2] at (s.side 4) {${X_4'}$};
\end{tikzpicture}

%% file: tikz/X6_23.tex
\begin{tikzpicture}[font=\footnotesize]
\node[name=s, regular polygon, rotate=180, regular polygon sides=4,inner sep=0.9cm] at (0,0) {};
\draw[-,ultra thick,ForestGreen] (s.center) to["3"] (s.side 1);
\draw[-,ultra thick,FireBrick] (s.center) to["2"] (s.side 2);
\draw[-,ultra thick,ForestGreen] (s.center) to["3"] (s.side 3);
\draw[-,ultra thick,FireBrick] (s.center) to["2"] (s.side 4);
\draw[-,ultra thick,FireBrick] (s.side 1) to[swap,"2"] (s.corner 2);
\draw[-,ultra thick,FireBrick] (s.side 1) to["2"] (s.corner 1);
\draw[-,ultra thick,ForestGreen] (s.side 2) to[swap,"3"] (s.corner 3);
\draw[-,ultra thick,ForestGreen] (s.side 2) to["3"] (s.corner 2);
\draw[-,ultra thick,FireBrick] (s.side 3) to[swap,"2"] (s.corner 4);
\draw[-,ultra thick,FireBrick] (s.side 3) to["2"] (s.corner 3);
\draw[-,ultra thick,ForestGreen] (s.side 4) to[swap,"3"] (s.corner 1);
\draw[-,ultra thick,ForestGreen] (s.side 4) to["3"] (s.corner 4);
\node[fill=white,circle,inner sep=2] at (s.center) {${X_1}$};
\node[fill=white,inner sep=2] at (s.corner 1) {${X_6}$};
\node[fill=white,inner sep=2] at (s.side 1) {${X_4}$};
\node[fill=white,inner sep=2] at (s.corner 2) {${X_6'}$};
\node[fill=white,inner sep=2] at (s.side 2) {${X_3}$};
\node[fill=white,inner sep=2] at (s.corner 3) {${X_6}$};
\node[fill=white,inner sep=2] at (s.side 3) {${X_4}$};
\node[fill=white,inner sep=2] at (s.corner 4) {${X_6'}$};
\node[fill=white,inner sep=2] at (s.side 4) {${X_3}$};
\end{tikzpicture}

%% file: tikz/X8_24.tex
\begin{tikzpicture}[font=\footnotesize]
\node[name=s, regular polygon, rotate=180, regular polygon sides=8,inner sep=2.169cm] at (0,0) {};
\draw[-,ultra thick,Orange] (s.center) to["4"] (s.side 1);
\draw[-,ultra thick,black] (s.center) to[""] (s.side 2);
\draw[-,ultra thick,Orange] (s.center) to["4"] (s.side 3);
\draw[-,ultra thick,FireBrick] (s.center) to["2"] (s.side 4);
\draw[-,ultra thick,Orange] (s.center) to["4"] (s.side 5);
\draw[-,ultra thick,black] (s.center) to[""] (s.side 6);
\draw[-,ultra thick,Orange] (s.center) to["4"] (s.side 7);
\draw[-,ultra thick,FireBrick] (s.center) to["2"] (s.side 8);
\draw[-,ultra thick,black] (s.side 1) to[swap,""] (s.corner 2);
\draw[-,ultra thick,FireBrick] (s.side 1) to["2"] (s.corner 1);
\draw[-,ultra thick,Orange] (s.side 2) to[swap,"4"] (s.corner 3);
\draw[-,ultra thick,Orange] (s.side 2) to["4"] (s.corner 2);
\draw[-,ultra thick,FireBrick] (s.side 3) to[swap,"2"] (s.corner 4);
\draw[-,ultra thick,black] (s.side 3) to[""] (s.corner 3);
\draw[-,ultra thick,Orange] (s.side 4) to[swap,"4"] (s.corner 5);
\draw[-,ultra thick,Orange] (s.side 4) to["4"] (s.corner 4);
\draw[-,ultra thick,black] (s.side 5) to[swap,""] (s.corner 6);
\draw[-,ultra thick,FireBrick] (s.side 5) to["2"] (s.corner 5);
\draw[-,ultra thick,Orange] (s.side 6) to[swap,"4"] (s.corner 7);
\draw[-,ultra thick,Orange] (s.side 6) to["4"] (s.corner 6);
\draw[-,ultra thick,FireBrick] (s.side 7) to[swap,"2"] (s.corner 8);
\draw[-,ultra thick,black] (s.side 7) to[""] (s.corner 7);
\draw[-,ultra thick,Orange] (s.side 8) to[swap,"4"] (s.corner 1);
\draw[-,ultra thick,Orange] (s.side 8) to["4"] (s.corner 8);
\node[fill=white,circle,inner sep=2] at (s.center) {${X_2}$};
\node[fill=white,inner sep=2] at (s.corner 1) {${X_8}$};
\node[fill=white,inner sep=2] at (s.side 1) {${X_6}$};
\node[fill=white,inner sep=2] at (s.corner 2) {${X_6}_{\!/\p^1}$};
\node[fill=white,inner sep=2] at (s.side 2) {${X_2}_{/\p^1}$};
\node[fill=white,inner sep=2] at (s.corner 3) {${X_6}_{\!/\p^1}$};
\node[fill=white,inner sep=2] at (s.side 3) {${X_6}$};
\node[fill=white,inner sep=2] at (s.corner 4) {${X_8}$};
\node[fill=white,inner sep=2] at (s.side 4) {${X_4}$};
\node[fill=white,inner sep=2] at (s.corner 5) {${X_8}$};
\node[fill=white,inner sep=2] at (s.side 5) {${X_6}$};
\node[fill=white,inner sep=2] at (s.corner 6) {${X_6}_{\!/\p^1}$};
\node[fill=white,inner sep=2] at (s.side 6) {${X_2}_{/\p^1}$};
\node[fill=white,inner sep=2] at (s.corner 7) {${X_6}_{\!/\p^1}$};
\node[fill=white,inner sep=2] at (s.side 7) {${X_6}$};
\node[fill=white,inner sep=2] at (s.corner 8) {${X_8}$};
\node[fill=white,inner sep=2] at (s.side 8) {${X_4}$};
\end{tikzpicture}

%% file: tikz/X8_34.tex
\begin{tikzpicture}[font=\footnotesize]
\node[name=s, regular polygon, rotate=180, regular polygon sides=8,inner sep=2.169cm] at (0,0) {};
\draw[-,ultra thick,Orange] (s.center) to["4"] (s.side 1);
\draw[-,ultra thick,RoyalBlue] (s.center) to["1"] (s.side 2);
\draw[-,ultra thick,Orange] (s.center) to["4"] (s.side 3);
\draw[-,ultra thick,ForestGreen] (s.center) to["3"] (s.side 4);
\draw[-,ultra thick,Orange] (s.center) to["4"] (s.side 5);
\draw[-,ultra thick,RoyalBlue] (s.center) to["1"] (s.side 6);
\draw[-,ultra thick,Orange] (s.center) to["4"] (s.side 7);
\draw[-,ultra thick,ForestGreen] (s.center) to["3"] (s.side 8);
\draw[-,ultra thick,RoyalBlue] (s.side 1) to[swap,"1"] (s.corner 2);
\draw[-,ultra thick,ForestGreen] (s.side 1) to["3"] (s.corner 1);
\draw[-,ultra thick,Orange] (s.side 2) to[swap,"4"] (s.corner 3);
\draw[-,ultra thick,Orange] (s.side 2) to["4"] (s.corner 2);
\draw[-,ultra thick,ForestGreen] (s.side 3) to[swap,"3"] (s.corner 4);
\draw[-,ultra thick,RoyalBlue] (s.side 3) to["1"] (s.corner 3);
\draw[-,ultra thick,Orange] (s.side 4) to[swap,"4"] (s.corner 5);
\draw[-,ultra thick,Orange] (s.side 4) to["4"] (s.corner 4);
\draw[-,ultra thick,RoyalBlue] (s.side 5) to[swap,"1"] (s.corner 6);
\draw[-,ultra thick,ForestGreen] (s.side 5) to["3"] (s.corner 5);
\draw[-,ultra thick,Orange] (s.side 6) to[swap,"4"] (s.corner 7);
\draw[-,ultra thick,Orange] (s.side 6) to["4"] (s.corner 6);
\draw[-,ultra thick,ForestGreen] (s.side 7) to[swap,"3"] (s.corner 8);
\draw[-,ultra thick,RoyalBlue] (s.side 7) to["1"] (s.corner 7);
\draw[-,ultra thick,Orange] (s.side 8) to[swap,"4"] (s.corner 1);
\draw[-,ultra thick,Orange] (s.side 8) to["4"] (s.corner 8);
\node[fill=white,circle,inner sep=2] at (s.center) {${X_1}$};
\node[fill=white,inner sep=2] at (s.corner 1) {${X_8}$};
\node[fill=white,inner sep=2] at (s.side 1) {${X_5}$};
\node[fill=white,inner sep=2] at (s.corner 2) {${X_6}$};
\node[fill=white,inner sep=2] at (s.side 2) {${X_2}$};
\node[fill=white,inner sep=2] at (s.corner 3) {${X_6}$};
\node[fill=white,inner sep=2] at (s.side 3) {${X_5}$};
\node[fill=white,inner sep=2] at (s.corner 4) {${X_8}$};
\node[fill=white,inner sep=2] at (s.side 4) {${X_4}$};
\node[fill=white,inner sep=2] at (s.corner 5) {${X_8}$};
\node[fill=white,inner sep=2] at (s.side 5) {${X_5}$};
\node[fill=white,inner sep=2] at (s.corner 6) {${X_6}$};
\node[fill=white,inner sep=2] at (s.side 6) {${X_2}$};
\node[fill=white,inner sep=2] at (s.corner 7) {${X_6}$};
\node[fill=white,inner sep=2] at (s.side 7) {${X_5}$};
\node[fill=white,inner sep=2] at (s.corner 8) {${X_8}$};
\node[fill=white,inner sep=2] at (s.side 8) {${X_4}$};
\end{tikzpicture}

%% file: tikz/X8_25.tex
\begin{tikzpicture}[font=\footnotesize]
\node[name=s, regular polygon, rotate=180, regular polygon sides=10,inner sep=2.7720000000000002cm] at (0,0) {};
\draw[-,ultra thick,HotPink] (s.center) to["5"] (s.side 1);
\draw[-,ultra thick,black] (s.center) to[""] (s.side 2);
\draw[-,ultra thick,HotPink] (s.center) to["5"] (s.side 3);
\draw[-,ultra thick,FireBrick] (s.center) to["2"] (s.side 4);
\draw[-,ultra thick,FireBrick] (s.center) to["2"] (s.side 5);
\draw[-,ultra thick,HotPink] (s.center) to["5"] (s.side 6);
\draw[-,ultra thick,black] (s.center) to[""] (s.side 7);
\draw[-,ultra thick,HotPink] (s.center) to["5"] (s.side 8);
\draw[-,ultra thick,FireBrick] (s.center) to["2"] (s.side 9);
\draw[-,ultra thick,FireBrick] (s.center) to["2"] (s.side 10);
\draw[-,ultra thick,black] (s.side 1) to[swap,""] (s.corner 2);
\draw[-,ultra thick,FireBrick] (s.side 1) to["2"] (s.corner 1);
\draw[-,ultra thick,HotPink] (s.side 2) to[swap,"5"] (s.corner 3);
\draw[-,ultra thick,HotPink] (s.side 2) to["5"] (s.corner 2);
\draw[-,ultra thick,FireBrick] (s.side 3) to[swap,"2"] (s.corner 4);
\draw[-,ultra thick,black] (s.side 3) to[""] (s.corner 3);
\draw[-,ultra thick,FireBrick] (s.side 4) to[swap,"2"] (s.corner 5);
\draw[-,ultra thick,HotPink] (s.side 4) to["5"] (s.corner 4);
\draw[-,ultra thick,HotPink] (s.side 5) to[swap,"5"] (s.corner 6);
\draw[-,ultra thick,FireBrick] (s.side 5) to["2"] (s.corner 5);
\draw[-,ultra thick,black] (s.side 6) to[swap,""] (s.corner 7);
\draw[-,ultra thick,FireBrick] (s.side 6) to["2"] (s.corner 6);
\draw[-,ultra thick,HotPink] (s.side 7) to[swap,"5"] (s.corner 8);
\draw[-,ultra thick,HotPink] (s.side 7) to["5"] (s.corner 7);
\draw[-,ultra thick,FireBrick] (s.side 8) to[swap,"2"] (s.corner 9);
\draw[-,ultra thick,black] (s.side 8) to[""] (s.corner 8);
\draw[-,ultra thick,FireBrick] (s.side 9) to[swap,"2"] (s.corner 10);
\draw[-,ultra thick,HotPink] (s.side 9) to["5"] (s.corner 9);
\draw[-,ultra thick,HotPink] (s.side 10) to[swap,"5"] (s.corner 1);
\draw[-,ultra thick,FireBrick] (s.side 10) to["2"] (s.corner 10);
\node[fill=white,circle,inner sep=2] at (s.center) {${X_1}$};
\node[fill=white,inner sep=2] at (s.corner 1) {${X_8}$};
\node[fill=white,inner sep=2] at (s.side 1) {${X_6}$};
\node[fill=white,inner sep=2] at (s.corner 2) {${X_6}_{\!/\p^1}$};
\node[fill=white,inner sep=2] at (s.side 2) {${X_1}_{/\p^1}$};
\node[fill=white,inner sep=2] at (s.corner 3) {${X_6}_{\!/\p^1}$};
\node[fill=white,inner sep=2] at (s.side 3) {${X_6}$};
\node[fill=white,inner sep=2] at (s.corner 4) {${X_8}$};
\node[fill=white,inner sep=2] at (s.side 4) {${X_3}$};
\node[fill=white,inner sep=2] at (s.corner 5) {${X_5}$};
\node[fill=white,inner sep=2] at (s.side 5) {${X_3}$};
\node[fill=white,inner sep=2] at (s.corner 6) {${X_8}$};
\node[fill=white,inner sep=2] at (s.side 6) {${X_6}$};
\node[fill=white,inner sep=2] at (s.corner 7) {${X_6}_{\!/\p^1}$};
\node[fill=white,inner sep=2] at (s.side 7) {${X_1}_{/\p^1}$};
\node[fill=white,inner sep=2] at (s.corner 8) {${X_6}_{\!/\p^1}$};
\node[fill=white,inner sep=2] at (s.side 8) {${X_6}$};
\node[fill=white,inner sep=2] at (s.corner 9) {${X_8}$};
\node[fill=white,inner sep=2] at (s.side 9) {${X_3}$};
\node[fill=white,inner sep=2] at (s.corner 10) {${X_5}$};
\node[fill=white,inner sep=2] at (s.side 10) {${X_3}$};
\end{tikzpicture}

%% file: tikz/F0_02.tex
\begin{tikzpicture}[font=\footnotesize]
\node[name=s, regular polygon, rotate=180, regular polygon sides=6,inner sep=1.557cm] at (0,0) {};
\draw[-,ultra thick,FireBrick] (s.center) to["2"] (s.side 1);
\draw[-,ultra thick,black] (s.center) to[""] (s.side 2);
\draw[-,ultra thick,FireBrick] (s.center) to["2"] (s.side 3);
\draw[-,ultra thick,black] (s.center) to[""] (s.side 4);
\draw[-,ultra thick,FireBrick] (s.center) to["2"] (s.side 5);
\draw[-,ultra thick,black] (s.center) to[""] (s.side 6);
\draw[-,ultra thick,black] (s.side 1) to[swap,""] (s.corner 2);
\draw[-,ultra thick,black] (s.side 1) to[""] (s.corner 1);
\draw[-,ultra thick,FireBrick] (s.side 2) to[swap,"2"] (s.corner 3);
\draw[-,ultra thick,FireBrick] (s.side 2) to["2"] (s.corner 2);
\draw[-,ultra thick,black] (s.side 3) to[swap,""] (s.corner 4);
\draw[-,ultra thick,black] (s.side 3) to[""] (s.corner 3);
\draw[-,ultra thick,FireBrick] (s.side 4) to[swap,"2"] (s.corner 5);
\draw[-,ultra thick,FireBrick] (s.side 4) to["2"] (s.corner 4);
\draw[-,ultra thick,black] (s.side 5) to[swap,""] (s.corner 6);
\draw[-,ultra thick,black] (s.side 5) to[""] (s.corner 5);
\draw[-,ultra thick,FireBrick] (s.side 6) to[swap,"2"] (s.corner 1);
\draw[-,ultra thick,FireBrick] (s.side 6) to["2"] (s.corner 6);
\node[fill=white,circle,inner sep=2] at (s.center) {${X_6}$};
\node[fill=white,inner sep=2] at (s.corner 1) {${\F_0}_{\!/\p^1}$};
\node[fill=white,inner sep=2] at (s.side 1) {${\F_0}$};
\node[fill=white,inner sep=2] at (s.corner 2) {${\F_0}_{\!/\p^1}$};
\node[fill=white,inner sep=2] at (s.side 2) {${X_6}_{/\p^1}$};
\node[fill=white,inner sep=2] at (s.corner 3) {${\F_0}_{\!/\p^1}$};
\node[fill=white,inner sep=2] at (s.side 3) {${\F_0}$};
\node[fill=white,inner sep=2] at (s.corner 4) {${\F_0}_{\!/\p^1}$};
\node[fill=white,inner sep=2] at (s.side 4) {${X_6}_{/\p^1}$};
\node[fill=white,inner sep=2] at (s.corner 5) {${\F_0}_{\!/\p^1}$};
\node[fill=white,inner sep=2] at (s.side 5) {${\F_0}$};
\node[fill=white,inner sep=2] at (s.corner 6) {${\F_0}_{\!/\p^1}$};
\node[fill=white,inner sep=2] at (s.side 6) {${X_6}_{/\p^1}$};
\end{tikzpicture}

%% file: tikz/F0_04.tex
\begin{tikzpicture}[font=\footnotesize]
\node[name=s, regular polygon, rotate=180, regular polygon sides=8,inner sep=2.169cm] at (0,0) {};
\draw[-,ultra thick,Orange] (s.center) to["4"] (s.side 1);
\draw[-,ultra thick,black] (s.center) to[""] (s.side 2);
\draw[-,ultra thick,Orange] (s.center) to["4"] (s.side 3);
\draw[-,ultra thick,black] (s.center) to[""] (s.side 4);
\draw[-,ultra thick,Orange] (s.center) to["4"] (s.side 5);
\draw[-,ultra thick,black] (s.center) to[""] (s.side 6);
\draw[-,ultra thick,Orange] (s.center) to["4"] (s.side 7);
\draw[-,ultra thick,black] (s.center) to[""] (s.side 8);
\draw[-,ultra thick,black] (s.side 1) to[swap,""] (s.corner 2);
\draw[-,ultra thick,black] (s.side 1) to[""] (s.corner 1);
\draw[-,ultra thick,Orange] (s.side 2) to[swap,"4"] (s.corner 3);
\draw[-,ultra thick,Orange] (s.side 2) to["4"] (s.corner 2);
\draw[-,ultra thick,black] (s.side 3) to[swap,""] (s.corner 4);
\draw[-,ultra thick,black] (s.side 3) to[""] (s.corner 3);
\draw[-,ultra thick,Orange] (s.side 4) to[swap,"4"] (s.corner 5);
\draw[-,ultra thick,Orange] (s.side 4) to["4"] (s.corner 4);
\draw[-,ultra thick,black] (s.side 5) to[swap,""] (s.corner 6);
\draw[-,ultra thick,black] (s.side 5) to[""] (s.corner 5);
\draw[-,ultra thick,Orange] (s.side 6) to[swap,"4"] (s.corner 7);
\draw[-,ultra thick,Orange] (s.side 6) to["4"] (s.corner 6);
\draw[-,ultra thick,black] (s.side 7) to[swap,""] (s.corner 8);
\draw[-,ultra thick,black] (s.side 7) to[""] (s.corner 7);
\draw[-,ultra thick,Orange] (s.side 8) to[swap,"4"] (s.corner 1);
\draw[-,ultra thick,Orange] (s.side 8) to["4"] (s.corner 8);
\node[fill=white,circle,inner sep=2] at (s.center) {${X_4}$};
\node[fill=white,inner sep=2] at (s.corner 1) {${\F_0}_{\!/\p^1}$};
\node[fill=white,inner sep=2] at (s.side 1) {${\F_0}$};
\node[fill=white,inner sep=2] at (s.corner 2) {${\F_0}_{\!/\p^1}$};
\node[fill=white,inner sep=2] at (s.side 2) {${X_4}_{/\p^1}$};
\node[fill=white,inner sep=2] at (s.corner 3) {${\F_0}_{\!/\p^1}$};
\node[fill=white,inner sep=2] at (s.side 3) {${\F_0}$};
\node[fill=white,inner sep=2] at (s.corner 4) {${\F_0}_{\!/\p^1}$};
\node[fill=white,inner sep=2] at (s.side 4) {${X_4}_{/\p^1}$};
\node[fill=white,inner sep=2] at (s.corner 5) {${\F_0}_{\!/\p^1}$};
\node[fill=white,inner sep=2] at (s.side 5) {${\F_0}$};
\node[fill=white,inner sep=2] at (s.corner 6) {${\F_0}_{\!/\p^1}$};
\node[fill=white,inner sep=2] at (s.side 6) {${X_4}_{/\p^1}$};
\node[fill=white,inner sep=2] at (s.corner 7) {${\F_0}_{\!/\p^1}$};
\node[fill=white,inner sep=2] at (s.side 7) {${\F_0}$};
\node[fill=white,inner sep=2] at (s.corner 8) {${\F_0}_{\!/\p^1}$};
\node[fill=white,inner sep=2] at (s.side 8) {${X_4}_{/\p^1}$};
\end{tikzpicture}

%% file: tikz/F0_06.tex
\begin{tikzpicture}[font=\footnotesize]
\node[name=s, regular polygon, rotate=180, regular polygon sides=8,inner sep=2.169cm] at (0,0) {};
\draw[-,ultra thick,Turquoise] (s.center) to["6"] (s.side 1);
\draw[-,ultra thick,black] (s.center) to[""] (s.side 2);
\draw[-,ultra thick,Turquoise] (s.center) to["6"] (s.side 3);
\draw[-,ultra thick,black] (s.center) to[""] (s.side 4);
\draw[-,ultra thick,Turquoise] (s.center) to["6"] (s.side 5);
\draw[-,ultra thick,black] (s.center) to[""] (s.side 6);
\draw[-,ultra thick,Turquoise] (s.center) to["6"] (s.side 7);
\draw[-,ultra thick,black] (s.center) to[""] (s.side 8);
\draw[-,ultra thick,black] (s.side 1) to[swap,""] (s.corner 2);
\draw[-,ultra thick,black] (s.side 1) to[""] (s.corner 1);
\draw[-,ultra thick,Turquoise] (s.side 2) to[swap,"6"] (s.corner 3);
\draw[-,ultra thick,Turquoise] (s.side 2) to["6"] (s.corner 2);
\draw[-,ultra thick,black] (s.side 3) to[swap,""] (s.corner 4);
\draw[-,ultra thick,black] (s.side 3) to[""] (s.corner 3);
\draw[-,ultra thick,Turquoise] (s.side 4) to[swap,"6"] (s.corner 5);
\draw[-,ultra thick,Turquoise] (s.side 4) to["6"] (s.corner 4);
\draw[-,ultra thick,black] (s.side 5) to[swap,""] (s.corner 6);
\draw[-,ultra thick,black] (s.side 5) to[""] (s.corner 5);
\draw[-,ultra thick,Turquoise] (s.side 6) to[swap,"6"] (s.corner 7);
\draw[-,ultra thick,Turquoise] (s.side 6) to["6"] (s.corner 6);
\draw[-,ultra thick,black] (s.side 7) to[swap,""] (s.corner 8);
\draw[-,ultra thick,black] (s.side 7) to[""] (s.corner 7);
\draw[-,ultra thick,Turquoise] (s.side 8) to[swap,"6"] (s.corner 1);
\draw[-,ultra thick,Turquoise] (s.side 8) to["6"] (s.corner 8);
\node[fill=white,circle,inner sep=2] at (s.center) {${X_2}$};
\node[fill=white,inner sep=2] at (s.corner 1) {${\F_0}_{\!/\p^1}$};
\node[fill=white,inner sep=2] at (s.side 1) {${\F_0}$};
\node[fill=white,inner sep=2] at (s.corner 2) {${\F_0}_{\!/\p^1}$};
\node[fill=white,inner sep=2] at (s.side 2) {${X_2}_{/\p^1}$};
\node[fill=white,inner sep=2] at (s.corner 3) {${\F_0}_{\!/\p^1}$};
\node[fill=white,inner sep=2] at (s.side 3) {${\F_0}$};
\node[fill=white,inner sep=2] at (s.corner 4) {${\F_0}_{\!/\p^1}$};
\node[fill=white,inner sep=2] at (s.side 4) {${X_2}_{/\p^1}$};
\node[fill=white,inner sep=2] at (s.corner 5) {${\F_0}_{\!/\p^1}$};
\node[fill=white,inner sep=2] at (s.side 5) {${\F_0}$};
\node[fill=white,inner sep=2] at (s.corner 6) {${\F_0}_{\!/\p^1}$};
\node[fill=white,inner sep=2] at (s.side 6) {${X_2}_{/\p^1}$};
\node[fill=white,inner sep=2] at (s.corner 7) {${\F_0}_{\!/\p^1}$};
\node[fill=white,inner sep=2] at (s.side 7) {${\F_0}$};
\node[fill=white,inner sep=2] at (s.corner 8) {${\F_0}_{\!/\p^1}$};
\node[fill=white,inner sep=2] at (s.side 8) {${X_2}_{/\p^1}$};
\end{tikzpicture}